\documentclass[a4,11pt]{amsart}

\usepackage[T1]{fontenc}
\usepackage[utf8]{inputenc}
\usepackage{lmodern}
\usepackage[english]{babel}
\usepackage[autostyle]{csquotes}
\usepackage[usenames, dvipsnames]{color}

\numberwithin{equation}{section}

\usepackage[pdfusetitle,linktocpage,colorlinks=false]{hyperref}
\usepackage{doi}
\usepackage{filecontents}
\usepackage{enumitem}

\usepackage{amsmath,amssymb,amsthm,amscd}
\usepackage{pictexwd,dcpic}
\newtheorem{thm}{Theorem}[section]
\newtheorem*{mainthm}{Theorem}
\newtheorem{prop}[thm]{Proposition}
\newtheorem{lem}[thm]{Lemma}
\newtheorem{cor}[thm]{Corollary}

\theoremstyle{definition}
\newtheorem{defn}[thm]{Definition}
\theoremstyle{remark}
\newtheorem{rem}[thm]{Remark}
\newtheorem{assump}[thm]{Assumption}

\newcommand{\X}{\mathsf{X}}
\newcommand{\Y}{\mathsf{Y}}
\newcommand{\V}{\mathsf{V}}
\newcommand{\W}{\mathsf{W}}
\newcommand{\Zm}{\mathsf{Z}}
\newcommand{\Um}{\mathsf{U}}
\newcommand{\cH}{\mathcal{H}}

\DeclareSymbolFont{symbolsC}{U}{pxsyc}{m}{n}
\DeclareMathSymbol{\coloneqq}{\mathrel}{symbolsC}{"42}
\newcommand{\cC}{\mathcal C}
\newcommand{\Bis}{\operatorname{Bis}}
\newcommand{\Bisrc}{\operatorname{Bis}_{\mathrm{rc}}}
\newcommand{\SPIso}{\operatorname{SPIso}}
\newcommand{\Bop}{\mathcal B}
\newcommand{\Span}{\operatorname{span}}
\newcommand{\supp}{\operatorname{supp}}
\newcommand{\op}{\mathrm{op}}

\begin{document}
\title[]{Morita equivalence for $L^p$-operator algebras associated with \'{e}tale groupoids}
\author[]{Yeong Chyuan Chung}
\address{School of Mathematics, Jilin University, Changchun 130012, Jilin, P.R. China}
\email{chungyc@jlu.edu.cn}
\author[]{Alonso Delfín}
\address{Department of Mathematics \& Statistics \\ 500 College Avenue \\ Swarthmore, PA 19081}
\email{adelfin1@swarthmore.edu}
\author[]{Zhen Wang}
\address{School of Mathematics, Hangzhou Normal University, Hangzhou 311121, P.R. China}
\email{wangzhen@hznu.edu.cn}
\date{\today}
\maketitle

\begin{abstract}
We prove that equivalent locally compact, locally Hausdorff, \'{e}tale groupoids with paracompact unit spaces have Morita equivalent reduced and full $L^p$-operator algebras for every $p\in[1,\infty]$. The reduced theorem requires pairing-valued approximate identities, reflecting the failure of unconditionality of the reduced norm for $1<p<\infty$, while the full theorem rests on a full-clopen reduction theorem proved by dilating spatial representations. For $1\leq p<\infty$, we also compare the concrete $p$-linking algebra of the reduced Morita equivalence with the reduced $L^p$-operator algebra of the linking groupoid: the canonical comparison is contractive, injective, and has dense range in general, and is an isometric isomorphism for $p=2$. We further study Morita cycles arising from proper \'{e}tale groupoid correspondences, including the reduced case under suitable extension hypotheses. Applications include an $L^p$ version of Green's symmetric imprimitivity theorem, results for coarse groupoids and inverse semigroups, and invariance of the corresponding $K$-theory.
\end{abstract}

\tableofcontents

\section{Introduction}

A topological groupoid is called \'{e}tale if its source and range maps are local homeomorphisms. \'{E}tale groupoids arise naturally in many areas of mathematics (including foliation theory, dynamical systems, noncommutative geometry, and inverse semigroup theory) and their operator algebras have been the subject of intensive study over the past few decades.

For an \'{e}tale groupoid, Renault's seminal work \cite{Ren1,Ren2} introduced the full and reduced groupoid $C^*$-algebras, and established that many geometric and dynamical properties of the groupoid are encoded in the structure of the $C^*$-algebra. A fundamental result in this theory, commonly known as Renault's equivalence theorem, states that if two \'{e}tale groupoids are equivalent, then their full $C^*$-algebras are strongly Morita equivalent \cite[Theorem 2.8]{MRW}. 
The case of the reduced $C^*$-algebra was proved by making use of the linking groupoid in \cite{SW}, though the authors mentioned that the result had been known to experts earlier (cf. \cite{Ren3,Tu}).
This theorem has Green's imprimitivity theorem for crossed products by group actions as a special case, and was extended to groupoid crossed products in \cite{MW}.

In recent years, there has been growing interest in the $L^p$ analogues of groupoid operator algebras.
Gardella and Lupini \cite{GL} systematically developed the theory of representations of \'{e}tale groupoids on $L^p$-spaces, and introduced the reduced $L^p$-operator algebra of such groupoids. When $p=2$, this is precisely the reduced groupoid $C^*$-algebra. Rigidity results for these $L^p$-operator algebras have been obtained recently \cite{CGT}. Other papers on groupoid $L^p$-operator algebras include \cite{AOP,HO}, and more general Banach algebras associated with \'{e}tale groupoids have also been studied \cite{BKM1,BKM2}. 

A natural question arises: does the Morita equivalence theorem in the $C^*$-algebra setting extend to the $L^p$ setting?
That is, if $G$ and $H$ are equivalent \'{e}tale groupoids, are their full and reduced $L^p$-operator algebras Morita equivalent as Banach algebras in the sense of \cite{Par09}?
In this paper, we prove the following result; see Theorem~\ref{thm:main} and Corollary~\ref{cor:full-morita}:

\begin{mainthm} 
If two locally compact, locally Hausdorff, \'{e}tale groupoids with paracompact unit spaces are equivalent, then, for every $p\in[1,\infty]$,
their reduced $L^p$-operator algebras are Morita equivalent, and their
full $L^p$-operator algebras are Morita equivalent.
\end{mainthm}
In the Hausdorff setting, the cases $p=1$ and $p=\infty$ are covered by \cite[Theorem 5.6]{Par09I} since the reduced norms in these cases are unconditional, but our approach is different.
For $p\in(1,\infty)$, the reduced norm is not unconditional, so \cite[Theorem 5.6]{Par09I} does not apply to these cases.

Although the statement formally parallels Renault's equivalence theorem, its proof is not a routine adaptation of the $C^*$-case. For $p\neq2$, the Hilbert-module structure and the adjoint and positivity arguments underlying $C^*$-imprimitivity are no longer available; moreover, for $1<p<\infty$ the reduced $L^p$-norm is not an unconditional completion norm. On the reduced side, this requires a direct construction of pairing-valued approximate identities in order to establish fullness and nondegeneracy. On the full side, a different difficulty arises: we prove a full-clopen reduction theorem by dilating spatial $L^p$-representations. Thus the reduced and full equivalence theorems require genuinely different $L^p$-analytic replacements for the standard $C^*$-algebraic arguments.

Rieffel's theory of strong Morita equivalence \cite{Rie74,Rie74a} has had a profound influence on the theory of $C^*$-algebras. Beyond the realm of $C^*$-algebras, one has the analogous equivalence relation for Banach algebras studied by Paravicini \cite{Par09}. Just as in the $C^*$-algebra case, Morita equivalence in this sense preserves $K$-theory and, for Banach algebras with bounded approximate identities, induces an isomorphism of closed ideal lattices \cite[Theorem 2.9]{CD2}. Only recently has it been used to study $L^p$-operator algebras \cite{Chung4}. The Morita equivalence results above provide another application of this Banach-algebraic framework to $L^p$-operator algebras.

The reduced and full results require different analytic mechanisms. For the reduced algebras, we follow the linking groupoid approach of \cite{SW}, but must establish fullness and nondegeneracy for the completed off-diagonal Banach pair in a norm that is not unconditional. The key analytic ingredient is a pairing-valued approximate identity, constructed by adapting techniques from \cite{MW}.

For the full algebras, the obstruction is different: the full $L^p$-norm is representation-universal rather than unconditional. We prove that extension by zero from a full clopen reduction preserves the intrinsic full $L^p$-operator algebra norm. For $1<p<\infty$, this is obtained by dilating spatial representations using the Borel extension machinery of \cite{BKM1}; the endpoint cases follow from the corresponding $I$-norm formulas.

We also compare the concrete $p$-linking algebra of the reduced Morita equivalence with the reduced $L^p$-operator algebra of the linking groupoid (Theorem~\ref{thm:LpLink}). This comparison highlights another distinction from the $C^*$-case: for general $p$ the canonical map from the reduced algebra of the linking groupoid to the concrete $p$-linking algebra is contractive, injective, and has dense range, whereas for $p=2$ it is an isometric isomorphism; under an additional factorization hypothesis it is a Banach algebra isomorphism for the fixed value of $p$. We also study Morita cycles associated with \'{e}tale groupoid correspondences. Proper correspondences give Morita cycles for the full algebras. On the reduced side, we construct the canonical Banach $F^p_{\mathrm{red}}(H)$-pair, characterize compactness whenever the left action extends to the reduced completion, and give sufficient conditions for such an extension; the extension always exists for $p=1$ and $p=\infty$, and for $1<p<\infty$ it exists when the left $G$-action is proper (Theorem~\ref{thm:corrected-proper-correspondence}).

Taken together, the paper has three main contributions. First, it establishes Morita equivalence for both the reduced and full $L^p$-operator algebras of equivalent \'{e}tale groupoids, using different analytic methods adapted to the two completions. Second, it develops structural results for concrete $p$-linking algebras and for Banach pairs and Morita cycles arising from groupoid correspondences, including phenomena that differ from the Hilbert-module case. Third, it applies these results to symmetric imprimitivity, coarse groupoids, inverse semigroups, and $K$-theory.

Several applications are obtained in Section~\ref{sect_app}.  These include an
$L^p$ version of Green's symmetric imprimitivity theorem, and Morita
equivalences associated with coarse groupoids and with groupoids of
inverse semigroups.  Since Banach algebra Morita equivalence
preserves $K$-theory, our main theorem also gives
\[
   K_*\bigl(F^p_{\mathrm{red}}(G)\bigr)
   \cong
   K_*\bigl(F^p_{\mathrm{red}}(H)\bigr)
\]
for equivalent groupoids.  The argument applies more generally to
the algebras $F^P_{\mathrm{red}}(G)$ for nonempty
$P\subseteq[1,\infty]$, and the resulting $K$-theory isomorphisms are
compatible with the canonical maps associated with inclusions
$P_1\subseteq P_2$ (Proposition~\ref{prop:commKT}).  Combining the $L^p$ imprimitivity theorem with the coincidence of the
full and reduced crossed products for amenable actions, we also obtain
that, for a free and proper action of a countable discrete group $G$ on a second-countable, locally compact Hausdorff space $X$, the canonical homomorphism
\[
F^p(G,C_0(X))\longrightarrow F^p_{\mathrm{red}}(G,C_0(X))
\]
is an isometric isomorphism, and this algebra is Morita
equivalent to $C_0(X/G)$.  In particular, its $K$-theory is independent
of $p\in[1,\infty)$.

\medskip

\subsection*{Organization of the paper.}

Section~\ref{sect_Prelim} contains the required preliminaries.
Sections~\ref{sect_LinkGp} and~\ref{sect_AI} establish the reduced
Morita equivalence, and Section~\ref{sec:full-Lp} proves the
corresponding result for the full $L^p$-operator algebras. Section~\ref{sect_likVSFp} treats concrete $L^p$-modules and the associated linking algebra comparison.  Groupoid correspondences and applications are
considered in Sections~\ref{sect_corresp} and~\ref{sect_app}, respectively.

\section{Preliminaries}\label{sect_Prelim}

Throughout this paper, groupoids will be assumed to be locally compact, locally Hausdorff, and \'{e}tale with locally compact Hausdorff unit space.

We denote the source and range maps on a groupoid $G$ by $s$ and $r$, respectively.

For each unit $u\in G^{(0)}$, we have the range fiber $G^u=r^{-1}(u)$ and source fiber $G_u=s^{-1}(u)$.
These fibers are discrete when $G$ is \'{e}tale.

If $G$ is \'{e}tale, then it has a left Haar system $\{\lambda^u\}_{u\in G^{(0)}}$, where $\lambda^u$ is counting measure on $G^u$ (cf. \cite[Proposition 2.2.5]{Pat}).
If $H$ is another \'{e}tale groupoid, we shall denote its left Haar system consisting of the respective counting measures by $\{\beta^v\}_{v\in H^{(0)}}$.

We refer the reader to \cite{Ren1,ADR,Pat,MW,Will,CRM} for additional background on topological groupoids and their operator algebras.

\subsection{Banach algebras associated with \'{e}tale groupoids}

Given a locally compact, locally Hausdorff groupoid $G$, we write 
\[
\mathcal{C}(G) \coloneqq \mathrm{span}\{ f\in C_c(U):U\subset G\;\text{open Hausdorff} \}.
\]
Of course, $\mathcal{C}(G)$ is simply $C_c(G)$ when $G$ is Hausdorff.
For an open Hausdorff subset $U\subset G$, we regard functions in $C_c(U)$ as functions on $G$ by extending them to be zero outside $U$.

The product in $\mathcal{C}(G)$ is defined by convolution:
\[
(f*g)(\gamma)=\sum_{r(\eta)=r(\gamma)}f(\eta)g(\eta^{-1}\gamma).
\]
The involution on $\mathcal{C}(G)$ is defined by $f^*(\gamma)=\overline{f(\gamma^{-1})}$.
These operations make $\mathcal{C}(G)$ a $*$-algebra.

We have the following three submultiplicative norms on $\mathcal{C}(G)$ associated with the source and range maps on $G$:
\begin{align*}
	\Vert f\Vert_{I,s} &= \sup_{u\in G^{(0)}}\sum_{s(\gamma)=u}|f(\gamma)|, \\
	\Vert f\Vert_{I,r} &= \sup_{u\in G^{(0)}}\sum_{r(\gamma)=u}|f(\gamma)|, \\
	\Vert f\Vert_I &= \max\{ \Vert f\Vert_{I,s},\Vert f\Vert_{I,r} \}.
\end{align*}
The norm $\Vert-\Vert_I$ is called the $I$-norm.
Note that $\Vert f^*\Vert_{I,s}=\Vert f\Vert_{I,r}$ so $\Vert f^*\Vert_I=\Vert f\Vert_I$ for all $f\in\mathcal{C}(G)$.

We denote the completions of $\mathcal{C}(G)$ with respect to these norms by $F_{I,s}(G)$, $F_{I,r}(G)$, and $F_I(G)$, respectively.

If $u\in G^{(0)}$ and $\delta_u$ is the point mass, then we denote by $\mathrm{Ind}\;\delta_u$ the representation of $\mathcal{C}(G)$ on $\ell^p(G_u)$ given by the convolution formula, i.e., 
for $f \in \mathcal{C}(G)$ and $\xi \in \ell^p(G_u)$ we get $(\mathrm{Ind}\;\delta_u)(f)\xi \in \ell^p(G_u)$
by letting for each $\gamma \in G_u$,
\[ (\mathrm{Ind}\;\delta_u)(f)\xi(\gamma)=\sum_{r(\eta)=r(\gamma)}f(\eta)\xi(\eta^{-1}\gamma).\]
The reduced norm on $\mathcal{C}(G)$ is defined to be 
\[ \Vert f\Vert_{p,red}=\sup_{u\in G^{(0)}}\Vert(\mathrm{Ind}\;\delta_u)(f)\Vert_{B(\ell^p(G_u))}. \]
For $p\in[1,\infty]$, the reduced $L^p$-operator algebra of $G$, denoted by $F^p_{red}(G)$, is defined to be the completion of $\mathcal{C}(G)$ with respect to the reduced norm $\Vert\cdot\Vert_{p,red}$.

In fact, as we shall explain below, $\Vert\cdot\Vert_{1,red}=\Vert\cdot\Vert_{I,s}$ and $\Vert\cdot\Vert_{\infty,red}=\Vert\cdot\Vert_{I,r}$ (cf. \cite[Proposition 5.1]{BKM1}).

For $f\in\mathcal{C}(G)$, $u\in G^{(0)}$, and $\gamma\in G_u$, we have
\[ (\mathrm{Ind}\;\delta_u)(f)\delta_u(\gamma) = \sum_{r(\eta)=r(\gamma)}f(\eta)\delta_u(\eta^{-1}\gamma) = f(\gamma). \]
Thus $\Vert (\mathrm{Ind}\;\delta_u)(f) \Vert_{B(\ell^1(G_u))} \geq \sum_{s(\gamma)=u} |f(\gamma)|$, and $\Vert f\Vert_{1,red}\geq\Vert f\Vert_{I,s}$.
On the other hand, for any $\xi\in\ell^1(G_u)$, we have
\begin{align*} 
\Vert (\mathrm{Ind}\;\delta_u)(f)\xi\Vert_1 &\leq \sum_{s(\gamma)=u}\sum_{r(\eta)=r(\gamma)} |f(\eta)\xi(\eta^{-1}\gamma)| = \sum_{s(\theta)=u}\sum_{s(\gamma)=s(\theta)} |f(\gamma\theta^{-1})\xi(\theta)| \\
&= \sum_{s(\theta)=u}\sum_{s(\gamma')=r(\theta)} |f(\gamma')| |\xi(\theta)| \leq \Vert f\Vert_{I,s}\Vert \xi\Vert_1,
\end{align*}
so we have \[ \Vert f\Vert_{1,red}=\Vert f\Vert_{I,s}. \]

For $\phi\in\ell^\infty(G_u)$, we have
\begin{align*}
\langle (\mathrm{Ind}\;\delta_u)(f^*)\xi,\phi \rangle &= \sum_{s(\gamma)=u} \sum_{r(\eta)=r(\gamma)}f^*(\eta)\xi(\eta^{-1}\gamma)\overline{\phi(\gamma)} \\
&= \sum_{s(\gamma)=u} \sum_{s(\eta)=r(\gamma)}\xi(\gamma)\overline{f(\eta^{-1})\phi(\eta\gamma)} \\
&= \sum_{s(\gamma)=u} \sum_{r(\eta)=r(\gamma)}\xi(\gamma)\overline{f(\eta)\phi(\eta^{-1}\gamma)} \\
&= \langle \xi,(\mathrm{Ind}\;\delta_u)(f)\phi \rangle.
\end{align*}

Hence, 
\[ \Vert (\mathrm{Ind}\;\delta_u)(f) \Vert_{B(\ell^\infty(G_u))}=\Vert (\mathrm{Ind}\;\delta_u)(f^*) \Vert_{B(\ell^1(G_u))}, \]
and \[ \Vert f\Vert_{\infty,red}=\Vert f^*\Vert_{1,red}=\Vert f^*\Vert_{I,s}=\Vert f\Vert_{I,r}. \]

By the Riesz--Thorin interpolation theorem, we then have
\[ \Vert f\Vert_{p,red} \leq \Vert f\Vert_{I,s}^{1/p} \Vert f\Vert_{I,r}^{1/q} \leq \Vert f\Vert _I \]
for $p,q\in(1,\infty)$ satisfying $\frac{1}{p}+\frac{1}{q}=1$.
(cf. \cite[Proposition 5.1]{BKM1})

For $p\in[1,\infty]$, the full $L^p$-operator algebra of $G$, denoted by $F^p(G)$, is defined to be the completion of $\mathcal{C}(G)$ with respect to the norm 
\[ \Vert f\Vert_{L^p}=\sup_{\psi\in\mathcal{R}}\Vert\psi(f)\Vert,\] 
where $\mathcal{R}$ consists of all contractive representations of $(\mathcal{C}(G),\Vert\cdot\Vert_I)$ on $L^p$-spaces
(cf. \cite[Definition 5.12 and Theorem 5.13]{BKM1}).
By \cite[Theorem 5.5]{BKM1}, we have 
\[ \Vert f\Vert_{L^p}\leq\Vert f\Vert_{I,s}^{1/p}\Vert f\Vert_{I,r}^{1/q}\leq\Vert f\Vert _I \] for all $p\in[1,\infty]$, where $q$ is the Hölder conjugate of $p$ with the usual endpoint conventions.
Since $\Vert f\Vert_{p,red}\leq\Vert f\Vert_{L^p}$, it follows that \[ \Vert f\Vert_{L^1}=\Vert f\Vert_{I,s}, \quad \Vert f\Vert_{L^\infty}=\Vert f\Vert_{I,r}. \]

For $f \in \mathcal{C}(G)$ and any nonempty $P\subseteq[1,\infty]$, define 
\[
 \Vert f\Vert_{P}=\sup_{p\in P}\Vert f\Vert_{L^p} \text{ \  and \ } \Vert f\Vert_{P,red}=\sup_{p\in P}\Vert f\Vert_{p,red}.
\]
In this notation, we have 
\[
\Vert f\Vert_{\{1,\infty\}}=\Vert f\Vert_{\{1,\infty\},red}=\Vert f\Vert_I.
\]
We then define 
\[
F^P(G)=\overline{\mathcal{C}(G)}^{\Vert\cdot\Vert_{P}}  \text{ \  and \ } F^P_{red}(G)=\overline{\mathcal{C}(G)}^{\Vert\cdot\Vert_{P,red}}
\] 
(cf. \cite[Definition 3.14]{BKM2}).

\subsection{Groupoid actions and equivalences}

\begin{defn}
If $G$ is a locally compact, locally Hausdorff groupoid, then we say that a locally compact, locally Hausdorff space $Z$ is a left $G$-space if there is a continuous, open map $r_Z:Z\rightarrow G^{(0)}$ and a continuous map $(\gamma,z)\mapsto\gamma\cdot z$ from $G*Z=\{(\gamma,z)\in G\times Z:s_G(\gamma)=r_Z(z)\}$ to $Z$ such that $r_Z(z)\cdot z=z$ for all $z\in Z$ and whenever  $(\gamma',\gamma)\in G^{(2)}$ and $(\gamma, z) \in G * Z$, then $(\gamma', \gamma \cdot z) \in G*Z$ and $(\gamma'\gamma)\cdot z=\gamma'\cdot(\gamma\cdot z)$.

The action is free if $\gamma\cdot z=z$ implies $\gamma=r_Z(z)$;
the action is proper if the map $\Theta:G*Z\rightarrow Z\times Z$ given by $\Theta(\gamma,z)=(\gamma\cdot z,z)$ is a proper map of $G*Z$ into $Z\times Z$, i.e., $\Theta$ is a closed map such that the inverse image of each compact set is compact.

Right actions are defined similarly except that the structure map is denoted by $s_Z$ instead of $r_Z$.
\end{defn}

\begin{defn}
Let $G$ and $H$ be locally compact, locally Hausdorff groupoids. A $(G,H)$-equivalence is a locally compact, locally Hausdorff space $Z$ such that
\begin{enumerate}
\item $Z$ is a free and proper left $G$-space;
\item $Z$ is a free and proper right $H$-space;
\item the actions of $G$ and $H$ on $Z$ commute;
\item $r_Z$ induces a homeomorphism of $Z/H$ onto $G^{(0)}$;
\item $s_Z$ induces a homeomorphism of $G\backslash Z$ onto $H^{(0)}$.
\end{enumerate}
If $G$ and $H$ are Hausdorff, then $Z$ is also taken to be Hausdorff.
\end{defn}

If $Z$ is a $(G,H)$-equivalence, then there is a continuous map $(y,z)\mapsto{}_G[y,z]$ from $Z*_sZ=\{ (y,z)\in Z\times Z:s_Z(y)=s_Z(z) \}$ to $G$ uniquely determined by ${}_G[y,z]\cdot z=y$ for all $(y,z)\in Z*_sZ$.
There is also a continuous map $(y,z)\mapsto[y,z]_H$ satisfying $y\cdot[y,z]_H=z$ for all $(y,z)\in Z*_rZ=\{ (y,z)\in Z\times Z:r_Z(y)=r_Z(z) \}$.

\begin{defn}
Given a $(G,H)$-equivalence $Z$, we define its opposite space to be a homeomorphic copy $Z^{op} \coloneqq \{\bar{z}:z\in Z\}$ of $Z$ with the structure of a $(H,G)$-equivalence determined by
\[ r(\bar{z})=s(z), s(\bar{z})=r(z), \eta\cdot\bar{z}=\overline{z\cdot\eta^{-1}}, \bar{z}\cdot\gamma=\overline{\gamma^{-1}\cdot z}. \]
\end{defn}

\subsection{Morita equivalence of Banach algebras}

\begin{defn}
Let $B$ be a Banach algebra. A \textit{right (resp. left) Banach $B$-module} is a Banach space $\X$ with the structure of a right (resp. left) $B$-module such that $\|x \cdot b\|_\X\leq\|x\|_\X \|b\|_B$ (resp. $\|b \cdot x\|_\X\leq\|b\|_B\|x\|_\X$) for all $x \in \X$ and $b\in B$.

We say that $\X$ is \textit{nondegenerate} if the linear span of $\X B$ (resp. $B\X$) is dense in $\X$. 
\end{defn}

\begin{defn}
Let $B$ be a Banach algebra. A \textit{Banach $B$-pair} is a pair $(\X,\Y)$ such that
\begin{enumerate}
\item $\X$ is a left Banach $B$-module,
\item $\Y$ is a right Banach $B$-module,
\item there is a $\mathbb{C}$-bilinear map $\langle -, - \rangle_B \colon \X \times \Y \to B$ such that 
\begin{itemize}
\item $\langle b\cdot x,y\rangle_B=b\langle x,y \rangle_B$,
\item $\langle x,y\cdot b \rangle_B=\langle x,y \rangle_B b$, and
\item $\|\langle x,y \rangle_B\|_B\leq\|x\|_\X \|y\|_\Y$
\end{itemize}
for all $x \in \X$, $y \in \Y$, and $b\in B$.
\end{enumerate}
We say that $(\X,\Y)$ is \textit{nondegenerate} if both $\X$ and $\Y$ are nondegenerate. We say that $(\X,\Y)$ is \textit{full} if the linear span of $\langle \X, \Y \rangle$ is dense in $B$.
\end{defn}

\begin{defn}\label{Def:PreBanachPair}
Let $B$ be a Banach algebra and let $B_0 \subseteq B$ be a dense subalgebra. 
A \textit{pre-Banach $B_0$-pair} is a pair $(\X_0, \Y_0)$ such that  
\begin{enumerate}
\item $\X_0$ is a normed left $B_0$-module satisfying $\|b \cdot x\|_{\X_0}\leq\|b\|_B\|x\|_{\X_0}$, 
\item $\Y_0$ is a normed right $B_0$-module satisfying $\|y \cdot b\|_{\Y_0}\leq\|y\|_{\Y_0}\|b\|_B$, 
\item there is a $\mathbb{C}$-bilinear map $\langle -, - \rangle_B \colon \X_0 \times \Y_0 \to B_0$ such that 
\begin{itemize}
\item $\langle b\cdot x,y\rangle_B=b\langle x,y \rangle_B$,
\item $\langle x,y\cdot b \rangle_B=\langle x,y \rangle_B b$, and
\item $\|\langle x,y \rangle_B\|_B\leq\|x\|_{\X_0} \|y\|_{\Y_0}$
\end{itemize}
for all $x \in \X_0$, $y \in \Y_0$, and $b\in B_0$.
\end{enumerate}
\end{defn}

\begin{rem}\label{Rem:CompletionBP}
If $(\X_0,\Y_0)$ is a pre-Banach $B_0$-pair, its completion is a
Banach $B$-pair.  It is nondegenerate if
\[
\operatorname{span}(B_0\X_0)=\X_0,
\qquad
\operatorname{span}(\Y_0B_0)=\Y_0,
\]
and is full if
\[
\operatorname{span}\langle\X_0,\Y_0\rangle_B=B_0.
\]
\end{rem}

\begin{defn}\label{defn_LinOP}
Let $(\X,\Y)$ and $(\V,\W)$ be Banach $B$-pairs. A \textit{bounded linear operator} from $(\X,\Y)$ to $(\V,\W)$ is a pair $T=(T^l,T^r)$ such that 
\begin{enumerate}
\item $T^l \colon \V \to \X$ is a homomorphism of left Banach $B$-modules with 
\[
\| T^l \| \coloneqq \sup_{\| v\|_\V \leq 1} \| T^l(v) \|_\X < \infty, 
\]
\item $T^r \colon \Y \to \W$ is a homomorphism of right Banach $B$-modules with 
\[
\| T^r \| \coloneqq \sup_{\| y\|_\Y \leq 1} \| T^r(y) \|_\W < \infty, 
\]
\item $T^l$ and $T^r$ are formal adjoints to each other, that is 
\[
\langle T^l(v),y \rangle_B =\langle v,T^r(y) \rangle_B
\]
 for all $v \in \V$ and $y \in \Y$.
\end{enumerate}
\end{defn}
The space of all bounded linear operators from $(\X,\Y)$ to $(\V,\W)$ is denoted by $\mathcal{L}_B((\X,\Y) \to (\V,\W))$ and is a Banach space under the norm 
\[
\|T\|=\|(T^l,T^r)\| \coloneqq \max(\|T^l\|,\|T^r\|).
\] 
Further, if $(\Um,\Zm)$ is another Banach $B$-pair and $S\in \mathcal{L}_B((\Um,\Zm) \to (\X,\Y))$, then 
it can be precomposed with $T$ by letting
\[
TS \coloneqq (S^lT^l, T^rS^r) \in \mathcal{L}_B((\Um,\Zm) \to (\V,\W)).
\]
We will write $\mathcal{L}_B(\X,\Y)$ for $\mathcal{L}_B((\X,\Y) \to (\X,\Y))$, which is a Banach algebra under this composition. For $T=(T^l,T^r) \in \mathcal{L}_B((\X,\Y) \to (\V,\W))$ we will often write 
\[
 v \cdot T  \coloneqq  T^l(v) \in \X, \ T \cdot y  \coloneqq T^r(y) \in \W,
\]
for all $y \in \Y$ and all $v \in \V$. With this notation, the formal adjoint identity becomes 
\[
\langle v \cdot T,y \rangle_B=\langle v,T \cdot y \rangle_B.
\]
Observe that this naturally gives a right action of $\mathcal{L}_B(\X,\Y)$ on $\X$ and 
a left action of $\mathcal{L}_B(\X,\Y)$ on $\Y$.

\begin{defn}
Let $(\X,\Y)$ and $(\V,\W)$ be Banach $B$-pairs. For $x \in \X$, $w \in \W$ 
we get a map $\theta_{w,x}^l \colon \V \to \X$ defined as
\[
\theta_{w,x}^l(v)=\langle v,w\rangle_B\cdot x
\]
and a map $ \theta_{w,x}^r \colon \Y \to \W$ given by 
\[
\theta_{w,x}^r (y) = w \cdot  \langle x, y  \rangle_B 
\]
\end{defn}

It is not hard to check that $\theta_{w,x}\coloneqq(\theta_{w,x}^l,\theta_{w,x}^r) \in \mathcal{L}_B((\X,\Y) \to (\V, \W))$, and that $\| \theta_{w,x} \| \leq \|w\|_\W\|x\|_\X$. Further, let $(\Um,\Zm)$ be another Banach $B$-pair. If $T \in \mathcal{L}_B((\V,\W) \to (\Um, \Zm))$, then a routine check gives
\[
T\theta_{w,x} = \theta_{T \cdot w,x} \in \mathcal{L}_B((\X,\Y) \to (\Um, \Zm)).
\]
Similarly, if now $S \in \mathcal{L}_B((\Um, \Zm) \to (\X,\Y))$, then 
\[
\theta_{w,x}S= \theta_{w,x\cdot S} \in \mathcal{L}_B((\Um, \Zm) \to (\V,\W)).
\]
\begin{defn}
Let $(\X,\Y)$ and $(\V,\W)$ be Banach $B$-pairs. We define 
the \textit{compact module linear operators} from $(\X,\Y)$ to $(\V,\W)$ by 
\[
\mathcal{K}_B((\X,\Y) \to (\V, \W)) = \overline{\mathrm{span}}\{ \theta_{w,x} \colon w \in \W, x \in \X\} \subseteq \mathcal{L}_B((\X,\Y) \to (\V, \W))
\]
\end{defn}
As usual, we write $\mathcal{K}_B(\X,\Y)$ for $\mathcal{K}_B((\X,\Y) \to (\X, \Y))$.
The preceding two identities show that $\mathcal{K}_B(\X,\Y)$ is a two-sided closed ideal in $\mathcal{L}_B(\X,\Y)$.

\begin{defn}
Let $A$ and $B$ be Banach algebras. We say $((\X,\Y), \pi_A)$ is a \textit{Banach $(A,B)$-correspondence} (also called a \textit{Banach $A$-$B$-pair} in \cite{Par09}) when $(\X,\Y)$ is a Banach $B$-pair and $\pi_A \colon A\to\mathcal{L}_B(\X,\Y)$ is a contractive homomorphism. 
\end{defn}

Observe that if $((\X,\Y), \pi_A)$ is a Banach $(A,B)$-correspondence, then $\X$ is a Banach $B$-$A$-bimodule and $\Y$ is a Banach $A$-$B$-bimodule. Further, for all $x\in \X$, $y\in \Y$, $a\in A$, we write for short $x \cdot a = x \cdot \pi_A(a)$ and $a \cdot y = \pi_A(a) \cdot y$, so that 
\[
\langle x\cdot a,y\rangle_B=\langle x,a\cdot y\rangle_B.
\]

\begin{defn}
Let $B$ be a Banach algebra, let $B_0\subseteq B$ be a dense subalgebra, and let $(\X_0,\Y_0)$ be a pre-Banach $B_0$-pair.  We write
\[
\mathcal{L}^{\mathrm{pre}}_{B_0}(\X_0,\Y_0)
\]
for the normed algebra of pairs $T=(T^l,T^r)$ such that $T^l\colon\X_0\to\X_0$ is a bounded left $B_0$-module map, $T^r\colon\Y_0\to\Y_0$ is a bounded right $B_0$-module map, and
\[
\langle T^l(x),y\rangle_B=\langle x,T^r(y)\rangle_B
\qquad (x\in\X_0,\ y\in\Y_0).
\]
We equip this space with the norm $\|T\|=\max\{\|T^l\|,\|T^r\|\}$ and with the same composition convention as for completed Banach pairs,
\[
TS=(S^lT^l,T^rS^r).
\]
\end{defn}

\begin{defn}\label{Def:PreCorresp}
Let $A$ and $B$ be Banach algebras and let $A_0\subseteq A$ and $B_0\subseteq B$ be dense subalgebras.  A \textit{pre-Banach $(A_0,B_0)$-correspondence} is a pair $((\X_0,\Y_0),\pi_{A_0})$ where $(\X_0,\Y_0)$ is a pre-Banach $B_0$-pair (see Definition \ref{Def:PreBanachPair}) and
\[
\pi_{A_0}\colon A_0\longrightarrow \mathcal{L}^{\mathrm{pre}}_{B_0}(\X_0,\Y_0)
\]
is a contractive homomorphism.
\end{defn}

\begin{rem}
As in Remark \ref{Rem:CompletionBP}, any pre-Banach $(A_0,B_0)$-correspondence can be completed into a 
Banach $(A,B)$-correspondence. 
\end{rem}

\begin{defn}\label{def:Morita}
Let $A$ and $B$ be Banach algebras. A \textit{Morita equivalence between $A$ and $B$} is a pair $({}_B\X_A,{}_A\Y_B)$ of bimodules equipped with a bilinear pairing $\langle\cdot,\cdot\rangle_B \colon \X\times \Y\to B$ and a bilinear pairing ${}_A\langle\cdot,\cdot\rangle \colon \Y\times \X \to A$ such that
\begin{enumerate}
\item $(\X,\Y)$ with $\langle \cdot, \cdot \rangle_B$ is a Banach $(A,B)$-correspondence that is also full and nondegenerate as a Banach $B$-pair,
\item $(\Y,\X)$ with ${}_A\langle\cdot , \cdot \rangle$ is a Banach $(B,A)$-correspondence that is also full and nondegenerate as a Banach $A$-pair,
\item $\langle x_1,y \rangle_B \cdot x_2=x_1 \cdot {}_A\langle y,x_2 \rangle$ for all $x_1, x_2 \in \X$ and $y \in \Y$,
\item $y_1\cdot \langle x,y_2 \rangle_B={}_A\langle y_1,x \rangle \cdot y_2$ for all $x \in \X$ and $y_1, y_2 \in \Y$.
\end{enumerate}
$A$ and $B$ are said to be \textit{Morita equivalent} if there is a Morita equivalence between $A$ and $B$.
\end{defn}

\begin{rem}
The conditions in this definition of Morita equivalence imply that the Banach algebras $A$ and $B$ are nondegenerate in the sense that the linear span of $AA$ (resp. $BB$) is dense in $A$ (resp. $B$).
\end{rem}

\begin{defn}\label{Def:PreMORITA}
Let $A$ and $B$ be Banach algebras and let $A_0 \subseteq A$ and $B_0 \subseteq B$ be dense subalgebras. 
A \textit{pre-Morita equivalence between $A_0$ and $B_0$} is a normed $B_0$-$A_0$-bimodule  $\X_0$, a normed $A_0$-$B_0$-bimodule $\Y_0$, together with bilinear pairings $\langle\cdot,\cdot\rangle_B \colon \X_0\times \Y_0\to B_0$ and ${}_A\langle\cdot,\cdot\rangle \colon \Y_0\times \X_0 \to A_0$, such that
\begin{enumerate}
\item $(\X_0,\Y_0)$ with $\langle \cdot, \cdot \rangle_B$ is a pre-Banach $(A_0,B_0)$-correspondence,
\item $(\Y_0,\X_0)$ with ${}_A\langle\cdot , \cdot \rangle$ is a pre-Banach $(B_0,A_0)$-correspondence,
\item $\langle x_1,y \rangle_B \cdot x_2=x_1 \cdot {}_A\langle y,x_2 \rangle$ for all $x_1, x_2 \in \X_0$ and $y \in \Y_0$,
\item $y_1\cdot \langle x,y_2 \rangle_B={}_A\langle y_1,x \rangle \cdot y_2$ for all $x \in \X_0$ and $y_1, y_2 \in \Y_0$.
\end{enumerate}
\end{defn}

\begin{rem}\label{Rem-CompletionPM}
Following Remark \ref{Rem:CompletionBP}, we note that the completion of 
a pre-Morita equivalence between $A_0$ and $B_0$ will result in a Morita equivalence between $A$ and $B$
provided that the nondegeneracy and fullness conditions are satisfied. 
\end{rem}

\section{The linking groupoid and the relevant Banach pair}\label{sect_LinkGp}

Let $Z$ be a $(G,H)$-equivalence.  Its \emph{linking groupoid} is the
topological disjoint union
\[
L=G\sqcup Z\sqcup Z^{op}\sqcup H,
\qquad
L^{(0)}=G^{(0)}\sqcup H^{(0)},
\]
with the structure maps described below; see \cite[Section~2]{SW} and
\cite[Section~1.4]{Par09I}.

Define $r,s:L\rightarrow L^{(0)}$ to be the maps inherited from the range and source maps on $G$, $Z$, $Z^{op}$, and $H$.
Let $L^{(2)}=\{(k,l)\in L\times L:s(k)=r(l)\}$, and let $(k,l)\mapsto kl$ be the map from $L^{(2)}$ to $L$ which restricts to multiplication on $G$ and $H$ and to the actions of $G$ and $H$ on $Z$ and $Z^{op}$, also satisfying
\begin{align*}
z\bar{y}={}_G[z,y], &\;\text{for}\; (z,y)\in Z*_sZ, \\
\bar{y}z=[y,z]_H, &\;\text{for}\; (y,z)\in Z*_rZ.
\end{align*}
Define $l\mapsto l^{-1}$ to be the map from $L$ to $L$ which restricts to inversion on $G$ and $H$, also satisfying $z^{-1}=\bar{z}$ and $\bar{z}^{-1}=z$ for $z\in Z$.

\begin{lem}
The space $L$ together with $L^{(0)}$,  $L^{(2)}$, and the maps $r,s$, $l\mapsto l^{-1}$ defined above 
is a locally compact, locally Hausdorff groupoid. Further, if $G$ and $H$ are both \'{e}tale, then so is $L$.
\end{lem}
\begin{proof}
This is a special case of \cite[Lemma 2.1]{SW}. 
The \'{e}tale part follows from \cite[Proposition 9.4.3]{CRM}. 
\end{proof}

Given a $(G,H)$-equivalence $Z$, the range map on $Z$ induces a homeomorphism from the orbit space $Z/H$ to $G^{(0)}$. In particular, $Z/H$ is Hausdorff.
Each orbit is a closed Hausdorff subset of $Z$ \cite[Lemma 1.2]{SW}, and for each orbit $z\cdot H$, every continuous function $f\in C(z\cdot H)$ extends to some $g\in\mathcal{C}(Z)$ \cite[Lemma 1.3]{SW}.
Thus if $u\in G^{(0)}$ and $z\in Z$ with $r(z)=u$, then we can define a Radon measure $\sigma^u_Z$ on $Z$, supported on the orbit $z\cdot H$, determined by 
\[ \sigma^u_Z(\phi)=\sum_{r(\eta)=s(z)}\phi(z\cdot\eta) \]
for $\phi\in\mathcal{C}(Z)$.
Moreover, $\sigma^u_Z$ does not depend on the choice of $z\in r^{-1}(u)$ due to left invariance of $\beta$.

By symmetry, we can also define a Radon measure $\sigma^v_{Z^{op}}$ on $Z^{op}$ supported on $r^{-1}_{Z^{op}}(v)$.
Using these measures, one can obtain a Haar system for $L$ that restricts to the Haar systems for $G$ and $H$.

\begin{lem}\label{lem:haarL} \cite[Lemma 2.2]{SW}
For each $w\in L^{(0)}$, let $\kappa^w$ be the Radon measure on $L$ given on $F\in\mathcal{C}(L)$ by
\[ \kappa^w(F)=\begin{cases} \lambda^w(F|_G) + \sigma^w_Z(F|_Z) \;&\text{if}\;w\in G^{(0)}, \\ \sigma^w_{Z^{op}}(F|_{Z^{op}}) + \beta^w(F|_H) \;&\text{if}\;w\in H^{(0)}. \end{cases} \]
Then $\{\kappa^w\}_{w\in L^{(0)}}$ is a Haar system for $L$.
\end{lem}

We shall always assume that $L$ is equipped with the above Haar system.

For $F\in\mathcal{C}(L)$, let $F_{11}=F|_G\in\mathcal{C}(G)$, $F_{12}=F|_Z\in\mathcal{C}(Z)$, $F_{21}=F|_{Z^{op}}\in\mathcal{C}(Z^{op})$, $F_{22}=F|_H\in\mathcal{C}(H)$.
We can view $F$ as a matrix \[ F=\begin{pmatrix} F_{11} & F_{12} \\ F_{21} & F_{22} \end{pmatrix}. \]
The convolution product on $\mathcal C(L)$ gives the following module
actions and pairings. 

Let $a \in\mathcal{C}(G)$ and let $\phi\in\mathcal{C}(Z)$. Observe that 
\[
 \begin{pmatrix} a & 0 \\ 0 & 0 \end{pmatrix}*\begin{pmatrix} 0 & \phi \\ 0 & 0 \end{pmatrix}(z)= \int_G a(l)\phi(l^{-1}z)\;d\kappa^{r(z)}(l) = \sum_{r(\gamma)=r(z)}a(\gamma)\phi(\gamma^{-1}\cdot z).
\]
This gives a left action of $\mathcal{C}(G)$ on $\mathcal{C}(Z)$ defined by
\begin{equation}\label{cG_left_cZ}
(a\cdot\phi)(z) = \sum_{r(\gamma)=r(z)} a(\gamma)\phi(\gamma^{-1}\cdot z). 
\end{equation}
Next, let $\phi\in\mathcal{C}(Z)$ and let $b\in\mathcal{C}(H)$. Then 
\begin{align*}
 \begin{pmatrix} 0 & \phi \\ 0 & 0 \end{pmatrix}*\begin{pmatrix} 0 & 0 \\ 0 & b \end{pmatrix}(z) 
&= \int_L \phi(l)b(l^{-1}z)\;d\kappa^{r(z)}(l) \\
&= \int_Z \phi(w)b(w^{-1}z)\;d\sigma_Z^{r(z)}(w) \\
&= \sum_{r(\eta)=s(z)}\phi(z\cdot\eta)b(\overline{z\cdot\eta}z) \\
&= \sum_{r(\eta)=s(z)}\phi(z\cdot\eta)b(\eta^{-1}\cdot[z,z]_H) \\
&= \sum_{r(\eta)=s(z)}\phi(z\cdot\eta)b(\eta^{-1}\cdot s(z)) \\
&= \sum_{r(\eta)=s(z)}\phi(z\cdot\eta)b(\eta^{-1}).
\end{align*}
Hence, a right action of $\mathcal{C}(H)$ on $\mathcal{C}(Z)$ 
is defined by letting 
\begin{equation}\label{cZ_right_cH}
(\phi\cdot b)(z) = \sum_{r(\eta)=s(z)} \phi(z\cdot\eta)b(\eta^{-1}).
\end{equation}
Similarly, we make $\mathcal{C}(Z^{op})$  into a $\mathcal{C}(H)$-$\mathcal{C}(G)$-bimodule with actions, for $a\in\mathcal{C}(G)$, $b\in\mathcal{C}(H)$, and $\psi\in\mathcal{C}(Z^{op})$, 
given by
\begin{align}\label{cH_Zop_cG}
(\psi \cdot a)(\bar{z})  & =  \sum_{r(\gamma)=r(\bar{z})} \psi(\bar{z}\cdot\gamma)a(\gamma^{-1}), \\
(b \cdot \psi)(\bar{z}) & =\sum_{r(\eta)=s(\bar{z})} b(\eta)\psi(\eta^{-1}\cdot \bar{z}).
 \end{align}
Now for the pairings, first let $\phi\in\mathcal{C}(Z)$ and let $\psi\in\mathcal{C}(Z^{op})$. 
Then, 
\begin{align*}
 \begin{pmatrix} 0 & 0 \\ \psi & 0 \end{pmatrix}*\begin{pmatrix} 0 & \phi \\ 0 & 0 \end{pmatrix}(\eta) 
&= \int_L \psi(l)\phi(l^{-1}\eta)\;d\kappa^{r(\eta)}(l) \\
&= \int_{Z^{op}}\psi(\bar{z})\phi(\bar{z}^{-1}\cdot\eta)\;d\sigma_{Z^{op}}^{r(\eta)}(\bar{z}) \\
&= \sum_{r(\gamma)=s(\bar{z})}\psi(\bar{z}\cdot\gamma)\phi((\bar{z}\cdot\gamma)^{-1}\cdot\eta) \\
&= \sum_{r(\gamma)=r(z)}\psi(\overline{\gamma^{-1}\cdot z})\phi((\overline{\gamma^{-1}\cdot z})^{-1}\cdot\eta) \\
&= \sum_{r(\gamma)=r(z)}\psi(\overline{\gamma^{-1}\cdot z})\phi(\gamma^{-1}\cdot z\cdot\eta). 
\end{align*}
Thus, $\langle\cdot,\cdot\rangle_H \colon \mathcal{C}(Z^{op})\times\mathcal{C}(Z)\rightarrow\mathcal{C}(H)$
defines a $\mathcal{C}(H)$-valued pairing via 
\begin{equation}\label{cZop_cZ-cH}
\langle \psi,\phi \rangle_H(\eta) =\sum_{r(\gamma)=r(z)}\psi\Big(\overline{\gamma^{-1}\cdot z}\Big)\phi(\gamma^{-1}\cdot z\cdot\eta), 
\end{equation}
for any $z\in Z$ such that $s_Z(z)=r(\eta)$. 
Finally, for $\phi\in\mathcal{C}(Z)$ and $\psi\in\mathcal{C}(Z^{op})$, we get
\begin{align*} 
 \begin{pmatrix} 0 & \phi \\ 0 & 0 \end{pmatrix}*\begin{pmatrix} 0 & 0 \\ \psi & 0 \end{pmatrix}(\gamma) 
&= \int_L \phi(l)\psi(l^{-1}\gamma)\;d\kappa^{r(\gamma)}(l) \\
&= \int_Z\phi(z)\psi(z^{-1}\cdot\gamma)\;d\sigma_Z^{r(\gamma)}(z) \\
&= \sum_{r(\eta)=s(z)}\phi(z\cdot\eta)\psi((z\cdot\eta)^{-1}\cdot\gamma) \\
&= \sum_{r(\eta)=s(z)}\phi(z\cdot\eta)\psi(\overline{z\cdot\eta}\cdot\gamma) \\
&= \sum_{r(\eta)=s(z)} \phi(z\cdot\eta)\psi(\overline{\gamma^{-1}\cdot z\cdot\eta}).
\end{align*}
Hence, we define a $\mathcal{C}(G)$-valued pairing, ${}_G\langle\cdot,\cdot\rangle \colon \mathcal{C}(Z)\times\mathcal{C}(Z^{op})\rightarrow\mathcal{C}(G)$, 
as 
\begin{equation}\label{cZ_cZop-cG}
 {}_G\langle \phi,\psi \rangle(\gamma) =\sum_{r(\eta)=s(z)} \phi(z\cdot\eta)\psi\left(\overline{\gamma^{-1}\cdot z\cdot\eta}\right),
\end{equation}
for any $z\in Z$ such that $r_Z(z)=r(\gamma)$.

As part of the proof of our Morita equivalence result, we shall need to compare the norms of $a\in\mathcal{C}(G)$ and 
\[
F_a \coloneqq \begin{pmatrix} a & 0 \\ 0 & 0 \end{pmatrix}\in\mathcal{C}(L).
\]
It is clear from the definitions that $\Vert a\Vert_{F_{I,s}(G)}=\Vert F_a\Vert_{F_{I,s}(L)}$ and $\Vert a\Vert_{F_{I,r}(G)}=\Vert F_a\Vert_{F_{I,r}(L)}$, from which it follows that $\Vert a\Vert_{F_I(G)}=\Vert F_a\Vert_{F_I(L)}$.

\begin{thm}\label{thm:mainF_red^p} (cf. \cite[Theorem 4.1]{SW} for the $p=2$ case)
If $a\in\mathcal{C}(G)$, and $F_a \in\mathcal{C}(L)$, then $\Vert F_a\Vert_{F^p_{red}(L)}=\Vert a\Vert_{F^p_{red}(G)}$ for $p\in[1,\infty]$.

A similar statement holds for functions in $\mathcal{C}(H)$.
\end{thm}

\begin{proof}
Since $\Vert\cdot\Vert_{1,red}=\Vert\cdot\Vert_{I,s}$ and $\Vert\cdot\Vert_{\infty,red}=\Vert\cdot\Vert_{I,r}$, we just need to consider $p\in(1,\infty)$.

For $u\in G^{(0)}$, we have $L_u=G_u\sqcup Z^{op}_u$, and $\mathrm{Ind}^L\;\delta_u$ acts on $L^p(L_u,\kappa_u)$, where $\kappa_u$ is the forward image of $\kappa^u$ under inversion.
Let $\rho^u_{Z^{op}}$ be the Radon measure on $Z^{op}$ which is the image of $\sigma^u_Z$ under inversion.
Then $L^p(L_u,\kappa_u)=L^p(G_u,\lambda_u)\oplus_p L^p(Z^{op}_u,\rho^u_{Z^{op}})$. 
With respect to this decomposition, we have $(\mathrm{Ind}^L\;\delta_u)(F_a)=(\mathrm{Ind}^G\;\delta_u)(a)\oplus 0$.
Thus we have 
\begin{align*}
\Vert F_a\Vert_{F^p_{red}(L)} &= \max\left\{ \sup_{u\in G^{(0)}}\Vert(\mathrm{Ind}^L\;\delta_u)(F_a)\Vert,\sup_{v\in H^{(0)}}\Vert(\mathrm{Ind}^L\;\delta_v)(F_a)\Vert \right\} \\
&= \max\left\{ \Vert a\Vert_{F^p_{red}(G)},\sup_{v\in H^{(0)}}\Vert(\mathrm{Ind}^L\;\delta_v)(F_a)\Vert \right\}.
\end{align*}

We shall show that $\Vert(\mathrm{Ind}^L\;\delta_v)(F_a)\Vert\leq\Vert a\Vert_{F^p_{red}(G)}$ for all $v\in H^{(0)}$.

For $v\in H^{(0)}$, let $Z_v=\{z\in Z:s(z)=v\}$. Then $L_v=Z_v\sqcup H_v$, and $L^p(L_v,\kappa_v)=L^p(Z_v,\rho^v_Z)\oplus_p L^p(H_v,\beta_v)$, where $\rho^v_Z$ is the image of $\sigma^v_{Z^{op}}$ under inversion, i.e., for $\phi\in\mathcal{C}(Z)$, we have $\rho^v_Z(\phi)=\int_G\phi(\gamma^{-1}z)\;d\lambda^{r(z)}(\gamma)$ for any $z\in Z$ with $s(z)=v$. Moreover, $Z_v$ may be identified with $G\cdot z$ since $s$ induces a homeomorphism of $G\backslash Z$ onto $H^{(0)}$.

The action of $G$ on $Z$ is free and proper, so the map $\gamma\mapsto\gamma\cdot z$ is a homeomorphism of $G_{r(z)}$ onto the orbit $G\cdot z$, which is closed in $Z$ \cite[Lemma 1.2]{SW}.
This induces an invertible isometry from $L^p(G\cdot z,\rho^{s(z)}_Z)$ onto $L^p(G_{r(z)},\lambda_{r(z)})$. 
This invertible isometry intertwines $\mathrm{Ind}\;\delta_{r(z)}$ with a representation $R_{G\cdot z}(a)$ on $L^p(G\cdot z,\rho^{s(z)}_Z)$ given by
\[ R_{G\cdot z}(a)\xi(\gamma\cdot z) = \sum_{r(\eta)=r(\gamma)}a(\eta)\xi(\eta^{-1}\gamma\cdot z), \]
for any $\xi \in L^p(G\cdot z,\rho^{s(z)}_Z)$. 
In particular, $\Vert R_{G\cdot z}(a)\Vert\leq\Vert a\Vert_{F^p_{red}(G)}$.

Finally, we have $(\mathrm{Ind}^L\;\delta_v)(F_a)=R_{G\cdot z}(a)\oplus 0$, so $\Vert F_a\Vert_{F^p_{red}(L)}=\Vert a\Vert_{F^p_{red}(G)}$.
\end{proof}

For $\phi\in\mathcal{C}(Z)$ and $\psi\in\mathcal{C}(Z^{op})$, we define
\[
\Vert\phi\Vert_L \coloneqq \left\Vert\begin{pmatrix} 0 & \phi \\ 0 & 0 \end{pmatrix}\right\Vert_{F^p_{red}(L)}, \ \Vert\psi\Vert_L  \coloneqq \left\Vert\begin{pmatrix} 0 & 0 \\ \psi & 0 \end{pmatrix}\right\Vert_{F^p_{red}(L)}.
\]
Then the actions and pairings defined above satisfy the norm inequalities required for a pre-Morita equivalence:

\begin{lem} \label{preMor}
Let $p \in [1, \infty]$. The actions and pairings defined in Equations~\eqref{cG_left_cZ}-\eqref{cZ_cZop-cG} 
make the pair $(\mathcal{C}(Z^{op}), \mathcal{C}(Z))$, equipped with the $\Vert-\Vert_L$ norms, into a pre-Morita equivalence between $\mathcal{C}(G) \subseteq F_{red}^p(G)$ and $\mathcal{C}(H) \subseteq F_{red}^p(H)$ (see Definition \ref{Def:PreMORITA}). 
\end{lem}

\begin{proof}
Tracing the required conditions for Definition~\ref{Def:PreMORITA} back to Definitions~\ref{Def:PreBanachPair} and~\ref{Def:PreCorresp}, we first check that, for every $a\in\mathcal{C}(G)$, $b\in\mathcal{C}(H)$, $\phi\in\mathcal{C}(Z)$, and $\psi\in\mathcal{C}(Z^{op})$, the following identities hold:
\begin{enumerate}
\item ${}_G\langle a \cdot \phi,\psi \rangle = a * {}_G\langle \phi,\psi \rangle$,
\item ${}_G\langle \phi,\psi \cdot a \rangle = {}_G\langle \phi,\psi \rangle*a$,
\item $\langle \psi\cdot a,\phi \rangle_H=\langle \psi,a\cdot\phi \rangle_H$,
\item $\langle b \cdot \psi,\phi \rangle_H = b*\langle \psi,\phi \rangle_H$, 
\item $\langle \psi,\phi \cdot b \rangle_H = \langle \psi,\phi \rangle_H * b$, 
\item ${}_G\langle \phi,b\cdot\psi \rangle={}_G\langle \phi\cdot b,\psi \rangle$,
\item $\langle \psi,\phi  \rangle_H \cdot \psi' = \psi \cdot {}_G\langle \phi,\psi' \rangle$,
\item $ {}_G\langle \phi,\psi \rangle \cdot \phi' =\phi \cdot \langle \psi,\phi'  \rangle_H$.
\end{enumerate}

Under the corner identifications \[ \mathcal{C}(L) = \begin{pmatrix} \mathcal{C}(G) & \mathcal{C}(Z) \\ \mathcal{C}(Z^{op}) & \mathcal{C}(H) \end{pmatrix}, \]
all eight identities follow from associativity of convolution. For example, the last identity follows by associating \[ 
 \begin{pmatrix}
0& \phi  \\
0 & 0 
\end{pmatrix}*
 \begin{pmatrix}
0& 0   \\
\psi & 0 
\end{pmatrix}*
 \begin{pmatrix}
0& \phi'  \\
0 & 0 
\end{pmatrix}.
\]
Hence all the algebraic compatibility identities follow from associativity of convolution.

Finally, at the norm level, we also have to check that, 
for all $a \in \mathcal{C}(G)$, $b \in \mathcal{C}(H)$, 
$\phi \in \mathcal{C}(Z)$, and $\psi \in \mathcal{C}(Z^{op})$
we have 
\begin{enumerate}
 \setcounter{enumi}{8}
 \item $\| a \cdot \phi \|_L \leq \| a\|_{F^p_{red}(G)} \|\phi\|_L$,
 \item $\| \psi \cdot a \|_L \leq  \|\psi\|_L\| a\|_{F^p_{red}(G)}$,
  \item $\| b \cdot \psi  \|_L \leq  \| b\|_{F^p_{red}(H)}\|\psi\|_L$,
  \item  $\|  \phi \cdot b \|_L \leq  \|\phi\|_L\| b\|_{F^p_{red}(H)}$,
\item $\| \langle \psi,\phi  \rangle_H \|_{F^p_{red}(H)} \leq \| \psi \|_{L}\|\phi\|_{L}$,
\item $\| _{G}\langle \phi,\psi  \rangle\|_{F^p_{red}(G)} \leq \| \phi \|_{L} \|\psi\|_{L}$.
\end{enumerate}
All six norm inequalities follow from submultiplicativity in $F^p_{\mathrm{red}}(L)$ and Theorem~\ref{thm:mainF_red^p}.  For example, for the first norm inequality, we have 
\[
\| a \cdot \phi \|_L  =  \left\| 
\begin{pmatrix}
a &  0\\
0 & 0 
\end{pmatrix} 
*
 \begin{pmatrix}
0& \phi  \\
0 & 0 
\end{pmatrix}
\right\|_{F^p_{red}(L)}
\leq \|F_a\|_{F^p_{red}(L)} \|\phi\|_L = \|a\|_{F^p_{red}(G)} \| \phi\|_L.
\]
\end{proof}

\section{Approximate identities}\label{sect_AI}

In this section, we adapt the techniques in \cite{MW} to show that $\mathcal{C}(G)$ contains an approximate identity $\{e_\lambda\}_{\lambda\in\Lambda}$ with respect to the inductive limit topology for both the convolution product on $\mathcal{C}(G)$ and the actions of $\mathcal{C}(G)$ on $\mathcal{C}(Z)$ and $\mathcal{C}(Z^{op})$, where each $e_\lambda$ is a finite sum of elements of the form ${}_G\langle \phi,\psi \rangle$ with $\phi\in\mathcal{C}(Z)$ and $\psi\in\mathcal{C}(Z^{op})$.

Given a locally Hausdorff space $X$, a net $\{f_i\}\subset\mathcal{C}(X)$ is said to converge to $f\in\mathcal{C}(X)$ in the \emph{inductive limit topology} on $\mathcal{C}(X)$ if there is a compact set $K\subset X$ (independent of $i$) and some $i_0$ such that $f_i\to f$ uniformly and $f_i$ vanishes outside $K$ for all $i\geq i_0$.

A subset $U\subset G$ is called \emph{conditionally compact} if $VU$ and $UV$ are relatively compact whenever $V$ is relatively compact in $G$.\footnote{As in \cite{MW}, in the non-Hausdorff setting, a set is said to be relatively compact if it is contained in a compact set; it need not have compact closure.}
We say that $U$ is \emph{diagonally compact} if $VU$ and $UV$ are compact whenever $V$ is compact.
Note that if $U$ is diagonally compact, then so is $U^{-1}$, and thus also $U\cup U^{-1}$.
Also, every diagonally compact subset of $G$ is conditionally compact.

\begin{lem} \cite[Lemma 2.10]{MW} \label{lem:diagcpt}  
Suppose that $G$ is a locally compact, locally Hausdorff groupoid. If $G^{(0)}$ is paracompact, then $G$ has a fundamental system $\mathcal{U}$ of conditionally (or diagonally) compact neighborhoods of $G^{(0)}$, i.e., given any neighborhood $V$ of $G^{(0)}$, there exists $U\in\mathcal{U}$ such that $U\subset V$.
\end{lem}

\begin{lem} \cite[Lemma 2.13]{MW} \label{lem:nbh}
Suppose that $G$ is a locally compact, locally Hausdorff groupoid, and that $K\subset G^{(0)}$ is compact.
Then there is a neighborhood $W$ of $G^{(0)}$ in $G$ such that $W\cap r^{-1}(K)$ is Hausdorff.
\end{lem}

\begin{lem} \cite[Lemma 2.14]{MW} \label{lem:nbh1}
Suppose that $G$ is a locally compact, locally Hausdorff groupoid, and that $Z$ is a locally compact, locally Hausdorff left $G$-space.
If $V$ is open in $Z$, and if $K\subset V$ is compact, then there is a neighborhood $W$ of $G^{(0)}$ in $G$ such that $W\cdot K\subset V$.
\end{lem}

\begin{lem} \cite[Lemma 2.15]{MW} \label{lem:nbh2} 
Suppose that $G$ is a locally compact, locally Hausdorff groupoid, and that $Z$ is a locally compact, locally Hausdorff, free and proper left $G$-space.
If $W$ is a neighborhood of $G^{(0)}$ in $G$, then each $z\in Z$ has a neighborhood $V$ such that $(\gamma\cdot z,z)\in V\times V$ implies that $\gamma\in W$. 
\end{lem}

\begin{prop} \cite[Proposition 2.16]{MW} \label{prop:fullsys} 
Suppose that $H$ is a locally compact, locally Hausdorff groupoid with Haar system $\{\beta^v\}_{v\in H^{(0)}}$, and that $Z$ is a locally compact, locally Hausdorff, free and proper right $H$-space.
Let $q:Z\rightarrow Z/H$ be the quotient map, and let $V\subset Z$ be an open Hausdorff set such that $q(V)$ is Hausdorff.
\begin{enumerate}
\item If $\psi\in C_c(V)$, then $\lambda(\psi)(q(z))=\int_H \psi(z\cdot\eta)\;d\beta^{s(z)}(\eta)$ defines an element $\lambda(\psi)\in C_c(q(V))$.
\item If $d\in C_c(q(V))$, then there is a $\psi\in C_c(V)$ such that $\lambda(\psi)=d$. Moreover, if $d$ is positive, then $\psi$ may be taken to be positive. 
\end{enumerate}
\end{prop}

We use the following form of \cite[Claim~6.7]{MW}.

\begin{lem} \label{claim6.7}
Let $V$ be an open Hausdorff subset of $Z$, let $\phi \in C_c(V)$, let $K_1 \coloneqq \mathrm{supp}_V \phi \subset V$, and assume that $G^{(0)}$ is paracompact. Then there is a diagonally compact neighborhood $W_1$ of $G^{(0)}$ in $G$
such that 
\begin{enumerate}
\item $K_2 \coloneqq W_1 \cdot K_1 \subset V$, 
\item  $r_Z(K_2)W_1=W_1 \cap r^{-1}(r_Z(K_2))$ is Hausdorff,
\item  for each $\varepsilon>0$, there is a diagonally compact neighborhood $U_1$ of $G^{(0)}$ in $G$ such that $U_1\subset W_1$ and such that $\gamma\in U_1$ implies that $|\phi(\gamma^{-1}\cdot z)-\phi(z)|<\varepsilon$ for all $z\in Z$ satisfying $r_Z(z)=r(\gamma)$.
\end{enumerate} 
\end{lem}

\begin{proof}
By Lemma~\ref{lem:nbh1}, there is a neighborhood $W$ of $G^{(0)}$ in $G$ such
that
\[
  W\cdot K_1\subseteq V.
\]
Using Lemma~\ref{lem:diagcpt}, choose a diagonally compact neighborhood $W'$ of
$G^{(0)}$ such that $W'\subseteq W$.  The set
\[
  K'\coloneqq W'\cdot K_1
\]
is compact.  Indeed, $r_Z(K_1)$ is compact, and diagonal compactness of
$W'$ implies that
\[
  W'\cdot r_Z(K_1) = W'\cap s^{-1}(r_Z(K_1))
\]
is compact.  The space of composable pairs
\[
  \bigl(W'\cdot r_Z(K_1)\bigr)
  *_{s,r_Z} K_1 = \{(\gamma,z)\in W'\cdot r_Z(K_1)\times K_1:s(\gamma)=r_Z(z)\}
\]
is a closed subset of a compact space because it is the preimage of the diagonal $\Delta_{G^{(0)}}$ under the continuous map
\begin{align*}
\bigl(W'\cdot r_Z(K_1)\bigr) \times K_1 &\to G^{(0)} \times G^{(0)}, \\
(\gamma,z) &\mapsto (s(\gamma),r_Z(z)),
\end{align*} 
and $G^{(0)}$ is Hausdorff.
Its image under the action map is $K'$, and hence $K'$ is compact.
Moreover, $K'\subseteq V$.

Apply Lemma~\ref{lem:nbh} to the compact set $r_Z(K')\subseteq G^{(0)}$.
This gives a neighborhood $N$ of $G^{(0)}$ such that
\[
  N\cap r^{-1}(r_Z(K'))
\]
is Hausdorff.  Apply Lemma~\ref{lem:diagcpt} once more to choose a diagonally compact
neighborhood $W_1$ of $G^{(0)}$ such that
\[
  W_1\subseteq W'\cap N,
\]
and put
\[
  K_2\coloneqq W_1\cdot K_1.
\]
The argument above, with $W_1$ in place of $W'$, shows that $K_2$ is compact,
and
\[
  K_1\subseteq K_2\subseteq K'\subseteq V.
\]
This proves~(1).  

Since $r_Z(K_2)\subseteq r_Z(K')$, we have
\[
  r_Z(K_2)W_1
  =W_1\cap r^{-1}(r_Z(K_2))
  \subseteq N\cap r^{-1}(r_Z(K')).
\]
Thus $r_Z(K_2)W_1$ is Hausdorff, proving~(2).  

Notice also that it is
compact, because $W_1$ is diagonally compact and $r_Z(K_2)$ is compact.

Fix $\varepsilon>0$.  Let $\mathcal{U}$ be the set of all diagonally
compact neighborhoods $U$ of $G^{(0)}$ such that $U\subseteq W_1$,
directed by reverse inclusion.  It is a directed set and a fundamental
system of neighborhoods of $G^{(0)}$ by Lemma~\ref{lem:diagcpt}.  We claim that some
$U\in\mathcal{U}$ satisfies~(3).  Suppose, to the contrary, that this is
false.  For every $U\in\mathcal{U}$, choose $\gamma_U\in U$ and
$z_U\in Z$ such that
\[
  r(\gamma_U)=r_Z(z_U)
\]
and
\begin{equation}
  \left|\phi(\gamma_U^{-1}\cdot z_U)-\phi(z_U)\right|
  \geq \varepsilon.
  \label{eq:lemma46-bad}
\end{equation}
We may, and shall, take $z_U\in K_2$.  To see this, if the left-hand
side of~\eqref{eq:lemma46-bad} is nonzero, then either
$\phi(z_U)\neq 0$, in which case $z_U\in K_1\subseteq K_2$, or
$\phi(\gamma_U^{-1}\cdot z_U)\neq 0$.  In the latter case,
$\gamma_U^{-1}\cdot z_U\in K_1$, and hence
\[
  z_U\in U\cdot K_1\subseteq W_1\cdot K_1=K_2.
\]

Because $r(\gamma_U)=r_Z(z_U)\in r_Z(K_2)$, the net
$(\gamma_U,z_U)$ lies in the compact space
\[
  \bigl(r_Z(K_2)W_1\bigr)\times K_2.
\]
Both factors are Hausdorff: the first by~(2), and the second because it
is contained in the Hausdorff open set $V$.  Passing to a subnet and
relabeling, we may therefore assume that
\[
  \gamma_U\longrightarrow\gamma\in r_Z(K_2)W_1,
  \qquad
  z_U\longrightarrow z\in K_2.
\]

We claim that $\gamma\in G^{(0)}$.  Fix $U_0\in\mathcal{U}$.  The
net is eventually indexed by neighborhoods $U\subseteq U_0$, and hence
is eventually contained in
\[
  r_Z(K_2)U_0
  =U_0\cap r^{-1}(r_Z(K_2)).
\]
This set is compact by diagonal compactness of $U_0$.  It is contained
in the Hausdorff space $r_Z(K_2)W_1$, and is therefore closed there.
Consequently,
\[
  \gamma\in r_Z(K_2)U_0\subseteq U_0.
\]
Since $U_0$ was arbitrary and $\mathcal{U}$ is a fundamental system of
neighborhoods of $G^{(0)}$, it follows that
\[
  \gamma\in\bigcap_{U_0\in\mathcal{U}}U_0=G^{(0)},
\]
where the equality is due to a locally Hausdorff space being $T_1$: if $\alpha\notin G^{(0)}$, then
$G\setminus\{\alpha\}$ is a neighborhood of $G^{(0)}$, and some member
of $\mathcal{U}$ is contained in it.

Continuity of $r$ and $r_Z$, together with
$r(\gamma_U)=r_Z(z_U)$, gives
\[
  r(\gamma)=r_Z(z).
\]
Since $\gamma$ is a unit, $r(\gamma)=\gamma$, and therefore
\[
  \gamma=r_Z(z).
\]
It follows from continuity of inversion and of the action that
\[
  \gamma_U^{-1}\cdot z_U
  \longrightarrow
  \gamma^{-1}\cdot z=z.
\]
Now $z_U\in K_2\subseteq V$ for all $U$, while
$\gamma_U^{-1}\cdot z_U$ is eventually in $V$ because it converges to
$z\in V$ and $V$ is open.  Since $\phi|_V$ is continuous, we obtain
\[
  \phi(z_U)\longrightarrow\phi(z),
  \qquad
  \phi(\gamma_U^{-1}\cdot z_U)\longrightarrow\phi(z),
\]
contradicting~\eqref{eq:lemma46-bad}.  Hence some
$U_1\in\mathcal{U}$ satisfies~(3).  
\end{proof}

\begin{prop} (cf. \cite[Proposition 6.6]{MW}) \label{prop:approxID1} 
Suppose that $G$ is a locally compact, locally Hausdorff, \'{e}tale groupoid with $G^{(0)}$ paracompact, and that $Z$ is a locally compact, locally Hausdorff left $G$-space. 
Suppose that for each triple $(K,U,\varepsilon)$ consisting of a nonempty compact subset $K\subset G^{(0)}$, a conditionally compact neighborhood $U$ of $G^{(0)}$ in $G$, and $\varepsilon>0$, there is $e=e_{(K,U,\varepsilon)}\in\mathcal{C}(G)$ such that
\begin{enumerate}
\item $e(\gamma)=0$ whenever $\gamma\notin U$,
\item $e(\gamma)\geq 0$ for all $\gamma\in G$, 
\item $|\sum_{r(\gamma)=u} e(\gamma)-1|<\varepsilon$ for all $u\in K$. 
\end{enumerate}
Then the net $\{e_{(K,U,\varepsilon)}\}$ directed by increasing $K$, and decreasing $U$ and $\varepsilon$, is an approximate identity in the inductive limit topology for the left action of $\mathcal{C}(G)$ on $\mathcal{C}(Z)$.
\end{prop}

\begin{proof}
Since $\mathcal{C}(Z)$ is spanned by the spaces $C_c(V)$, where $V\subset Z$ is open and Hausdorff, it suffices to show that $e\cdot\phi\rightarrow\phi$ in the inductive limit topology for $\phi\in C_c(V)$. 
Let $K_1 \coloneqq \mathrm{supp}_V \phi \subset V$. 
Apply Lemma~\ref{claim6.7} to obtain a diagonally compact neighborhood $W_1$ of $G^{(0)}$ such that $K_2 \coloneqq W_1\cdot K_1$ is a compact subset of $V$. Notice that $K_1\subseteq K_2$, since $G^{(0)}\subseteq W_1$.
Fix $\delta>0$. By Lemma~\ref{claim6.7}, there is a conditionally compact neighborhood $U_1$ of $G^{(0)}$, with $U_1\subseteq W_1$, such that 
$\left|\phi(\gamma^{-1}\cdot z)-\phi(z)\right|<\delta$
whenever $\gamma\in U_1$ and $r(\gamma)=r_Z(z)$.

Consider $U\subset U_1$, $K\supset r_Z(K_2)$, $\varepsilon<\delta$, and write $e=e_{(K,U,\varepsilon)}$.
For $U\subseteq U_1$, the function $e\cdot\phi$ vanishes outside
$K_2$. 
Indeed, if a summand $e(\gamma)\phi(\gamma^{-1}\cdot z)$ is nonzero, then $\gamma\in U\subset U_1$ while $\gamma^{-1}\cdot z \in K_1$. Consequently, $z\in U\cdot K_1\subseteq U_1\cdot K_1
\subseteq W_1\cdot K_1=K_2$.
Thus $e\cdot\phi$ vanishes outside $K_2$.
Since $\phi$ also vanishes outside $K_1\subseteq K_2$, the difference $e\cdot\phi-\phi$ vanishes outside the fixed compact set $K_2$.

On the other hand, if $z\in K_2$, then $r_Z(z)\in K$, and 
\begin{align*}
&\left|e\cdot\phi(z)-\phi(z)\right| \\ = &\left|\sum_{r(\gamma)=r_Z(z)}e(\gamma)\phi(\gamma^{-1}\cdot z)-\phi(z)\right| \\
\leq &\left|\sum_{r(\gamma)=r_Z(z)}e(\gamma)(\phi(\gamma^{-1}\cdot z)-\phi(z))\right| + \left|\sum_{r(\gamma)=r_Z(z)}e(\gamma)-1\right| |\phi(z)|.
\end{align*}
If $e(\gamma)\neq 0$, then $\gamma\in U\subset U_1$, so $\left|\phi(\gamma^{-1}\cdot z)-\phi(z)\right|<\delta$.
Also, $0\leq\sum_{r(\gamma)=r_Z(z)}e(\gamma)<1+\varepsilon<1+\delta$.
It follows that \[ \left|e\cdot\phi(z)-\phi(z)\right| < \delta(1+\delta)+\delta\|\phi\|_\infty. \]
Hence, $e\cdot\phi\to\phi$ in the inductive limit topology.
\end{proof}

For $\phi\in\mathcal{C}(Z)$, define $\tilde{\phi}\in\mathcal{C}(Z^{op})$ by $\tilde{\phi}(\bar{z})=\phi(z)$.

\begin{prop} (cf. \cite[Proposition 6.8]{MW}) \label{prop:ai} 
Suppose that $G$ and $H$ are locally compact, locally Hausdorff, \'{e}tale groupoids with $G^{(0)}$ paracompact, and that $Z$ is a $(G,H)$-equivalence.
There is a net $\{e_\lambda\}_{\lambda\in\Lambda}$ in $\mathcal{C}(G)$ that is an approximate identity with respect to the inductive limit topology for the left action of $\mathcal{C}(G)$ on $\mathcal{C}(Z)$ and for the right action of $\mathcal{C}(G)$ on $\mathcal{C}(Z^{op})$, where each $e_\lambda$ is a finite sum of elements of the form
${}_G\langle \phi_i^\lambda,\widetilde{\phi_i^\lambda} \rangle$ 
with $\phi_i^\lambda\in\mathcal{C}(Z)$.
\end{prop}

\begin{proof}
Fix a triple $(K,U,\varepsilon)$ as in Proposition~\ref{prop:approxID1}.  Choose
relatively compact Hausdorff open subsets
$O_1,\ldots,O_n\subseteq Z$ such that
\[
    K\subseteq\bigcup_{i=1}^n r_Z(O_i).
\]
Since $G^{(0)}$ is paracompact, choose
$h_i\in C_c^+(G^{(0)})$ such that
\[
    \mathrm{supp}(h_i)\subseteq r_Z(O_i),\qquad
    \sum_{i=1}^n h_i=1\ \text{on }K,\qquad
    \sum_{i=1}^n h_i\leq 1.
\]
After discarding the indices for which
$K\cap\mathrm{supp}(h_i)=\emptyset$, we may assume that
$K\cap\mathrm{supp}(h_i)\neq\emptyset$ for every $i$.

We use the following observation.  If
$C\subseteq r_Z(O)$ is compact, where $O\subseteq Z$ is open
Hausdorff, then there is a compact set $D\subseteq O$ such that
$r_Z(D)=C$.  Indeed, for each $u\in C$ choose
$z_u\in O$ with $r_Z(z_u)=u$ and a compact neighborhood
$N_u\subseteq O$ of $z_u$.  Since $r_Z$ is open, finitely many of
the sets $r_Z(\operatorname{int}N_u)$ cover $C$.  Intersecting the
corresponding finite union of the $N_u$ with $r_Z^{-1}(C)$ gives
the desired compact set.

For each $i$, apply this observation to obtain a nonempty compact
set
\[
    C_i\subseteq O_i,\qquad
    r_Z(C_i)=K\cap\mathrm{supp}(h_i).
\]
Choose a compact neighborhood $D_i$ of $C_i$ contained in $O_i$,
and choose $\chi_i\in C_c^+(O_i)$ such that
\[
    0\leq\chi_i\leq1,\qquad \chi_i=1\ \text{on }D_i.
\]
Put
\[
    \theta_i(z)
    =
    \bigl(\chi_i(z)h_i(r_Z(z))\bigr)^{1/2}
    \qquad (z\in O_i),
\]
extended by zero outside $O_i$.  Then
$\theta_i\in C_c^+(O_i)$,
\[
    0\leq\theta_i\leq1,
    \qquad
    \theta_i(z)^2=h_i(r_Z(z))
    \quad (z\in D_i).
\]

Set
\[
    \delta=\min\left\{\frac12,\frac{\varepsilon}{3n}\right\}.
\]
Applying Lemma~\ref{claim6.7} to $\theta_i$, and shrinking the resulting
neighborhood by intersecting with $U$ if necessary, choose a
conditionally compact neighborhood $W_i$ of $G^{(0)}$ such that
$W_i\subseteq U$ and
\[
    |\theta_i(\gamma^{-1}\cdot z)-\theta_i(z)|<\delta
\]
whenever $\gamma\in W_i$ and $r(\gamma)=r_Z(z)$.  Put
\[
    W=\bigcap_{i=1}^n W_i.
\]
Then $W$ is a conditionally compact neighborhood of $G^{(0)}$
contained in $U$.

By Lemma~\ref{lem:nbh2}, for each $i$ we may choose relatively compact open
subsets
\[
    V_1^i,\ldots,V_{k_i}^i\subseteq D_i
\]
covering $C_i$ such that
\[
    (\gamma\cdot z,z)\in V_j^i\times V_j^i
    \quad\Longrightarrow\quad
    \gamma\in W.
\]
Choose $d_j^i\in C_c^+(G^{(0)})$ satisfying
\[
    \mathrm{supp}(d_j^i)\subseteq r_Z(V_j^i),\qquad
    \sum_{j=1}^{k_i}d_j^i=1\ \text{on }r_Z(C_i),\qquad
    \sum_{j=1}^{k_i}d_j^i\leq1.
\]
By Proposition~\ref{prop:fullsys}, there are
$\psi_j^i\in C_c^+(V_j^i)$ such that
\[
    \sum_{r(\eta)=s_Z(z)}
        \psi_j^i(z\cdot\eta)
    =
    d_j^i(r_Z(z)).
\]

Put
\[
    M_i
    =
    \sup_{z\in Z}
    \sum_{j=1}^{k_i}
    \sum_{r(\eta)=s_Z(z)}
        \chi_{V_j^i}(z\cdot\eta).
\]
Then $1\leq M_i<\infty$.  The lower bound follows because
$C_i\neq\emptyset$ and the $V_j^i$ cover $C_i$.  For the upper
bound, recall that $\chi_i=1$ on $D_i$ and
$V_j^i\subseteq D_i$.  Hence
\[
    \sum_{j=1}^{k_i}
    \sum_{r(\eta)=s_Z(z)}
        \chi_{V_j^i}(z\cdot\eta)
    \leq
    k_i\sum_{r(\eta)=s_Z(z)}
        \chi_i(z\cdot\eta),
\]
and the expression on the right is bounded by Proposition~\ref{prop:fullsys}.

By \cite[Lemma~2.14]{MRW}, choose
$\phi_j^i\in C_c^+(O_i)$, with
$\mathrm{supp}(\phi_j^i)\subseteq V_j^i$, such that
\[
    \left|
      \psi_j^i(z)
      -
      \phi_j^i(z)
      \sum_{r(\gamma)=r_Z(z)}
          \phi_j^i(\gamma^{-1}\cdot z)
    \right|
    <
    \frac{\delta}{M_i}.
\]
For $z\in Z$, put
\[
    P_i(z)
    =
    \sum_{r(\gamma)=r_Z(z)}
    \sum_{r(\eta)=s_Z(z)}
    \sum_{j=1}^{k_i}
      \phi_j^i(z\cdot\eta)
      \phi_j^i(\gamma^{-1}\cdot z\cdot\eta).
\]
Summing the preceding estimate over $j$ and $\eta$ gives
\[
    \left|
      P_i(z)
      -
      \sum_{j=1}^{k_i}d_j^i(r_Z(z))
    \right|
    \leq\delta.
\]
Consequently,
\[
    0\leq P_i(z)<1+\delta
    \qquad (z\in Z),
\]
and
\[
    |P_i(z)-1|<\delta
    \qquad
    \bigl(r_Z(z)\in r_Z(C_i)\bigr).
\]

Now define
\[
    e
    =
    \sum_{i=1}^n\sum_{j=1}^{k_i}
    {}_G\!\left\langle
       \phi_j^i\theta_i,
       \widetilde{\phi_j^i\theta_i}
    \right\rangle.
\]
Clearly $e\geq 0$.  Moreover, $e$ vanishes outside $W$.  Indeed, a
nonzero summand at $\gamma$ forces, for some $i,j,\eta$,
\[
    z\cdot\eta,\;
    \gamma^{-1}\cdot z\cdot\eta
    \in V_j^i.
\]
Taking $w=\gamma^{-1}\cdot z\cdot\eta$, we have
$(\gamma\cdot w,w)\in V_j^i\times V_j^i$, and hence
$\gamma\in W$.  Since $W\subseteq U$, it follows that
\[
    e(\gamma)=0\qquad(\gamma\notin U).
\]

For $u\in G^{(0)}$, let $\rho_i(u)$ denote the contribution of the
$i$th block to the row sum:
\[
    \rho_i(u)
    =
    \sum_{r(\gamma)=u}
    \sum_{r(\eta)=s_Z(z)}
    \sum_{j=1}^{k_i}
      \phi_j^i(z\cdot\eta)
      \phi_j^i(\gamma^{-1}\cdot z\cdot\eta)
      \theta_i(z\cdot\eta)
      \theta_i(\gamma^{-1}\cdot z\cdot\eta),
\]
where $z$ is any point with $r_Z(z)=u$.  Thus
\[
    \sum_{r(\gamma)=u}e(\gamma)
    =
    \sum_{i=1}^n\rho_i(u).
\]

We claim that
\[
    |\rho_i(u)-h_i(u)|<3\delta
    \qquad (u\in K).
\]
Suppose first that
$u\in K\cap\mathrm{supp}(h_i)=r_Z(C_i)$, and choose
$z\in C_i$ with $r_Z(z)=u$.  Write
\[
    a_{\gamma,\eta,j}
    =
    \phi_j^i(z\cdot\eta)
    \phi_j^i(\gamma^{-1}\cdot z\cdot\eta).
\]
For every nonzero term,
$z\cdot\eta\in V_j^i\subseteq D_i$, so
\[
    \theta_i(z\cdot\eta)^2=h_i(u).
\]
Moreover,
$z\cdot\eta,\gamma^{-1}\cdot z\cdot\eta\in V_j^i$ implies
$\gamma\in W\subseteq W_i$, and hence
\[
    \left|
      \theta_i(\gamma^{-1}\cdot z\cdot\eta)
      -
      \theta_i(z\cdot\eta)
    \right|
    <\delta.
\]
Since $\sum a_{\gamma,\eta,j}=P_i(z)$ and $0\leq\theta_i\leq1$,
we obtain
\[
\begin{aligned}
    |\rho_i(u)-h_i(u)|
    &\leq
      \sum a_{\gamma,\eta,j}\,
      \theta_i(z\cdot\eta)
      \left|
        \theta_i(\gamma^{-1}\cdot z\cdot\eta)
        -
        \theta_i(z\cdot\eta)
      \right|                                      \\
    &\quad+
      h_i(u)|P_i(z)-1|                              \\
    &<
      \delta(1+\delta)+\delta
      <3\delta.
\end{aligned}
\]

If instead $u\in K\setminus\mathrm{supp}(h_i)$, choose any
$z\in Z$ with $r_Z(z)=u$.  Every potentially nonzero summand in
$\rho_i(u)$ has
$z\cdot\eta\in V_j^i\subseteq D_i$, and therefore
\[
    \theta_i(z\cdot\eta)^2=h_i(u)=0.
\]
Thus $\rho_i(u)=0=h_i(u)$, proving the claim.

Finally, for $u\in K$,
\[
\begin{aligned}
    \left|
      \sum_{r(\gamma)=u}e(\gamma)-1
    \right|
    &=
      \left|
        \sum_{i=1}^n\bigl(\rho_i(u)-h_i(u)\bigr)
      \right|                                      \\
    &<
      3n\delta
      \leq\varepsilon.
\end{aligned}
\]
Thus $e=e(K,U,\varepsilon)$ satisfies all the hypotheses of
Proposition~\ref{prop:approxID1}.  The resulting net is therefore an approximate
identity in the inductive limit topology for the left action of
$\mathcal{C}(G)$ on $\mathcal{C}(Z)$.

It remains to consider $\mathcal{C}(Z^{op})$.  Each summand of $e$ has the
form
\[
    e_\xi={}_G\langle\xi,\widetilde{\xi}\rangle,
    \qquad
    \xi\in \mathcal{C}(Z),
\]
with $\xi$ nonnegative.  Formula~\eqref{cZ_cZop-cG}, using
$\gamma^{-1}\cdot z$ as the base point when evaluating
$e_\xi(\gamma^{-1})$, gives
\[
    e_\xi(\gamma^{-1})=e_\xi(\gamma).
\]
Hence $e(\gamma^{-1})=e(\gamma)$ for every $\gamma\in G$.
Using the right action formula~\eqref{cH_Zop_cG}, it follows that
\[
    \widetilde{\psi}\cdot e
    =
    \widetilde{e\cdot\psi}
    \qquad(\psi\in \mathcal{C}(Z)).
\]
Since $e\cdot\psi\to\psi$ in the inductive limit topology, we obtain
$\widetilde{\psi}\cdot e\to\widetilde{\psi}$.  Thus the same net is
an approximate identity for the right action of $\mathcal{C}(G)$ on
$\mathcal{C}(Z^{op})$.
\end{proof}

Note that we may apply Proposition~\ref{prop:ai} to the case $Z=H=G$ since $G$ is a $(G,G)$-equivalence, getting an approximate identity for the convolution product on $\mathcal{C}(G)$.

\begin{cor} \label{approxid}
In the context of Proposition~\ref{prop:ai}, the net $\{e_\lambda\}_{\lambda\in\Lambda}$ is an approximate identity with respect to the $I$-norm for the left action of $\mathcal{C}(G)$ on $\mathcal{C}(Z)$ and for the right action of $\mathcal{C}(G)$ on $\mathcal{C}(Z^{op})$, and thus also in the norms $\Vert\cdot\Vert_{I,r}$, $\Vert\cdot\Vert_{I,s}$, $\Vert\cdot\Vert_{p,red}$, and $\Vert\cdot\Vert_{L^p}$ for $1\leq p\leq\infty$, as well as $\Vert\cdot\Vert_{P,red}$ and $\Vert\cdot\Vert_{P}$ for nonempty $P\subseteq[1,\infty]$. 
\end{cor}

\begin{proof}
Convergence in the inductive limit topology implies convergence in the $I$-norm \cite[Proposition 2.2.2]{Pat}. 
The other norms are dominated by the $I$-norm.
\end{proof}

\begin{thm} \label{thm:main}
If $G$ and $H$ are locally compact, locally Hausdorff, \'{e}tale groupoids that are equivalent and have paracompact unit spaces, then for each $p\in[1,\infty]$, $F^p_{red}(G)$ and $F^p_{red}(H)$ are Morita equivalent.
\end{thm}

\begin{proof}
Let $\X_Z$ and $\Y_Z$ be the completions of
$\mathcal C(Z^{op})$ and $\mathcal C(Z)$, respectively, in the norms
of Lemma~\ref{preMor}.  By that lemma and
Remark~\ref{Rem-CompletionPM}, it remains to verify fullness and
nondegeneracy.

Let $(e_\lambda)$ be the approximate identity from
Corollary~\ref{approxid}.  Since each $e_\lambda$ is a finite sum of
pairings, for $a\in\mathcal C(G)$ we have
\[
e_\lambda a
 =
\sum_i {}_G\langle
\phi_i^\lambda,\widetilde{\phi_i^\lambda}\cdot a\rangle
 \longrightarrow a.
\]
Thus the $G$-valued pairing is full.  The $H$-valued pairing is full
by symmetry.  The convergence
\[
e_\lambda\cdot\phi\longrightarrow\phi,
\qquad \phi\in\mathcal C(Z),
\]
and its symmetric counterpart give nondegeneracy.
\end{proof}

\begin{cor}
Let $G$ be a second-countable, locally compact Hausdorff, \'{e}tale groupoid.
If $G$ is transitive, and $H=\{\gamma\in G:s(\gamma)=u=r(\gamma)\}$ is any isotropy group, then $F^p_{red}(G)$ and $F^p_{red}(H)$ are Morita equivalent.
\end{cor}

\begin{proof}
Such a groupoid $G$ is equivalent to any of its isotropy groups \cite[Example 2.32]{Will}.
\end{proof}

\begin{rem} \label{rem:L^P}
Lemma~\ref{preMor} remains valid with $p$ replaced by any nonempty
$P\subseteq[1,\infty]$.  Hence $F^P_{\mathrm{red}}(G)$ and
$F^P_{\mathrm{red}}(H)$ are Morita equivalent.  These equivalences
are compatible with inclusions of exponent sets as follows.
\end{rem}

\begin{prop} \label{prop:commKT}
If $\emptyset\neq P_1\subseteq P_2\subseteq [1,\infty]$, then there is a contractive homomorphism \[ \iota_G \colon F^{P_2}_{red}(G)\to F^{P_1}_{red}(G) \] that extends the identity map on $\mathcal{C}(G)$, and similarly for $H$. Further, we then have the following commutative diagram on $K$-theory, where the horizontal isomorphisms are induced by the respective Morita equivalences:
\[
\begin{CD}
K_*(F^{P_2}_{red}(G))	@>\cong>>	K_*(F^{P_2}_{red}(H)) \\
@V(\iota_G)_*VV  @VV(\iota_H)_*V \\
K_*(F^{P_1}_{red}(G))	@>\cong>>	K_*(F^{P_1}_{red}(H))
\end{CD}
\]
\end{prop}
\begin{proof}
Since $P_1\subseteq P_2\subseteq [1,\infty]$, it is clear that we have a contractive homomorphism \[ \iota_G \colon F^{P_2}_{red}(G)\to F^{P_1}_{red}(G) \] that extends the identity map on $\mathcal{C}(G)$, and similarly for $H$.
Now let $(\X_Z,\Y_Z)^{P_i}$ be the respective Morita equivalences for $i=1,2$.
The identity maps on $\mathcal{C}(Z)$ and $\mathcal{C}(Z^{op})$ extend continuously to give a concurrent homomorphism from $(\X_Z,\Y_Z)^{P_2}$ to $(\X_Z,\Y_Z)^{P_1}$ with coefficient maps $\iota_G$ and $\iota_H$ in the sense of \cite[Definition 1.13]{Par15}.
The six intertwining identities in \cite[Definition 1.13]{Par15} hold on the dense subspaces $\mathcal{C}(Z)$, $\mathcal{C}(Z^{op})$, $\mathcal{C}(G)$, and $\mathcal{C}(H)$ by Lemma~\ref{preMor}. The module actions and pairings are jointly continuous, and $\iota_G$, $\iota_H$ together with the continuous extensions of the identity maps are continuous, so the identities pass to the completions. Also, Banach algebra $K$-theory is homotopy invariant and Morita invariant in the sense of \cite[Definition 1.15]{Par15} (see \cite[Theorem 1.2]{Par15}), which is the standing hypothesis of \cite[Lemma 1.19]{Par15}; our algebras are nondegenerate by Corollary~\ref{approxid}, so this applies. 
Thus, by \cite[Lemma 1.19]{Par15}, we have the desired commutative diagram on $K$-theory. 
\end{proof}

\section{Morita equivalence for the full $L^p$-operator algebras} \label{sec:full-Lp}

The main result of this section is a full-clopen reduction theorem for
the full $L^p$-operator algebra.  For $1<p<\infty$, it is proved by
dilating a spatial representation of the reduction; the cases
$p=1,\infty$ follow from the source- and range-$I$-norm formulas.
Corollary~\ref{cor:full-morita} then gives the full Morita equivalence.

\subsection{Representation-theoretic input for the full norm}

We use the notation and conventions of Section~\ref{sect_Prelim}.  Recall that, in the non-Hausdorff setting, a set is \emph{relatively compact} if it is contained in a compact subset; compact closure is not required.

For an \'{e}tale groupoid $K$, write $\Bis(K)$ for the inverse semigroup of open bisections of $K$; see \cite[Example~2.19]{BKM1}.  If $U\subseteq K^{(0)}$, write
\[
K|_U=r^{-1}(U)\cap s^{-1}(U)
\]
for the reduction.  We call $U$ \emph{full} if every $K$-orbit in $K^{(0)}$ meets $U$, equivalently
\[
r\bigl(s^{-1}(U)\bigr)=K^{(0)}.
\]

We shall use repeatedly the $I$-norm and the full norm from Section~\ref{sect_Prelim}.  In the complex untwisted case, \cite[Theorem~5.13(2),(4)]{BKM1} gives
\[
\|f\|_{F^p(K)}
=
\sup_{\rho\in\mathcal{R}}\bigl\{\|\rho(f)\|\bigr\},
\]
where $\mathcal{R}$ consists of all $\|\cdot\|_I$-contractive homomorphisms $\rho:\cC(K)\to B(L^p(\mu))$,
and hence
\[
\|f\|_{F^p(K)}\leq \|f\|_I.
\]
For $1<p<\infty$, in the second-countable case this agrees with Gardella--Lupini \cite[Definition~6.4]{GL}; see \cite[Remark~5.14]{BKM1}.

\begin{lem}[Bisection decomposition]\label{lem:bisection-decomposition}
Every $f\in\cC(K)$ is a finite sum $f=\sum_{m=1}^n f_m$ such that, for each $m$, there is a relatively compact open bisection $A_m\subseteq K$ with $f_m\in C_c(A_m)$.
\end{lem}

\begin{proof}
The relatively compact open bisections cover $K$. Indeed, let $\gamma\in K$ lie in an open bisection $A$. Since $K^{(0)}$ is locally compact Hausdorff, it is regular. Hence, because $r(A)$ is an open neighborhood of $r(\gamma)$, there is an open neighborhood $V$ of $r(\gamma)$ such that $\overline V$ is compact and $\overline V\subseteq r(A)$. As $r|_A:A\to r(A)$ is a homeomorphism, $(r|_A)^{-1}(V)$ is contained in the compact set $(r|_A)^{-1}(\overline V)$; thus it is a relatively compact open bisection containing $\gamma$. Since $K^{(0)}$ is Hausdorff, every open bisection is Hausdorff. The assertion now follows from Exel's bisection-cover decomposition \cite[Proposition~3.10]{Exel}.
\end{proof}

\subsubsection{Spatial covariant representations}

Let $S\subseteq\Bis(K)$ be a unital inverse semigroup of open bisections covering $K$.  For $A\in S$, put
\[
h_A:s(A)\longrightarrow r(A),
\qquad
h_A=r\circ(s|_A)^{-1}.
\]
The associated partial automorphism is
\[
\alpha_A:C_0(s(A))\longrightarrow C_0(r(A)),
\qquad
\alpha_A(a)=a\circ h_A^{-1},
\]
where these $C_0$-spaces are regarded as ideals of $C_0(K^{(0)})$ by extension by zero.

Following \cite[Definition~3.17]{BKM1}, a covariant representation $(\pi,v)$ on a Banach space $E$ consists of a representation
\[
\pi:C_0(K^{(0)})\to B(E)
\]
and a map $A\mapsto v_A$ into the contractive operators such that $v_{A^{-1}}$ is a contractive generalized inverse of $v_A$ and
\begin{enumerate}[label=(SCR\arabic*)]
\item $v_A\pi(a)v_{A^{-1}}=\pi(a\circ h_A^{-1})$ for $a\in C_0(s(A))$;
\item the initial and final spaces of $v_A$ are the closed subspaces generated by $\pi(C_0(s(A)))E$ and $\pi(C_0(r(A)))E$, respectively;
\item $v_Av_B=v_{AB}$.
\end{enumerate}
Bardadyn-Kwa\'{s}niewski-McKee formulate (SCR3) for a twisted action.  For the canonical untwisted action the twisting multiplier is the identity on the relevant range ideal, so their relation reduces to the displayed equality $v_Av_B=v_{AB}$.

A covariant representation on $L^p(\mu)$ is \emph{spatial} if $\pi$ acts by multiplication operators and $A\mapsto v_A$ takes values in the inverse semigroup $\SPIso(L^p(\mu))$ of spatial partial isometries; see \cite[Definition~5.15]{BKM1}.  Recall that $S$ is \emph{wide} if it covers $K$ and, for every $A,B\in S$, the intersection $A\cap B$ is a union of bisections belonging to $S$; see the discussion in \cite[Example~2.20]{BKM1}.  In this case the canonical action models $K$ in the sense of \cite[Definition~4.20]{BKM1}: by \cite[Lemma~4.21]{BKM1}, the germ groupoid is canonically isomorphic to $K$ through
\[
S\ltimes K^{(0)}\longrightarrow K,
\qquad
[A,x]\longmapsto (s|_A)^{-1}(x).
\]
In particular, $\Bis(K)$ is wide: it covers $K$ by \'{e}taleness, and the intersection of two open bisections is again an open bisection.

\begin{prop} (\cite[Theorem~5.19(1)]{BKM1}) \label{prop:BKM-isometric}
Let $p\in[1,\infty]$ and suppose that the canonical action of $S$ models $K$. There is an isometric representation
\[
\pi\rtimes v:F^p(K)\longrightarrow B(L^p(\mu))
\]
with $(\pi,v)$ spatial, with the localizable measure space convention of \cite[Section~5]{BKM1}. If $p<\infty$, it may be chosen nondegenerate.
\end{prop}

\subsubsection{Borel extension}

We use the Borel extension theorem from \cite[Section~4.5]{BKM1}.  For an open bisection $A$, let $\Bop(A)$ denote the bounded Borel functions on $A$, extended by zero outside $A$, and set
\[
\Bop^{\Bis(K)}(K)
\coloneqq
\Span\{\Bop(A):A\in\Bis(K)\}.
\]
This is an algebra under convolution; see the discussion preceding \cite[Proposition~4.29]{BKM1}.  No compact-support condition is imposed.

\begin{prop}[Spatial Borel extension]\label{prop:Borel-extension}
Let $1<p<\infty$.  Let $(\pi,v)$ be a nondegenerate spatial covariant representation on $E=L^p(\mu)$, with $\mu$ localizable, of the canonical inverse semigroup action of $\Bis(K)$ on $C_0(K^{(0)})$, and let
\[
\psi=\pi\rtimes v:\cC(K)\longrightarrow B(E)
\]
be its integrated form.  Then $\psi$ is $\|\cdot\|_I$-contractive and extends to a homomorphism
\[
\widetilde\psi:\Bop^{\Bis(K)}(K)\longrightarrow B(E).
\]
Its restriction
\[
\widetilde\pi \coloneqq \widetilde\psi|_{\Bop(K^{(0)})}
\]
to bounded Borel functions on the unit space is a multiplication representation.  More precisely, there is a unital $*$-homomorphism
\[
\widetilde\pi_0:\Bop(K^{(0)})\longrightarrow L^\infty(\mu)
\]
such that $\widetilde\pi(b)=M_{\widetilde\pi_0(b)}$.

If $A\in\Bis(K)$ and $b$ is a bounded Borel function on $r(A)$, write $b\delta_A$ for the Borel function on $A$ given by $(b\delta_A)(\gamma)=b(r(\gamma))$.  Then
\begin{equation}\label{eq:Borel-slice}
\widetilde\psi(b\delta_A)=\widetilde\pi(b)v_A.
\end{equation}
For every Borel set $D\subseteq K^{(0)}$ there is, modulo null sets, a measurable subset $Y_D$ of the underlying measure space such that
\[
\widetilde\pi(1_D)=M_{1_{Y_D}}.
\]
Thus $\widetilde\pi(1_D)E$ is a measurable band of $E$; the sets $D$ and $Y_D$ belong to different measure spaces and will not be identified.
\end{prop}

\begin{proof}
By \cite[Remark~5.7]{BKM1}, the integrated form is $\|\cdot\|_I$-contractive.  Hence for $f\in\cC(K)$,
\[
\|\psi(f)\|\leq \|f\|_I\leq \|f\|_{\max}^{\Bis(K)},
\]
where the second inequality is \cite[Remark~4.7]{BKM1}.  This is exactly the contractivity hypothesis in \cite[Proposition~4.30]{BKM1}.  Since $1<p<\infty$, $E$ is reflexive and has a predual, so that proposition gives the Borel extension; \cite[Proposition~4.29 and Theorem~4.31]{BKM1} give its Borel covariant description.

Let $(e_\lambda)\subseteq C_c(r(A))$ be a positive contractive approximate identity of $C_0(r(A))$.  On continuous compactly supported coefficients, the Borel extension agrees with the original integrated form, and \cite[Proposition~4.23]{BKM1} gives
\[
\widetilde\psi(e_\lambda\delta_A)
=
\psi(e_\lambda\delta_A)
=
\pi(e_\lambda)v_A.
\]
By (SCR2), the last term converges strongly to $v_A$.  The Radon measure formula in \cite[Proposition~4.30]{BKM1} identifies the weak limit of the first term with $\widetilde\psi(1_A)$, hence $\widetilde\psi(1_A)=v_A$.  Proposition~4.29 of \cite{BKM1} then gives \eqref{eq:Borel-slice}.

We next consider the restriction to bounded Borel functions on $K^{(0)}$. By the construction in
\cite[Proposition~4.30]{BKM1},
\[
\|\widetilde\pi(b)\|\leq\|b\|_\infty
\qquad (b\in\Bop(K^{(0)})).
\]
Since $\pi$ is nondegenerate, the same
approximate identity argument on the unit space part gives
\[
\widetilde\pi(1)=I_E.
\]
Thus
\[
\widetilde\pi:\Bop(K^{(0)})\longrightarrow B(L^p(\mu))
\]
is a unital, hence nondegenerate, representation of the commutative
$C^*$-algebra $\Bop(K^{(0)})$.

Suppose first that $p\neq2$.  By the Gelfand representation theorem,
$\Bop(K^{(0)})$ is isomorphic to $C(X)$ for some compact Hausdorff
space $X$.  Hence \cite[Theorem~2.13]{BKM1}, applied to
$\widetilde\pi$, shows that $\widetilde\pi$ is a multiplication
representation on $L^p(\mu)$.

It remains to consider $p=2$.  Let $P=M_{1_Y}$ be an arbitrary
multiplication projection on $E=L^2(\mu)$.  For
$\xi\in E'$ and $\eta\in E$, let $\lambda_{\xi,\eta}$ be the Radon
measure on $K^{(0)}$ representing
\[
a\longmapsto \langle\xi,\pi(a)\eta\rangle,
\qquad a\in C_c(K^{(0)}).
\]
Since every $\pi(a)$ is a multiplication operator,
$P\pi(a)=\pi(a)P$, and therefore
\[
\lambda_{P'\xi,\eta}=\lambda_{\xi,P\eta}.
\]
For bounded Borel $b$, \cite[Proposition~4.30]{BKM1} now gives
\[
\begin{aligned}
\langle\xi,P\widetilde\pi(b)\eta\rangle
&=\int b\,d\lambda_{P'\xi,\eta}
 =\int b\,d\lambda_{\xi,P\eta} \\
&=\langle\xi,\widetilde\pi(b)P\eta\rangle.
\end{aligned}
\]
Thus $\widetilde\pi(b)$ commutes with every multiplication projection.
It therefore commutes with the multiplication algebra
$L^\infty(\mu)\subseteq \mathcal B(L^2(\mu))$.  Since this algebra is maximal
abelian, $\widetilde\pi(b)$ is itself a multiplication operator.
Consequently $\widetilde\pi$ is a multiplication representation also
when $p=2$.

We have therefore shown, for every $1<p<\infty$, that there is a
unital homomorphism
\[
\widetilde\pi_0:\Bop(K^{(0)})\longrightarrow L^\infty(\mu)
\]
such that
\[
\widetilde\pi(b)=M_{\widetilde\pi_0(b)}
\qquad
(b\in\Bop(K^{(0)})).
\]
Since $\widetilde\pi$ is contractive, so is $\widetilde\pi_0$.
A unital contractive homomorphism between commutative $C^*$-algebras
is a $*$-homomorphism.  Finally, for every Borel set
$D\subseteq K^{(0)}$, the function $\widetilde\pi_0(1_D)$ is an
idempotent in $L^\infty(\mu)$, and hence is the characteristic
function of a measurable set modulo null sets.  This gives the last
assertion.
\end{proof}

\begin{cor}[Borel restriction of a spatial bisection]\label{cor:Borel-restriction}
In the setting of Proposition~\ref{prop:Borel-extension}, let $A\in\Bis(K)$ and let $D\subseteq A$ be Borel.  Then
\[
W_D\coloneqq\widetilde\psi(1_D)
\]
is the restriction of $v_A$ from the measurable source band $\widetilde\pi(1_{s(D)})E$ onto the measurable range band $\widetilde\pi(1_{r(D)})E$.  More precisely,
\[
W_D=\widetilde\pi(1_{r(D)})v_A
=v_A\widetilde\pi(1_{s(D)}),
\]
its generalized inverse is $W_{D^{-1}}$, and
\[
W_DW_{D^{-1}}=\widetilde\pi(1_{r(D)}),
\qquad
W_{D^{-1}}W_D=\widetilde\pi(1_{s(D)}).
\]
\end{cor}

\begin{proof}
Since $A$ is a bisection, $1_D=1_{r(D)}\delta_A$, so the first formula follows from \eqref{eq:Borel-slice}; Borel covariance gives the second.  The remaining assertions follow by restricting the spatial partial isometry $v_A$ to the displayed source and range bands and applying the same argument to $D^{-1}$.
\end{proof}

\subsection{A locally finite coordinate system}

Let $K$ be a locally compact, locally Hausdorff, \'{e}tale groupoid with paracompact locally compact Hausdorff unit space, let $U\subseteq K^{(0)}$ be full and clopen, and put
\[
G=K|_U.
\]
Write
\[
j_U:\cC(G)\longrightarrow\cC(K)
\]
for extension by zero. Since $G$ is a clopen subgroupoid, $j_U$ is an algebra homomorphism, and it is isometric for each of $\|\cdot\|_{I,r}$, $\|\cdot\|_{I,s}$, and $\|\cdot\|_I$.

\begin{lem}[Locally finite bisection coordinates]\label{lem:coordinates}
There exist an index set $I$, relatively compact open sets $V_i\subseteq K^{(0)}\setminus U$, and relatively compact open bisections $B_i\subseteq K$ such that
\begin{enumerate}[label=(\roman*)]
\item $(V_i)_{i\in I}$ is a locally finite cover of $K^{(0)}\setminus U$;
\item $r(B_i)=V_i$ and $s(B_i)\subseteq U$ for every $i$;
\item after choosing a well-order on $I$ and setting
\[
R_i=V_i\setminus\bigcup_{j<i}V_j,
\]
the sets $R_i$ form a pairwise disjoint Borel partition of $K^{(0)}\setminus U$, and the family $(R_i)_{i\in I}$ is locally finite;
\item if
\[
\widetilde B_i=B_i\cap r^{-1}(R_i),
\qquad
T_i=s(\widetilde B_i),
\]
then $\widetilde B_i$ is a Borel bisection from $T_i\subseteq U$ onto $R_i$, and $T_i$ is Borel.
\end{enumerate}
We also set $V_0=R_0=U$, $B_0=\widetilde B_0=U$, and $T_0=U$.
\end{lem}

\begin{proof}
Fix $x\in K^{(0)}\setminus U$.  Since $U$ is full, there is $\gamma\in K$ with
\[
r(\gamma)=x,
\qquad
s(\gamma)\in U.
\]
Choose an open bisection neighborhood $C_x$ of $\gamma$.  Replacing it by
\[
C_x\cap s^{-1}(U)\cap r^{-1}(K^{(0)}\setminus U),
\]
which is still an open bisection containing $\gamma$, we may assume
\[
s(C_x)\subseteq U,
\qquad
r(C_x)\subseteq K^{(0)}\setminus U.
\]
Choose a relatively compact open neighborhood $W_x$ of $x$ such that $\overline{W_x}\subseteq r(C_x)$.  The sets $W_x$ cover the paracompact locally compact Hausdorff space $K^{(0)}\setminus U$.  By paracompactness and the shrinking lemma, there is a locally finite open refinement $(V_i)_{i\in I}$ and, for each $i$, an index $x(i)$ such that
\[
\overline{V_i}\subseteq W_{x(i)}.
\]
Define
\[
B_i \coloneqq (r|_{C_{x(i)}})^{-1}(V_i)
=C_{x(i)}\cap r^{-1}(V_i).
\]
Then $r(B_i)=V_i$ and $s(B_i)\subseteq U$.  Moreover,
\[
B_i\subseteq(r|_{C_{x(i)}})^{-1}(\overline{V_i}),
\]
and the set on the right is compact; hence $B_i$ is relatively compact.  This proves (i) and (ii).

Choose a well-order on $I$ and define $R_i$ as in (iii).  Since
\[
R_i
=
V_i\cap\left(K^{(0)}\setminus\bigcup_{j<i}V_j\right),
\]
each $R_i$ is Borel.  If $i<k$, then the definition of $R_k$ removes $V_i$, while $R_i\subseteq V_i$; hence $R_i\cap R_k=\varnothing$.  Since every pair of indices is comparable, the $R_i$ are pairwise disjoint.  They cover the same set as the $V_i$, and local finiteness follows from $R_i\subseteq V_i$.

Finally, $r|_{B_i}:B_i\to V_i$ is a homeomorphism, so
\[
\widetilde B_i=(r|_{B_i})^{-1}(R_i)
\]
is Borel.  Moreover,
\[
T_i=s(\widetilde B_i)
\]
is Borel because $s\circ(r|_{B_i})^{-1}:V_i\to s(B_i)$ is a homeomorphism.  The restrictions of $r$ and $s$ are one-to-one, so $\widetilde B_i$ is a Borel bisection from $T_i$ onto $R_i$.
\end{proof}

Let $\Bisrc(K)$ denote the relatively compact open bisections of $K$.

\begin{lem}\label{lem:SK}
The set
\[
S_K=\Bisrc(K)\cup\{K^{(0)}\}
\]
is a unital wide inverse subsemigroup of $\Bis(K)$ and covers $K$.
\end{lem}

\begin{proof}
Inversion preserves relative compactness.  If $A,C\in\Bisrc(K)$, choose compact subsets $M,N\subseteq K$ containing them.  Since $K^{(0)}$ is Hausdorff, its diagonal is closed and therefore
\[
K^{(2)}=(s\times r)^{-1}(\Delta_{K^{(0)}})
\]
is closed in $K\times K$.  Hence
\[
L_{M,N} \coloneqq (M\times N)\cap K^{(2)}
\]
is compact.  Multiplication $m$ is continuous on $K^{(2)}$, so $m(L_{M,N})$ is compact and contains $AC$.  Thus $AC$ is relatively compact.  Together with $K^{(0)}$, this proves that $S_K$ is a unital inverse subsemigroup.

It covers $K$ by the shrinking argument in Lemma~\ref{lem:bisection-decomposition}.  It is wide as well.  If $A,C\in S_K$ and at least one is relatively compact, then $A\cap C$ is an open sub-bisection contained in a compact subset and hence belongs to $\Bisrc(K)$; the remaining case is $A=C=K^{(0)}$.  Thus $A\cap C$ itself belongs to $S_K$, which is stronger than the requirement that it be a union of bisections belonging to $S_K$.
\end{proof}

\subsection{Spatial matrix dilation and the full-clopen reduction theorem}

Fix $1<p<\infty$ and put $G=K|_U$.  The canonical action of $\Bis(G)$ on $C_0(G^{(0)})$ models $G$ by \cite[Lemma~4.21]{BKM1}.  Applying Proposition~\ref{prop:BKM-isometric}, choose an isometric nondegenerate spatial representation
\begin{equation}\label{eq:psi}
\psi=\pi\rtimes v \colon F^p(G)\longrightarrow B(E),
\qquad
E=L^p(\mu).
\end{equation}
By Proposition~\ref{prop:Borel-extension}, its restriction to $\cC(G)$ has a Borel extension
\[
\widetilde\psi \colon \Bop^{\Bis(G)}(G)\longrightarrow B(E),
\qquad
\widetilde\pi \coloneqq \widetilde\psi|_{\Bop(G^{(0)})},
\]
where $\widetilde\pi$ is a multiplication representation.

Let $R_i,\widetilde B_i,T_i$ be the coordinate system from Lemma~\ref{lem:coordinates}.  Put
\[
P_i=\widetilde\pi(1_{T_i}),
\qquad
E_i=P_iE.
\]
By Proposition~\ref{prop:Borel-extension}, there is a measurable subset $Y_i$ of the underlying measure space such that $P_i=M_{1_{Y_i}}$.  Hence restriction and extension by zero identify
\[
E_i\cong L^p(Y_i,\mu|_{Y_i}).
\]
Here $T_i\subseteq U$ is a subset of the groupoid unit space, whereas $Y_i$ belongs to the measure space representing $E$.

Define
\[
\widehat E=\bigoplus_{i\in I\cup\{0\}}^{p}E_i.
\]
Since $T_0=U$ and $\pi$ is nondegenerate, $P_0=I_E$ and hence $E_0=E$.  Taking the disjoint union of the measure spaces representing the bands $E_i$ identifies $\widehat E$ isometrically with an $L^p$-space.  Measurable subspaces and arbitrary direct sums preserve localizability; see \cite[214K(a)--(b), 214L]{Fremlin}.  Thus the resulting measure may again be taken localizable.

For each $i$, the Borel bisection $\widetilde B_i$ gives a Borel bijection
\[
\theta_i:T_i\longrightarrow R_i,
\qquad
\theta_i(u)=r\bigl((s|_{\widetilde B_i})^{-1}(u)\bigr),
\]
with $\theta_0=\mathrm{id}_U$.

\subsubsection{The transported diagonal}

The $i$th summand $E_i$ is a band associated with $T_i\subseteq U$, but it will serve as a coordinate copy of $R_i\subseteq K^{(0)}$ through $\theta_i:T_i\to R_i$.  Accordingly, a bounded Borel function on $K^{(0)}$ acts on $E_i$ by restriction to $R_i$, pullback along $\theta_i$, and then the Borel diagonal representation $\widetilde\pi$.

For a bounded Borel function $b$ on $K^{(0)}$, define
\[
\widehat{\widetilde\pi}(b)|_{E_i}
=
\widetilde\pi\bigl(1_{T_i}((b|_{R_i})\circ\theta_i)\bigr)|_{E_i}.
\]
For fixed $i$, all operations in this formula are $*$-homomorphisms.  Moreover,
\[
\bigl\|1_{T_i}((b|_{R_i})\circ\theta_i)\bigr\|_\infty\leq\|b\|_\infty.
\]
Although extension by zero sends the constant function $1$ to $1_{T_i}$ rather than $1_U$, on $E_i=P_iE$ one has
\[
\widetilde\pi(1_{T_i})|_{E_i}=I_{E_i}.
\]
Each block in the definition above is a unital contractive $*$-representation, and the $\ell^p$-direct sum gives a unital contractive multiplication representation
\[
\widehat{\widetilde\pi}:\Bop(K^{(0)})\longrightarrow B(\widehat E).
\]
Set
\[
\widehat\pi \coloneqq \widehat{\widetilde\pi}|_{C_0(K^{(0)})}.
\]
Thus, for $a\in C_0(K^{(0)})$,
\[
\widehat\pi(a)|_{E_i}
=
\widetilde\pi\bigl(1_{T_i}((a|_{R_i})\circ\theta_i)\bigr)|_{E_i}.
\]

\begin{lem}[Support projections]\label{lem:support-projections}
For every open set $D\subseteq K^{(0)}$, put
\[
Q_D \coloneqq \widehat{\widetilde\pi}(1_D).
\]
Then
\[
Q_D\widehat E
=
\overline{\Span\{\widehat\pi(a)\xi:a\in C_0(D),\ \xi\in\widehat E\}}.
\]
In particular, $Q_{K^{(0)}}=I_{\widehat E}$ and $\widehat\pi$ is nondegenerate.
\end{lem}

\begin{proof}
Let
\[
M_D=\overline{\Span\{\widehat\pi(a)\xi:a\in C_0(D),\ \xi\in\widehat E\}}.
\]
If $a\in C_0(D)$, then $1_Da=a$, and multiplicativity gives
\[
Q_D\widehat\pi(a)
=\widehat{\widetilde\pi}(1_Da)
=\widehat\pi(a).
\]
Hence $M_D\subseteq Q_D\widehat E$.

For the reverse inclusion, fix $\xi\in Q_D\widehat E$.  Since $M_D$ is a norm-closed linear subspace of the Banach space $\widehat E$, it is weakly closed by Hahn--Banach.  It therefore suffices to put $\xi$ in the weak closure of $M_D$.  Let $\eta^{(1)},\ldots,\eta^{(n)}\in\widehat E'$ and $\varepsilon>0$.  For each $k$, the functional
\[
a\longmapsto\langle\eta^{(k)},\widehat\pi(a)\xi\rangle,
\qquad a\in C_0(K^{(0)}),
\]
is represented by a finite complex Radon measure $\nu_k$.

Write $q$ for the conjugate exponent to $p$, $\xi=(\xi_i)_i\in\bigoplus_i^pE_i$, and $\eta^{(k)}=(\eta_i^{(k)})_i\in\bigoplus_i^qE_i'$.  If $\Phi_i(b)$ denotes the $i$th block of $\widehat{\widetilde\pi}(b)$, then
\[
\begin{aligned}
\sum_i\bigl|\langle\eta_i^{(k)},\Phi_i(b)\xi_i\rangle\bigr|
&\leq \|b\|_\infty\sum_i\|\eta_i^{(k)}\|\,\|\xi_i\|\\
&\leq \|b\|_\infty\|\eta^{(k)}\|\,\|\xi\|.
\end{aligned}
\]
For the corresponding transported scalar measures, the $i$th total variation is at most $\|\eta_i^{(k)}\|\,\|\xi_i\|$.  Hence the same H\"older estimate shows that these measures sum in total variation.  By uniqueness of the representing measure, the blockwise formulas from \cite[Proposition~4.30]{BKM1} give
\[
\langle\eta^{(k)},\widehat{\widetilde\pi}(b)\xi\rangle
=\int b\,d\nu_k
\]
for every bounded Borel function $b$ on $K^{(0)}$.

If $b$ is supported in $K^{(0)}\setminus D$, then, since $Q_D\xi=\xi$,
\[
\widehat{\widetilde\pi}(b)\xi
=\widehat{\widetilde\pi}(b)Q_D\xi
=\widehat{\widetilde\pi}(b1_D)\xi
=0.
\]
Hence $\int b\,d\nu_k=0$ for every such $b$, so each $\nu_k$ is concentrated on $D$.

By inner regularity, and taking a finite union, choose a compact set $K_0\subseteq D$ such that
\[
|\nu_k|(D\setminus K_0)<\varepsilon
\qquad(k=1,\ldots,n).
\]
Choose $e\in C_c(D)$ with $0\leq e\leq1$ and $e=1$ on $K_0$.  Then $\widehat\pi(e)\xi\in M_D$ and
\[
\begin{aligned}
\bigl|\langle\eta^{(k)},\xi-\widehat\pi(e)\xi\rangle\bigr|
&=\left|\int_D(1-e)\,d\nu_k\right|\\
&\leq\int_{D\setminus K_0}|1-e|\,d|\nu_k|\\
&\leq|\nu_k|(D\setminus K_0)<\varepsilon.
\end{aligned}
\]
Thus every basic weak neighborhood of $\xi$ meets $M_D$, so $\xi\in M_D$.  This proves $Q_D\widehat E\subseteq M_D$.

Finally, unitality of $\widehat{\widetilde\pi}$ gives $Q_{K^{(0)}}=I_{\widehat E}$.  Taking $D=K^{(0)}$ in the displayed equality proves nondegeneracy of $\widehat\pi$.
\end{proof}

\subsubsection{The matrix associated with a bisection}

Set
\[
\widehat v_{K^{(0)}}=I_{\widehat E}.
\]
Every other element of $S_K$ is relatively compact.  Let $A\in S_K\setminus\{K^{(0)}\}$.  For $i,j\in I\cup\{0\}$, define
\[
D_{ij}(A)=\widetilde B_i^{-1}A\widetilde B_j.
\]
This is contained in the open $G$-bisection
\[
C_{ij}(A)=B_i^{-1}AB_j\subseteq G,
\]
and
\[
D_{ij}(A)
=
C_{ij}(A)\cap r^{-1}(T_i)\cap s^{-1}(T_j).
\]
Indeed, if $c=b_i^{-1}ab_j\in C_{ij}(A)$, then $r(c)=s(b_i)$ and $s(c)=s(b_j)$.  Since $s|_{B_i}$ and $s|_{B_j}$ are injective,
\[
r(c)\in T_i\iff b_i\in\widetilde B_i,
\qquad
s(c)\in T_j\iff b_j\in\widetilde B_j.
\]
Thus $D_{ij}(A)$ is a Borel bisection with source in $T_j$ and range in $T_i$.

Only finitely many $D_{ij}(A)$ are nonempty.  Indeed, nonemptiness implies
\[
R_i\cap r(A)\neq\varnothing,
\qquad
R_j\cap s(A)\neq\varnothing.
\]
Since $A$ is relatively compact, $r(A)$ and $s(A)$ are contained in compact subsets of $K^{(0)}$, and a compact set meets only finitely many members of the locally finite family $(R_i)$.  Thus only finitely many pairs $(i,j)$ occur.

Define
\[
W_{ij}(A)=\widetilde\psi(1_{D_{ij}(A)})
\]
and
\[
\widehat v_A=(W_{ij}(A))_{i,j}\in B(\widehat E).
\]
By Corollary~\ref{cor:Borel-restriction}, $W_{ij}(A)$ maps the initial band
\[
\widetilde\pi(1_{s(D_{ij}(A))})E\subseteq E_j
\]
onto the final band
\[
\widetilde\pi(1_{r(D_{ij}(A))})E\subseteq E_i.
\]
We regard $W_{ij}(A)$ as the corresponding operator
$E_j\to E_i$, which vanishes on the complement of its initial band.
The preceding finiteness observation therefore makes $\widehat v_A$ a finite operator matrix.

\begin{lem}[Spatiality and support]\label{lem:vhat-spatial}
For every $A\in S_K$, $\widehat v_A$ is a spatial partial isometry.  Its initial and final spaces are the closed linear spans of
\[
\widehat\pi(C_0(s(A)))\widehat E
\quad\text{and}\quad
\widehat\pi(C_0(r(A)))\widehat E,
\]
respectively, and its generalized inverse is $\widehat v_{A^{-1}}$.
\end{lem}

\begin{proof}
The case $A=K^{(0)}$ is immediate.  Otherwise $A$ is relatively compact.  By Corollary~\ref{cor:Borel-restriction}, each $W_{ij}(A)$ is a spatial restriction of the partial isometry associated with the open $G$-bisection $C_{ij}(A)$, with initial projection $\widetilde\pi(1_{s(D_{ij}(A))})$ and final projection $\widetilde\pi(1_{r(D_{ij}(A))})$.

Fix $j$.  If a point $u\in T_j$ belonged to both $s(D_{ij}(A))$ and $s(D_{i'j}(A))$, then the point $\theta_j(u)\in R_j$ would be the source of two arrows of the bisection $A$, whose ranges lie in $R_i$ and $R_{i'}$.  Source injectivity of $A$, followed by disjointness of the $R_i$, forces $i=i'$.  Hence these source sets are pairwise disjoint as $i$ varies.  The symmetric argument, using range injectivity of $A$, shows that for fixed $i$ the range sets $r(D_{ij}(A))$ are pairwise disjoint as $j$ varies.  Consequently the finitely many compatible restrictions in $\widehat v_A$ piece together to form a spatial partial isometry on $\widehat E$.

For fixed $j$, the sets $s(D_{ij}(A))$ partition precisely the pullback under $\theta_j$ of $R_j\cap s(A)$.  Thus the sum of their initial projections is the $j$th block of $Q_{s(A)}$.  Summing over $j$ gives initial projection $Q_{s(A)}$; symmetrically, the final projection is $Q_{r(A)}$.  By Lemma~\ref{lem:support-projections}, the ranges of these projections are the closed spans stated in this lemma.  Finally,
\[
D_{ji}(A^{-1})=D_{ij}(A)^{-1},
\]
so Corollary~\ref{cor:Borel-restriction} gives the generalized inverse $\widehat v_{A^{-1}}$.
\end{proof}

\subsubsection{Covariance and multiplication}

\begin{lem}[Covariance]\label{lem:SCR1}
For $A\in S_K$ and $a\in C_0(s(A))$,
\[
\widehat v_A\widehat\pi(a)\widehat v_{A^{-1}}
=
\widehat\pi(a\circ h_A^{-1}).
\]
\end{lem}

\begin{proof}
The case $A=K^{(0)}$ is immediate.  Suppose $A$ is relatively compact.  For each $i$, write
\[
\beta_i=(s|_{\widetilde B_i})^{-1}:T_i\longrightarrow\widetilde B_i.
\]
Take $u\in s(D_{ij}(A))$, put $b_j=\beta_j(u)$, and write the unique arrow of $D_{ij}(A)$ with source $u$ as
\[
d=b_i^{-1}\gamma b_j,
\qquad \gamma\in A,
\quad b_i\in\widetilde B_i.
\]
Then
\[
s(\gamma)=r(b_j)=\theta_j(u),
\qquad
r(\gamma)=h_A(\theta_j(u))=r(b_i)=\theta_i(s(b_i)).
\]
Since $r(d)=s(b_i)$, the partial map implemented by $D_{ij}(A)$ is therefore
\[
h_{D_{ij}(A)}
=
\theta_i^{-1}\circ h_A\circ\theta_j
\]
on its source.

On the $j$th block, $a$ acts by the pullback $a\circ\theta_j$.  By Borel covariance \cite[Proposition~4.29]{BKM1}, $W_{ij}(A)$ transports this multiplication operator to multiplication on the range piece by
\[
(a\circ\theta_j)\circ h_{D_{ij}(A)}^{-1}
=
a\circ h_A^{-1}\circ\theta_i,
\]
which is exactly the $i$th-block pullback of $a\circ h_A^{-1}$.  The relevant source and range pieces are pairwise disjoint and only finitely many occur, so summing the block identities gives the stated covariance relation.
\end{proof}

\begin{lem}[Multiplication]\label{lem:SCR3}
For all $A,C\in S_K$,
\[
\widehat v_A\widehat v_C=\widehat v_{AC}.
\]
\end{lem}

\begin{proof}
If one of $A,C$ equals $K^{(0)}$, the claim is immediate.  Assume that both are relatively compact.

Suppose first that $AC=K^{(0)}$.  Then $ACA=A$ and $CAC=C$, so $C$ is a generalized inverse of $A$ in the inverse semigroup $\Bis(K)$.  Hence $C=A^{-1}$, and $AA^{-1}=K^{(0)}$ gives $r(A)=K^{(0)}$.  By Lemma~\ref{lem:vhat-spatial}, $\widehat v_A\widehat v_{A^{-1}}$ is the final projection $Q_{r(A)}$; by Lemma~\ref{lem:support-projections}, $Q_{K^{(0)}}=I_{\widehat E}$.  Thus
\[
\widehat v_A\widehat v_C
=Q_{r(A)}
=Q_{K^{(0)}}
=I_{\widehat E}
=\widehat v_{AC}.
\]
We may therefore assume $AC\neq K^{(0)}$.

For every $i,k$ one has the disjoint decomposition
\[
D_{ik}(AC)
=
\bigsqcup_j D_{ij}(A)D_{jk}(C).
\]
To see the nontrivial inclusion, write an element of $D_{ik}(AC)$ as
\[
d=b_i^{-1}acb_k,
\qquad a\in A,\ c\in C.
\]
The intermediate unit $x=s(a)=r(c)$ lies in a unique $R_j$.  Since $r:\widetilde B_j\to R_j$ is bijective, there is a unique $b_j\in\widetilde B_j$ with $r(b_j)=x$.  Inserting the unit $x=b_jb_j^{-1}$ gives
\[
d=(b_i^{-1}ab_j)(b_j^{-1}cb_k)
\in D_{ij}(A)D_{jk}(C).
\]
Conversely, if $d_1=b_i^{-1}ab_j\in D_{ij}(A)$ and $d_2=(b_j')^{-1}cb_k\in D_{jk}(C)$ are composable, then $s(b_j)=s(b_j')$.  Injectivity of $s|_{\widetilde B_j}$ gives $b_j=b_j'$, so $d_1d_2=b_i^{-1}acb_k\in D_{ik}(AC)$.  The intermediate unit belongs to a unique $R_j$, which also proves disjointness.

Only finitely many $j$ occur.  Indeed, every intermediate unit belongs to $s(A)\cap r(C)$, which is contained in a compact subset of $K^{(0)}$, while $(R_j)$ is locally finite.  Since $\widetilde\psi$ is a homomorphism on $\Bop^{\Bis(G)}(G)$,
\[
W_{ij}(A)W_{jk}(C)
=
\widetilde\psi\bigl(1_{D_{ij}(A)D_{jk}(C)}\bigr).
\]
Finite additivity over the disjoint decomposition gives
\[
\sum_jW_{ij}(A)W_{jk}(C)=W_{ik}(AC),
\]
which is precisely the $(i,k)$-entry of $\widehat v_{AC}$.  Hence $\widehat v_A\widehat v_C=\widehat v_{AC}$.
\end{proof}

\begin{prop}[Spatial matrix dilation]
The pair $(\widehat\pi,\widehat v)$ is a nondegenerate spatial covariant representation of the inverse-semigroup action of $S_K$ on $C_0(K^{(0)})$ on the $L^p$-space $\widehat E$.
\end{prop}

\begin{proof}
By Lemma~\ref{lem:support-projections}, $\widehat\pi$ is a nondegenerate multiplication representation. By Lemma~\ref{lem:vhat-spatial}, $A\mapsto\widehat v_A$ takes values in spatial partial isometries and satisfies (SCR2). The covariance identity (SCR1) is Lemma~\ref{lem:SCR1}, and the multiplicativity identity (SCR3) is Lemma~\ref{lem:SCR3}. Thus all conditions in \cite[Definitions~3.17 and 5.15]{BKM1} hold.
\end{proof}

Since $S_K$ is wide by Lemma~\ref{lem:SK}, \cite[Lemma~4.21]{BKM1} identifies the germ groupoid of its canonical action with $K$ via
\[
[A,x]\longmapsto(s|_A)^{-1}(x).
\]
Thus the canonical $S_K$-action models $K$ in the sense of \cite[Definition~4.20]{BKM1}, and \cite[Proposition~4.23]{BKM1} integrates $(\widehat\pi,\widehat v)$ to a representation of $\cC(K)$.  In the complex case, \cite[Remark~5.7]{BKM1}, following from \cite[Theorem~5.5]{BKM1}, makes this integrated representation $\|\cdot\|_I$-contractive.  Hence \cite[Theorem~5.13(4)]{BKM1} gives a unique contractive extension to $F^p(K)$.  Denote it by
\begin{equation}\label{eq:psihat}
\widehat\psi:F^p(K)\longrightarrow B(\widehat E).
\end{equation}

\begin{lem}[Exact corner recovery]\label{lem:corner-recovery}
For every $f\in\cC(G)$,
\[
\widehat\psi(j_U(f))=\psi(f)\oplus0
\]
with respect to $\widehat E=E_0\oplus_p\bigoplus_{i\neq0}E_i$.
\end{lem}

\begin{proof}
If $U=K^{(0)}$, then $G=K$, $I=\varnothing$, $\widehat E=E$, and the assertion is immediate. Assume $U\neq K^{(0)}$.

Let $A\subseteq G$ be a relatively compact open bisection. Because $G$ is clopen in $K$, $A$ is also a relatively compact open bisection of $K$, hence $A\in S_K\setminus\{K^{(0)}\}$. Since every arrow of $A$ has source and range in $U=R_0$,
\[
D_{00}(A)=A,
\qquad
D_{ij}(A)=\varnothing\quad((i,j)\neq(0,0)).
\]
Thus
\[
\widehat v_A=v_A\oplus0.
\]
Likewise, for $a\in C_0(U)$ extended by zero to $K^{(0)}$,
\[
\widehat\pi(a)=\pi(a)\oplus0.
\]
If $f\in C_c(A)$, write $f=a\delta_A$ with $a\in C_c(r(A))\subseteq C_0(U)$. The integrated-form formula of \cite[Proposition~4.23]{BKM1} gives
\[
\widehat\psi(j_U(f))
=
\widehat\pi(a)\widehat v_A
=
(\pi(a)v_A)\oplus0
=
\psi(f)\oplus0.
\]
By Lemma~\ref{lem:bisection-decomposition}, every element of $\cC(G)$ is a finite sum of functions of this form, so the assertion follows by linearity.
\end{proof}

\begin{thm}[Full-clopen reduction]\label{thm:full-clopen-all-p}
Let $K$ be a locally compact, locally Hausdorff, \'{e}tale groupoid with paracompact locally compact Hausdorff unit space, let $U\subseteq K^{(0)}$ be full and clopen, and put $G=K|_U$. Then
\[
\|j_U(f)\|_{F^p(K)}=\|f\|_{F^p(G)}
\qquad
(f\in\cC(G),\ p\in[1,\infty]).
\]
Hence $j_U$ extends to an isometric homomorphism from $F^p(G)$ onto
\[
\overline{j_U(\cC(G))}^{\,F^p(K)}.
\]
\end{thm}

\begin{proof}
For every $p\in[1,\infty]$, the easy inequality
\[
\|j_U(f)\|_{F^p(K)}\leq\|f\|_{F^p(G)}
\]
follows because $j_U$ is $\|\cdot\|_I$-isometric: every $\|\cdot\|_I$-contractive representation of $\cC(K)$ restricts along $j_U$ to an $\|\cdot\|_I$-contractive representation of $\cC(G)$.

At $p=1$ and $p=\infty$, \cite[Theorem~5.13(1)]{BKM1} identifies the full norms with $\|\cdot\|_{I,s}$ and $\|\cdot\|_{I,r}$, respectively, and $j_U$ preserves the corresponding fiber sums. Hence equality holds at the endpoints.

Now let $1<p<\infty$. Choose the isometric spatial representation \eqref{eq:psi}; the dilation above gives the contractive representation \eqref{eq:psihat} of $F^p(K)$. By Lemma~\ref{lem:corner-recovery},
\[
\|j_U(f)\|_{F^p(K)}
\geq
\|\widehat\psi(j_U(f))\|
=
\|\psi(f)\|
=
\|f\|_{F^p(G)}.
\]
Together with the easy inequality this proves equality. The assertion about the completion is immediate.
\end{proof}

\begin{cor}[Full Morita equivalence]\label{cor:full-morita}
Let $G$ and $H$ be locally compact, locally Hausdorff, \'{e}tale groupoids with paracompact unit spaces.  If $G$ and $H$ are equivalent, then $F^p(G)$ and $F^p(H)$ are Morita equivalent for every $p\in[1,\infty]$.
\end{cor}

\begin{proof}
Let $Z$ be a $G$--$H$ equivalence and
\[
L=G\sqcup Z\sqcup Z^{\op}\sqcup H
\]
its linking groupoid.  Since
\[
L^{(0)}=G^{(0)}\sqcup H^{(0)},
\]
the unit space is paracompact.  Both $G^{(0)}$ and $H^{(0)}$ are full and clopen in $L^{(0)}$, so Theorem~\ref{thm:full-clopen-all-p} identifies the closures of the two diagonal convolution algebras in $F^p(L)$ isometrically with $F^p(G)$ and $F^p(H)$.

Set
\[
\X_Z^\mathrm{full}=\overline{\cC(Z^{\op})}^{\,F^p(L)},
\qquad
\Y_Z^\mathrm{full}=\overline{\cC(Z)}^{\,F^p(L)}.
\]
The convolution formulas of Lemma~\ref{preMor} give the module actions, pairings, and Morita compatibility identities.  Their norm estimates now follow from submultiplicativity in $F^p(L)$ and the diagonal identifications above.

It remains to verify fullness and nondegeneracy.  Let $(e_\lambda)$ be the pairing-valued approximate identity from Proposition~\ref{prop:ai}, so that each $e_\lambda\in\cC(G)$ is a finite sum
\[
e_\lambda=\sum_i {}_G\!\langle\varphi_i^\lambda,\widetilde\varphi_i^\lambda\rangle,
\qquad \varphi_i^\lambda\in\cC(Z).
\]
By Corollary~\ref{approxid},
$e_\lambda*a\to a$ for $a\in\cC(G)$ and
$e_\lambda\cdot\varphi\to\varphi$ for $\varphi\in\cC(Z)$
in the corresponding $I$-norms.  Since the full $L^p$-norm is
dominated by the $I$-norm, these convergences also hold in the
norms inherited from $F^p(L)$. Moreover,
\[
e_\lambda*a
=\sum_i {}_G\!\langle\varphi_i^\lambda,
       \widetilde\varphi_i^\lambda\cdot a\rangle.
\]
It follows that the linear span of the $G$-valued pairing is dense in $F^p(G)$ and that the left action of $F^p(G)$ on $\Y_Z^\mathrm{full}$ is nondegenerate.  The $H$-valued fullness statement and the remaining nondegeneracy statements follow symmetrically.  Hence the completed Banach pair is full and nondegenerate, and we have a Morita equivalence.
\end{proof}

\section{Concrete $L^p$-modules and linking algebras} \label{sect_likVSFp}

\subsection{Concrete $L^p$-modules and a factorization criterion}

\begin{defn} (\cite[Definition 3.1]{Delfin26})
Let $p \in [1, \infty)$, let $(\Omega, \Sigma, \mu)$ 
and $(\Gamma, \Upsilon, \nu)$ be measure spaces, 
and let $B \subseteq \mathcal{B}(L^p(\nu))$
be an $L^p$-operator algebra. 
We say $(\X, \Y)$ is a concrete $L^p$-module over $B$ 
if $\X \subseteq \mathcal{B}(L^p(\mu), L^p(\nu))$ and $\Y \subseteq  \mathcal{B}(L^p(\nu), L^p(\mu))$
are closed subspaces satisfying 
\begin{enumerate}
\item $yb \in \Y$ for all $y \in \Y$, $b \in B$,
\item $bx \in \X$ for all $x \in \X$, $b \in B$, 
\item $xy \in B$ for all $x \in \X$, $y \in \Y$.   
\end{enumerate}
\end{defn}

\begin{lem}\label{lem:(1/2)ineq}
Let $p \in [1, \infty)$. Then any concrete $L^p$-module $(\X, \Y)$ over 
an $L^p$-operator algebra $B \subseteq \mathcal{B}(L^p(\nu))$ is a Banach $B$-pair. 
Moreover, if we set $A_0 \coloneqq \mathrm{span}(\Y\X)  \subseteq \mathcal{B}(L^p(\mu))$, then the map
\begin{align*}
A_0 &\to \mathcal{K}_B(\X,\Y) \\
\sum_{j=1}^n y_j x_j &\mapsto \sum_{j=1}^n \theta_{y_j, x_j}
\end{align*}
is a contractive homomorphism with dense range. 
\end{lem}

\begin{proof}
The pairing $\langle x, y \rangle_B \coloneqq xy \in B$ makes $(\X, \Y)$ a Banach $B$-pair by submultiplicativity of the operator norm.  The displayed map is multiplicative and has dense range by definition, so we only need to show it is contractive. For any $y_0 \in \Y$,
\[
\Big\| \sum_{j=1}^n \theta^r_{y_j,x_j }(y_0)\Big\|_\Y = \Big\| \Big(\sum_{j=1}^n y_jx_j \Big) y_0\Big\|_\Y  \leq \Big\| \sum_{j=1}^n y_jx_j \Big\|_{\mathcal{B}(L^p(\mu))}\|y_0\|_\Y,
\]
and similarly, for any $x_0 \in \X$,  
\[
\Big\| \sum_{j=1}^n \theta^l_{y_j,x_j }(x_0)\Big\|_\X \leq \| x_0\|_\X  \Big\| \sum_{j=1}^n y_jx_j \Big\|_{\mathcal{B}(L^p(\mu))}.
\]
It now follows at once that 
\[
\Big\|  \sum_{j=1}^n \theta_{y_j, x_j} \Big\|_{\mathcal{K}_B(\X,\Y)}\leq  \Big\| \sum_{j=1}^n y_jx_j\Big\|_{\mathcal{B}(L^p(\mu))}.
\] 
\end{proof}

\begin{rem}
Let $\overline{A_0}$ denote the closure of $A_0=\mathrm{span}(\Y\X)$ in its concrete operator norm in $\mathcal{B}(L^p(\mu))$. Lemma~\ref{lem:(1/2)ineq} gives a unique contractive homomorphism
\[
\overline{\kappa}\colon \overline{A_0}\longrightarrow \mathcal{K}_B(\X,\Y)
\]
extending the algebraic map $\kappa$, and the range of $\overline{\kappa}$ is dense. In general there is no reason for this dense range to be closed, so no surjectivity is claimed at this level. A reverse estimate
\[
\|t\|_{\mathcal{B}(L^p(\mu))}\leq C\|\kappa(t)\|_{\mathcal{K}_B(\X,\Y)},\qquad t\in A_0,
\]
forces the extension to have closed range and hence, by density, to be onto. Assumption~\ref{assumption:concrete} gives such a reverse estimate: Theorem~\ref{thm_modnorm=repnorm} gives the isometric case, while Corollary~\ref{cor_modnorm=repnorm} gives a bounded inverse under a finite factorization bound.

In \cite[Definition 3.20]{Delfin26}, the compact-module morphisms of the concrete $L^p$-module $(\X,\Y)$ are instead defined to be precisely $\overline{A_0}$, so that they are automatically concrete $L^p$-operator algebras. Even when $p=2$, where $B$ is a $C^*$-algebra and $\Y$ is a concrete Hilbert $B$-module, the comparison with $\mathcal{K}_B(\Y)$ requires the relevant nondegeneracy hypotheses (see \cite[Proposition 2.7]{Delfin24}).
\end{rem}

We next give a sufficient condition for the comparison map in
Lemma~\ref{lem:(1/2)ineq} to have a bounded inverse.

\begin{assump}\label{assumption:concrete}
Fix $p\in[1,\infty)$. Let
\[
A\subseteq\mathcal B(L^p(\mu)),\qquad
B\subseteq\mathcal B(L^p(\nu))
\]
be concrete $L^p$-operator algebras, and let
\[
\X\subseteq\mathcal B(L^p(\mu),L^p(\nu)),\qquad
\Y\subseteq\mathcal B(L^p(\nu),L^p(\mu))
\]
be closed subspaces such that $(\X,\Y)$ is a concrete
$L^p$-module over $B$ and $(\Y,\X)$ is a concrete
$L^p$-module over $A$.
Assume in addition that there are dense subalgebras
$\mathcal A_0\subseteq A$ and $\mathcal B_0\subseteq B$
and complex vector spaces $\X_0\subseteq\X$ and
$\Y_0\subseteq\Y$ such that:
\begin{enumerate}
\item\label{assumption:concretepreM}
$(\X_0,\Y_0)$, with the norms inherited from $\X$ and $\Y$,
forms a pre-Morita equivalence between
$\mathcal A_0\subseteq A$ and $\mathcal B_0\subseteq B$,
with pairings
\[
\langle x,y\rangle_{\mathcal B_0}=xy,\qquad
{}_{\mathcal A_0}\langle y,x\rangle=yx.
\]
\item\label{assumption:L1nondeg}
The linear span of $A L^p(\mu)$ is dense in $L^p(\mu)$.
\item\label{assumption:Apnondeg}
The linear span of $\Y\X$ is dense in $A$.
\item\label{assumption:concretecai}
There are nets $(n_\lambda)_{\lambda\in\Lambda}$ in
$\mathbb Z_{>0}$,
$x_\lambda=(x_{\lambda,1},\ldots,x_{\lambda,n_\lambda})
\in\X_0^{n_\lambda}$, and
$y_\lambda=(y_{\lambda,1},\ldots,y_{\lambda,n_\lambda})
\in\Y_0^{n_\lambda}$ such that:
\begin{enumerate}
\item\label{assump_a_cai}
\[
e_\lambda \coloneqq \sum_{j=1}^{n_\lambda}
y_{\lambda,j}x_{\lambda,j}\in\mathcal A_0
\]
is an approximate identity for $A$;
\item\label{assump_b_p}
for every $\lambda$, the operator
\[
x_\lambda:L^p(\mu)\to L^p(\nu)^{n_\lambda},\qquad
x_\lambda\xi=(x_{\lambda,1}\xi,\ldots,
x_{\lambda,n_\lambda}\xi),
\]
is contractive;
\item\label{assump_c_q}
there is a constant $K>0$, independent of $\lambda$, such that
\[
\big\|(\|y_{\lambda,1}\|_\Y,\ldots,
\|y_{\lambda,n_\lambda}\|_\Y)\big\|_{p'}\le K
\qquad(\lambda\in\Lambda),
\]
where $p'$ is the H\"{o}lder conjugate of $p$.
\end{enumerate}
In particular, $\|e_\lambda\|_A\le K$ for every $\lambda$.
If $K\le1$, then $(e_\lambda)_\lambda$ is a contractive
approximate identity for $A$.
\end{enumerate}
\end{assump}

\begin{thm}\label{thm_modnorm=repnorm}
Under Assumption~\ref{assumption:concrete}, suppose $K\leq1$.
Then the extension of the map in Lemma~\ref{lem:(1/2)ineq}
is an isometric isomorphism
\[
A\cong\mathcal K_B(\X,\Y).
\]
In particular, $\mathcal K_B(\X,\Y)$ is an $L^p$-operator
algebra isometrically represented on $L^p(\mu)$.
\end{thm}

\begin{proof}
Let $\kappa \colon\mathrm{span}(\Y\X) \to \mathcal{K}_{B}(\X, \Y)$ be the map from Lemma~\ref{lem:(1/2)ineq}.
That is, for each $n \in \mathbb{Z}_{\geq 1}$ and  $x_1, \ldots, x_n \in \X$, $y_1, \ldots, y_n \in \Y$, we
put
\[
\kappa\Big( \sum_{j=1}^n y_jx_j\Big) \coloneqq  \sum_{j=1}^n \theta_{y_j,x_j}.
\]
Since $\mathrm{span}(\Y\X)$ is dense in $A$ by Assumption~\ref{assumption:concrete} \eqref{assumption:Apnondeg} and $\kappa$ has dense range, it suffices to show that $\kappa$ is isometric, that is $\| \kappa(t)\|_{\mathcal{K}_{B}(\X, \Y)} = \| t \|_{A}$ for every $t \in \mathrm{span}(\Y\X)$. Let
\[
t \coloneqq \sum_{j=1}^n y_jx_j \in \mathrm{span}(\Y\X) \subseteq A. 
\]
Lemma~\ref{lem:(1/2)ineq} already gives that $\kappa$ is a contraction, independently of Assumption~\ref{assumption:concrete}, so it suffices to show that
\[
\| t \|_{A} \leq \| \kappa(t) \|_{\mathcal{K}_{B}(\X, \Y)}.
\]

For every measure space and every $p\in[1,\infty)$, $L^{p'}(\mu)$ is norming for $L^p(\mu)$:
\[
\|\zeta\|_p=\sup_{\|\eta\|_{p'}\leq1}|(\eta\mid\zeta)|,\qquad \zeta\in L^p(\mu),
\]
where $(-\mid -)$ is the usual integral pairing. Hence
\[
\| t \|_{A} = \sup_{\| \xi \|_p \leq 1} \|t\xi\|_p = \sup_{\| \xi \|_p \leq 1, \| \eta\|_{p'} \leq 1} |(\eta \mid t\xi)|.
\]
It therefore suffices to show that 
\[
|(\eta \mid t\xi)| \leq \| \kappa(t) \|_{\mathcal{K}_{B}(\X, \Y)} \| \xi \|_p \| \eta \|_{p'}
\]
for arbitrary $\xi \in L^p(\mu)$ and $\eta \in L^{p'}(\mu)$. To do so, we fix  $\xi \in L^p(\mu)$ and use Assumption~\ref{assumption:concrete} \eqref{assump_a_cai} to define a net $(\xi_\lambda)_{\lambda \in \Lambda}$ in $ L^p(\mu)$ 
by letting
\[
\xi_\lambda \coloneqq e_\lambda\xi = \sum_{j=1}^{n_\lambda} y_{\lambda,j}x_{\lambda,j}\xi \in L^p(\mu). 
\]
Since $K\leq1$, Assumption~\ref{assumption:concrete}\eqref{assumption:concretecai} gives $\sup_\lambda\|e_\lambda\|_{A}\leq1$. We claim that $e_\lambda\xi\to\xi$ for every $\xi\in L^p(\mu)$. Indeed, if $\xi=a\zeta$ with $a\in A$, then
\[
\|e_\lambda\xi-\xi\|_p\leq\|e_\lambda a-a\|_{A}\|\zeta\|_p\longrightarrow0.
\]
The linear span of such vectors is dense by Assumption~\ref{assumption:concrete}\eqref{assumption:L1nondeg}, and the uniform bound on $(e_\lambda)_\lambda$ extends the convergence to all of $L^p(\mu)$. Thus $\xi_\lambda\to\xi$, and hence $(\eta\mid t\xi_\lambda)\to(\eta\mid t\xi)$ for every $\eta\in L^{p'}(\mu)$. Next, for each $\lambda \in \Lambda$ and each $j \in \{1, \ldots, n_\lambda\}$, we have
\[
 \|ty_{\lambda,j}\|_{\Y} \leq \| \kappa(t)^r\|\|y_{\lambda,j}\|_{\Y} \leq \| \kappa(t)\|_{\mathcal{K}_{B}(\X, \Y)} \|y_{\lambda,j}\|_{\Y}.
 \]
Therefore, 
\begin{align*}
|(\eta \mid t\xi_\lambda)| & = \Big| \sum_{j=1}^{n_\lambda} \big( \eta \mid t(y_{\lambda,j}x_{\lambda,j}\xi ) \big) \Big|\\
& \leq\| \eta\|_{p'}  \sum_{j=1}^{n_\lambda} \|t(y_{\lambda,j}x_{\lambda,j}\xi )\|_p\\
& \leq \| \eta\|_{p'}\sum_{j=1}^{n_\lambda}  \|ty_{\lambda,j}\|_{\Y} \|x_{\lambda,j}\xi\|_p\\
& \leq \| \eta\|_{p'}\| \kappa(t)\|_{\mathcal{K}_{B}(\X, \Y)} \sum_{j=1}^{n_\lambda}  \|y_{\lambda,j}\|_{\Y} \|x_{\lambda,j}\xi\|_{p}.
\end{align*}
Next, notice that Hölder's inequality in $\mathbb{R}^{n_\lambda}$ together with Assumption~\ref{assumption:concrete}\eqref{assump_b_p} and \eqref{assump_c_q}  (here $K \leq 1$) give
\[
\sum_{j=1}^{n_\lambda}  \|y_{\lambda,j}\|_{\Y} \|x_{\lambda,j}\xi\|_{p} \leq 
\big\| (\| y_{\lambda,1} \|_{\Y},  \ldots, \| y_{\lambda,n_\lambda} \|_{\Y}) \big\|_{p'}\Big( \sum_{j=1}^{n_\lambda}  \|x_{\lambda,j}\xi\|_{p}^p \Big)^{1/p} \leq \|\xi\|_p.
\]
Combining these estimates gives 
\[
|(\eta \mid t\xi_\lambda)| \leq \| \eta\|_{p'}\| \kappa(t)\|_{\mathcal{K}_{B}(\X, \Y)} \|\xi\|_p,
\]
whence, taking the limit over $\lambda \in \Lambda$, we get that the preceding inequality holds, as desired. 
\end{proof}

\begin{cor}\label{cor_modnorm=repnorm}
Under Assumption~\ref{assumption:concrete}, the extension
\[
\overline\kappa:A\longrightarrow\mathcal K_B(\X,\Y)
\]
is a Banach algebra isomorphism. More precisely,
\[
\|\overline\kappa(a)\|\leq\|a\|_A
\leq K\|\overline\kappa(a)\|,
\qquad a\in A.
\]
\end{cor}

\begin{proof}
Repeating the proof of Theorem~\ref{thm_modnorm=repnorm}
with the bound $K$ in place of $1$ gives
\[
\|t\|_A\leq K\|\kappa(t)\|_{\mathcal K_B(\X,\Y)},
\qquad t\in\operatorname{span}(\Y\X).
\]
Hence $\overline\kappa$ is bounded below, so its range is
closed. Lemma~\ref{lem:(1/2)ineq} gives dense range, and
therefore $\overline\kappa$ is onto.
\end{proof}

\subsection{Concrete linking algebras}

Let $A$ and $B$ be Morita equivalent Banach algebras via the pair $({}_B\X_{A}, {}_A\Y_B)$. 
The \textit{linking algebra of the equivalence} is, algebraically, 
\[
\mathbb{L} = 
\begin{pmatrix}
A & \Y \\
\X & B
\end{pmatrix}.
\]
 We now describe a general way to put norms on $\mathbb{L}$. Consider the Banach $B$-pair direct sum:
\[ 
(\X,\Y) \oplus_{\alpha} B  =( ({}_B\X \oplus {}_B B, \Y_B \oplus  B_B), \alpha ),
\]
where $\alpha=(\alpha_l, \alpha_r)$ consists of a pair of norms
\[
\alpha_l \colon {}_B\X \oplus {}_B B \to \mathbb{R}_{\geq 0}, \ 
\alpha_r \colon \Y_B \oplus B_B \to \mathbb{R}_{\geq 0},
\]
satisfying,
for all $x \in \X, y \in \Y, b, b_0 \in B$,
\begin{enumerate}[label=($A_\arabic*$)]
\item $\alpha_l(b_0x,b_0b) \leq \|b_0\|\alpha_l(x,b)$, \label{cond_alpha1}
\item $ \max\{\|x\|, \|b\|\} \leq \alpha_l(x,b) \leq \|x\| + \| b \|$, 
\item $ \alpha_r(yb_0,bb_0) \leq \alpha_r(y,b)\|b_0\|$, \label{cond_alpha3}
\item $\max\{\|y\|, \|b\|\} \leq \alpha_r(y,b) \leq \|y\| + \| b \|$, 
\item $\| \langle x,y\rangle_B + b_0b \| \leq \alpha_l(x,b_0)\alpha_r(y,b)$.  \label{cond_alpha5}  
\end{enumerate}
\begin{rem}
Note that an example that satisfies the above conditions is to let
$1<p<\infty$, let $q$ be the Hölder conjugate of $p$, and put
\[
\alpha_l(x,b) = (\|x\|^q + \| b\|^q )^{1/q}, \ \alpha_r(y,b) = (\|y\|^p + \| b\|^p )^{1/p}.
\]
We will see in Equations~\eqref{Eq:alpha_p_l} and \eqref{Eq:alpha_p_r} that, with some $p$-operator space representability conditions imposed on the pair $(\X, \Y)$, there
is another natural option for $\alpha$ that generalizes the $C^*$-norm of the classic linking algebra. 
\end{rem}
Observe that Conditions \ref{cond_alpha1}, \ref{cond_alpha3}, and  \ref{cond_alpha5} make $(\X,\Y) \oplus_{\alpha} B$ a Banach $B$-pair in its own right when equipped with the 
bilinear pairing
\[
\langle (x,b_1), (y,b_2) \rangle \coloneqq \langle x,y\rangle_B + b_1b_2.
\]
Notice that elements of $\mathbb{L}$ naturally act on 
the pair $(\X, \Y) \oplus_{\alpha} B$ as linear operators (see Definition \ref{defn_LinOP}). 
Indeed, fix
\[
T = \begin{pmatrix}
a & y\\
x & b
\end{pmatrix} \in \mathbb{L},
\]
so that the formal matrix multiplication 
\[
\begin{pmatrix}
x_0 & b_0
\end{pmatrix}
\begin{pmatrix}
a & y \\
x & b
\end{pmatrix}
= 
\begin{pmatrix}
x_0\cdot a+b_0\cdot x  & \langle x_0,y\rangle_B +  b_0b
\end{pmatrix}
\]
gives a map 
$T^l \colon \X \oplus B \to \X \oplus B$ given by 
\[
T^l(x_0, b_0) =  (
x_0\cdot a+b_0\cdot x,    \langle x_0,y\rangle_B +  b_0b).
\]
Similarly, since 
\[
\begin{pmatrix}
a & y\\
x & b
\end{pmatrix}
\begin{pmatrix}
y_0\\
b_0
\end{pmatrix}
=
\begin{pmatrix}
a\cdot y_0 + y\cdot b_0 \\
\langle x,y_0\rangle_B + bb_0
\end{pmatrix},
\]
we define $T^r \colon \Y \oplus B \to \Y \oplus B$
\[
T^r(y_0,b_0) = (
a\cdot y_0 + y\cdot b_0, 
\langle x,y_0\rangle_B + bb_0
).
\]

Let
\[
\iota_\alpha\colon \mathbb{L}\longrightarrow \mathcal{L}_B((\X,\Y)\oplus_\alpha B)
\]
denote this concrete block representation. We write $\mathbb{L}^{\mathrm{alg}}_\alpha=\iota_\alpha(\mathbb{L})$ with the inherited operator norm, and define the \emph{concrete linking algebra associated with $\alpha$} by
\[
\mathbb{L}_\alpha\coloneqq \overline{\mathbb{L}^{\mathrm{alg}}_\alpha}^{\,\mathcal{L}_B((\X,\Y)\oplus_\alpha B)}.
\]

\begin{lem}\label{lem_Lclosed}
For every choice of $\alpha$ satisfying Conditions \ref{cond_alpha1}--\ref{cond_alpha5}, $\mathbb{L}_\alpha$ is a norm-closed Banach subalgebra of $\mathcal{L}_B((\X,\Y)\oplus_\alpha B)$. Moreover, for every algebraic block
\[
T=\begin{pmatrix}a&y\\x&b\end{pmatrix}\in\mathbb{L}
\]
one has
\[
\|\iota_\alpha(T)\|\leq \|a\|+\|y\|+\|x\|+\|b\|.
\]
No converse estimate on the four entry norms is asserted in this generality.
\end{lem}
\begin{proof}
The formulas above and Conditions \ref{cond_alpha1}--\ref{cond_alpha5} show that $\iota_\alpha(\mathbb{L})$ is a subalgebra of $\mathcal{L}_B((\X,\Y)\oplus_\alpha B)$ and give the displayed upper bound. Its operator-norm closure is therefore a norm-closed subalgebra.
\end{proof}

The following is the analogue of \cite[Corollary 3.21]{RaWi98}.
\begin{prop}\label{Prop:L_alpha=K_B}
Let $(\X,\Y)$ be a Morita equivalence between the Banach algebras $A$ and $B$, let $\alpha$ satisfy Conditions \ref{cond_alpha1}--\ref{cond_alpha5}, and let $\mathbb{L}_\alpha$ be the concrete operator-norm closure defined above. Then
\[
\mathcal{K}_B((\X,\Y)\oplus_\alpha B)=\mathbb{L}_\alpha.
\]
For any $x\in\X$, $y\in\Y$, and $b_1,b_2\in B$,
\[
\theta_{(y,b_1),(x,b_2)}=
\begin{pmatrix}
{}_A\langle y,x\rangle&y\cdot b_2\\
b_1\cdot x&b_1b_2
\end{pmatrix}.
\]
\end{prop}
\begin{proof}
Direct computations using the equivalence compatibility conditions give
\[
\theta_{(y,b_1),(x,b_2)}^{l}(x_0,b_0)=
\begin{pmatrix}x_0&b_0\end{pmatrix}
\begin{pmatrix}
{}_A\langle y,x\rangle&y\cdot b_2\\
b_1\cdot x&b_1b_2
\end{pmatrix}
\]
and
\[
\theta_{(y,b_1),(x,b_2)}^{r}(y_0,b_0)=
\begin{pmatrix}
{}_A\langle y,x\rangle&y\cdot b_2\\
b_1\cdot x&b_1b_2
\end{pmatrix}
\begin{pmatrix}y_0\\b_0\end{pmatrix}.
\]
Hence every rank-one operator belongs to $\mathbb{L}^{\mathrm{alg}}_\alpha$, and therefore
\[
\mathcal{K}_B((\X,\Y)\oplus_\alpha B)\subseteq\mathbb{L}_\alpha.
\]

Conversely, fullness and nondegeneracy give
\[
\overline{\mathrm{span}}\,{}_A\langle\Y,\X\rangle=A,\qquad
\overline{\mathrm{span}}(\Y B)=\Y,\qquad
\overline{\mathrm{span}}(B\X)=\X,\qquad
\overline{\mathrm{span}}(BB)=B.
\]
The corresponding four kinds of corner matrices are finite sums of rank-one operators:
\begin{align*}
\begin{pmatrix}{}_A\langle y,x\rangle&0\\0&0\end{pmatrix}
&=\theta_{(y,0),(x,0)},&
\begin{pmatrix}0&yb\\0&0\end{pmatrix}
&=\theta_{(y,0),(0,b)},\\
\begin{pmatrix}0&0\\bx&0\end{pmatrix}
&=\theta_{(0,b),(x,0)},&
\begin{pmatrix}0&0\\0&b_1b_2\end{pmatrix}
&=\theta_{(0,b_1),(0,b_2)}.
\end{align*}
By the upper estimate in Lemma~\ref{lem_Lclosed}, convergence in the four corner norms implies convergence in the concrete block-operator norm. Thus $\mathbb{L}^{\mathrm{alg}}_\alpha\subseteq\mathcal{K}_B((\X,\Y)\oplus_\alpha B)$, and taking operator-norm closures gives the reverse inclusion.
\end{proof}

We now specialize to concrete $L^p$-modules.  We use the following
bimodule version of the setting in \cite{Delfin26}.

\begin{assump}\label{assumption:repn} 
 Assume that $A \subseteq \mathcal{B}(L^p(\mu))$ 
 and $B \subseteq \mathcal{B}(L^p(\nu))$ are concrete $L^p$-operator 
 algebras and that 
 \[
 \X \subseteq \mathcal{B}(L^p(\mu), L^p(\nu)), \Y \subseteq  \mathcal{B}(L^p(\nu), L^p(\mu)),
 \]
are closed subspaces satisfying
\begin{enumerate}
\item the linear spans of $B\X$ and $\X A$ are dense in $\X$,
\item the linear spans of $A\Y$ and $\Y B$ are dense in $\Y$,
\item the linear span of $\Y\X$ is dense in $A$,
\item the linear span of $\X\Y$ is dense in $B$.
\end{enumerate}
\end{assump}
The conditions in Assumption~\ref{assumption:repn} make $(\X,\Y)$ a Morita equivalence between $A$ and $B$, with module actions given by composition and pairings ${}_A\langle y, x \rangle=yx$ and $\langle x, y \rangle_B=xy$. 
The concrete nature also allows us to put a row $p$-operator norm on 
$ {}_B\X \oplus {}_B B $ as follows: For each $x \in \X$ and $b \in B$,
we think of $(x,b)$ as a row operator   
\begin{align*}
(x,b)  \colon L^p(\mu) \oplus_p L^p(\nu) & \to L^p(\nu)\\
 (\xi, \eta) & \mapsto (x,b)(\xi, \eta)  \coloneqq x\xi + b\eta. 
\end{align*}
The $p$-operator norm on ${}_B\X \oplus {}_B B$ is then
given by
\begin{equation}\label{Eq:alpha_p_l}
\alpha_{p,l}(x,b) = \sup\{ \|x\xi + b\eta\|_{p} \colon \xi \in L^p(\mu), \eta \in L^p(\nu), \|\xi\|_p^p+\|\eta\|_p^p \leq 1\}.
\end{equation}
Similarly, we put a column $p$-operator norm on 
$\Y_B\oplus  B_B $ by thinking of the pair $(y,b)$, 
for each $y \in \Y$ and $b\in B$, 
as a column operator 
\begin{align*}
(y,b) \colon L^p(\nu)  & \to L^p(\mu) \oplus_p L^p(\nu)\\
\eta & \mapsto (y,b)\eta \coloneqq (y\eta, b\eta). 
\end{align*}
The $p$-operator norm on $\Y_B\oplus  B_B$ becomes 
\begin{equation}\label{Eq:alpha_p_r}
\alpha_{p,r}(y,b) = \sup\{ (\| y\eta\|_p^p + \| b\eta\|_p^p)^{1/p} \colon \eta \in L^p(\nu), \| \eta\|_p\leq 1\}.
\end{equation}
Furthermore, the work in \cite[Section 3.2]{Delfin26} shows that the pair $\alpha_p = (\alpha_{p,l}, \alpha_{p,r})$ satisfies 
Conditions \ref{cond_alpha1}-\ref{cond_alpha5} and therefore it makes sense 
to talk about the $p$-linking algebra $\mathbb{L}_{\alpha_p}$.

\begin{prop}\label{prop_pnormLink}
Let $A\subseteq\mathcal{B}(\cH_A)$ and
$B\subseteq\mathcal{B}(\cH_B)$ be nondegenerately represented
$C^*$-algebras, and let
\[
\Y\subseteq\mathcal{B}(\cH_B,\cH_A)
\]
be a concrete realization of an $A$-$B$ imprimitivity bimodule. Put
\[
\X=\Y^*
=\{y^*:y\in\Y\}
\subseteq\mathcal{B}(\cH_A,\cH_B).
\]
Then the concrete linking algebra $\mathbb{L}_{\alpha_2}$ associated with
$(\X,\Y)$ is $*$-isomorphic to the usual linking $C^*$-algebra of the
imprimitivity bimodule. In particular, every $C^*$-imprimitivity
bimodule admits a concrete realization for which this conclusion holds.
\end{prop}

\begin{proof}
For $y\in\Y$ and $b\in B$, the row and column norms defining
$\alpha_2$ are
\[
\alpha_{2,l}(y^*,b)
 =\|y^*y+bb^*\|_B^{1/2},
\qquad
\alpha_{2,r}(y,b)
 =\|y^*y+b^*b\|_B^{1/2}.
\]
Thus the concrete representation of the algebraic block matrix
\[
\begin{pmatrix}
a&y\\
y'^*&b
\end{pmatrix}
\]
on $\cH_A\oplus\cH_B$ is exactly the standard linking representation.
By the standard linking-algebra theorem
(see \cite[Corollary~3.21]{RaWi98}), the operator-norm closure of these
algebraic blocks is the usual linking $C^*$-algebra. Since
$\mathbb{L}_{\alpha_2}$ is precisely this concrete operator-norm
completion, the first assertion follows.

Finally, every imprimitivity bimodule admits a concrete realization of
the above form by \cite[Proposition~4.8]{Exel93}.
\end{proof}

\subsection{The linking groupoid algebra}

For $p\in[1,\infty)$, we compare the concrete linking algebra $\mathbb L_{\alpha_p}$ with $F^p_{\mathrm{red}}(L)$.  

To do so, we begin by showing that the conditions in Assumption~\ref{assumption:repn} are met 
for the pair $(\X_Z, \Y_Z)$ defined in the proof of Theorem~\ref{thm:main}. 
For notational convenience, let $(\Omega_{0}, \Sigma_{0}, \mu_{0})$ be the measure space defined so that 
\[
L^p(\mu_{0}) = \bigoplus_{u\in G^{(0)}} \ell^p(G_u) \oplus_p \bigoplus_{v\in H^{(0)}} \ell^p(Z_v).
\]
Similarly, we let $(\Gamma_{0}, \Upsilon_{0},\nu_{0})$ be the measure space such that 
\[
L^p(\nu_{0}) = \bigoplus_{u\in G^{(0)}} \ell^p(Z^{op}_u) \oplus_p \bigoplus_{v\in H^{(0)}} \ell^p(H_v).
\]

\begin{lem} \label{lem:repn}
Consider the pre-Morita equivalence $(\mathcal{C}(Z^{op}), \mathcal{C}(Z))$ between $\mathcal{C}(G) \subseteq F_{red}^p(G)$ and $\mathcal{C}(H) \subseteq F_{red}^p(H)$ provided by Lemma~\ref{preMor}. Then there are natural embeddings
\[
\mathcal{C}(Z^{op}) \hookrightarrow \mathcal{B}\left(L^p(\mu_0) , L^p(\nu_0) \right), \ 
\mathcal{C}(Z) \hookrightarrow \mathcal{B}\left(L^p(\nu_0) ,L^p(\mu_0) \right), 
\]
and 
 \[
 \mathcal{C}(G) \hookrightarrow \mathcal{B}\left( L^p(\mu_0) \right),  \ 
  \mathcal{C}(H) \hookrightarrow \mathcal{B}\left( L^p(\nu_0)  \right)
 \]
 such that the induced $p$-operator norms coincide with the norms used in Lemma~\ref{preMor}. 
\end{lem}
\begin{proof}
Recall that the Haar system from Lemma~\ref{lem:haarL} realizes $\mathcal{C}(L)\subseteq F^p_{red}(L)$ as a $2\times2$ block algebra acting on $L^p(\mu_0)\oplus_pL^p(\nu_0)$ via the decomposition 
\[
\mathcal{C}(L) = 
\begin{pmatrix}
\mathcal{C}(G) & \mathcal{C}(Z) \\
\mathcal{C}(Z^{op}) & \mathcal{C}(H)
\end{pmatrix} \subseteq \begin{pmatrix}
\mathcal{B}\left(L^p(\mu_0) \right)& \mathcal{B}\left(L^p(\nu_0) , L^p(\mu_0) \right)\\
\mathcal{B}\left(L^p(\mu_0) , L^p(\nu_0) \right)& \mathcal{B}\left(L^p(\nu_0) \right)
\end{pmatrix}.
\]
The desired embeddings come from looking at the appropriate corner actions, which coincide isometrically with 
the norms used in Lemma~\ref{preMor} thanks to Theorem~\ref{thm:mainF_red^p}. 
\end{proof}

Let $\X_Z$ and $\Y_Z$ be the completions of $\mathcal{C}(Z^{op})$ and $\mathcal{C}(Z)$, respectively, as defined in the proof of Theorem~\ref{thm:main}. Thanks to Lemma~\ref{lem:repn}, we
regard
\[
\X_Z\subseteq\mathcal{B}(L^p(\mu_0),L^p(\nu_0)),
\qquad
\Y_Z\subseteq\mathcal{B}(L^p(\nu_0),L^p(\mu_0)),
\]
and $A=F^p_{\mathrm{red}}(G)$ and $B=F^p_{\mathrm{red}}(H)$ as concrete
$L^p$-operator algebras on $L^p(\mu_0)$ and $L^p(\nu_0)$, respectively. 

Theorem~\ref{thm:main} says that $(\X_Z,\Y_Z)$ implements a Morita
equivalence between $A$ and $B$. Hence
\[
\begin{aligned}
\overline{\mathrm{span}}(B\X_Z)
&=\overline{\mathrm{span}}(\X_ZA)=\X_Z,\\
\overline{\mathrm{span}}(A\Y_Z)
&=\overline{\mathrm{span}}(\Y_ZB)=\Y_Z,\\
\overline{\mathrm{span}}(\Y_Z\X_Z)&=A,
\qquad
\overline{\mathrm{span}}(\X_Z\Y_Z)=B.
\end{aligned}
\]
Thus Assumption~\ref{assumption:repn} is satisfied, and the concrete
algebra $\mathbb{L}_{\alpha_p}$ is defined by
Equations~\eqref{Eq:alpha_p_l} and \eqref{Eq:alpha_p_r}.

We will also use spatial nondegeneracy. The regular representations of $A$ and $B$ are nondegenerate. Combining this with the last two density
relations above gives
\begin{enumerate}
\item $AL^p(\mu_0)$ and $\Y_ZL^p(\nu_0)$ are dense in $L^p(\mu_0)$;
\item $BL^p(\nu_0)$ and $\X_ZL^p(\mu_0)$ are dense in $L^p(\nu_0)$.
\end{enumerate}
For example, density of $\Y_ZL^p(\nu_0)$ follows from
\[
AL^p(\mu_0)
\subseteq
\overline{\mathrm{span}}\,\Y_Z(\X_ZL^p(\mu_0)).
\]

\begin{thm} \label{thm:LpLink}
For every $p\in[1,\infty)$ there is a canonical contractive injective homomorphism
\[
\Phi_p\colon F_{red}^p(L)\longrightarrow\mathbb{L}_{\alpha_p}
\]
with dense range. For $p=2$, $\Phi_2$ is an isometric isomorphism.

More generally, fix $p\in[1,\infty)$ and put
\[
\X'_p=\X_Z\oplus_{\alpha_{p,l}}B,
\qquad
\Y'_p=\Y_Z\oplus_{\alpha_{p,r}}B.
\]
If $(\X'_p,\Y'_p)$ satisfies Assumption
\ref{assumption:concrete} with
\[
A=F^p_{\mathrm{red}}(L),\qquad
B=F^p_{\mathrm{red}}(H),
\]
then $\Phi_p$ is a Banach algebra isomorphism. It is
isometric when the constant $K$ in Assumption
\ref{assumption:concrete} satisfies $K\leq1$.
\end{thm}
\begin{proof}
With $\X'_p$ and $\Y'_p$ as above, Equations~\eqref{Eq:alpha_p_l} and \eqref{Eq:alpha_p_r} show that $(\X'_p,\Y'_p)$ is a concrete $L^p$-module over $B$, represented as row and column operators between $L^p(\mu_0)\oplus_pL^p(\nu_0)$ and $L^p(\nu_0)$. Fullness and nondegeneracy of the Morita equivalence from Theorem~\ref{thm:main} imply that the linear span of $\Y'_p\X'_p$ is dense, in the concrete operator norm, in $F_{red}^p(L)$. Indeed, the four corners of such products are dense in $A$, $\Y_Z$, $\X_Z$, and $B$, respectively. Approximating the four corners separately and using the elementary upper estimate for a $2\times2$ block operator (the same one-sided estimate recorded in Lemma~\ref{lem_Lclosed}) gives convergence in the ambient block-operator norm.

Lemma~\ref{lem:(1/2)ineq} therefore gives a contractive homomorphism
\[
\Phi_p\colon F_{red}^p(L)\longrightarrow\mathcal{K}_B(\X'_p,\Y'_p)
\]
with dense range. Proposition~\ref{Prop:L_alpha=K_B}, applied with the specific operator norms $\alpha_p$, identifies the target with the concrete operator-norm completion $\mathbb{L}_{\alpha_p}$.

For injectivity, suppose that $\Phi_p(T)=0$. Then the right component of the module operator is zero, so
\[
T\begin{pmatrix}y\\b\end{pmatrix}=0
\]
for every $y\in\Y_Z$ and $b\in B$. Hence $T$ vanishes on the linear span of
\[
\left\{(y\eta,b\eta):y\in\Y_Z,\ b\in B,\ \eta\in L^p(\nu_0)\right\}.
\]
By the spatial nondegeneracy established above, that span is dense in $L^p(\mu_0)\oplus_pL^p(\nu_0)$. Since $F_{red}^p(L)$ is concretely represented on this space, $T=0$.

When $p=2$, the involution on $\mathcal{C}(L)$ interchanges
$\mathcal{C}(Z)$ and $\mathcal{C}(Z^{op})$. Hence, in the regular
Hilbert space realization above,
\[
\X_Z=\Y_Z^*.
\]
Proposition~\ref{prop_pnormLink} therefore identifies
$\mathbb{L}_{\alpha_2}$ with the usual linking $C^*$-algebra of the
reduced imprimitivity bimodule. On the other hand,
\cite[Remark~2.4 and Theorem~4.1]{SW} identifies
$F_{\mathrm{red}}^2(L)=C_r^*(L)$
with the same linking $C^*$-algebra. These identifications agree on the
dense algebraic block algebra $\mathcal{C}(L)$, so $\Phi_2$ is an
isometric isomorphism.

Finally, if the hypotheses of Assumption~\ref{assumption:concrete} hold for $(\X'_p,\Y'_p)$, Corollary~\ref{cor_modnorm=repnorm} supplies the reverse norm estimate. The range of $\Phi_p$ is then closed as well as dense, and hence $\Phi_p$ is onto; if $K\leq1$, Theorem~\ref{thm_modnorm=repnorm} makes the comparison isometric.
\end{proof}

\begin{rem}
Let $\Phi_p$ be the map from Theorem~\ref{thm:LpLink}. The contractive estimate
\[
\|\Phi_p(T)\|_{\mathbb L_{\alpha_p}}\leq \|T\|_{F^p_{\mathrm{red}}(L)}
\]
follows immediately from Lemma~\ref{lem:(1/2)ineq}. We remark that the reverse estimate requires an additional uniform factorization
estimate such as that in Assumption~\ref{assumption:concrete}, or a
separate norm argument.
\end{rem}

\section{Proper groupoid correspondences and Morita cycles}\label{sect_corresp}

We next consider Morita cycles associated with \'{e}tale groupoid
correspondences.  Proper correspondences give Morita cycles for the
full $L^p$-operator algebras.  For the reduced algebras, we construct
the associated Banach pair, characterize compactness when the left
action extends to the reduced completion, and give sufficient
conditions for this extension.

For correspondences we use the convention of \cite[Definition~2.3]{AKM}: the anchor map of
a groupoid action is continuous, but need not be open.

\begin{defn} \cite[Definitions~2.3, 3.1 and 3.3]{AKM} \label{def:groupoid-correspondence}
Let $G$ and $H$ be locally compact, locally Hausdorff, \'{e}tale groupoids.
An (\'{e}tale) groupoid correspondence from $G$ to $H$ is a topological
space $Z$ equipped with a continuous left $G$-action with anchor map
$r_Z:Z\to G^{(0)}$ and a continuous right $H$-action with anchor map
$s_Z:Z\to H^{(0)}$ such that
\begin{enumerate}
\item the actions of $G$ and $H$ on $Z$ commute;
\item $s_Z:Z\to H^{(0)}$ is a local homeomorphism;
\item the right $H$-action is free and proper.
\end{enumerate}
The correspondence is said to be proper if $r_Z$ induces a proper map
\[
r_*:Z/H\longrightarrow G^{(0)}.
\]
\end{defn}

We refer to properness of $r_*$ as \emph{orbit properness}, to distinguish it from properness of the
left $G$-action.

Note that since $H^{(0)}$ is locally compact Hausdorff, and $s_Z$ is a local homeomorphism, $Z$ is locally compact and locally Hausdorff.
Also, since the $H$-action is free and proper, the orbit space $Z/H$ is Hausdorff \cite[Proposition 2.16]{AKM}.
Let $q:Z\to Z/H$ denote the quotient map.

\begin{defn} \cite[Definition 7.2]{AKM}
A slice of a groupoid correspondence $Z$ from $G$ to $H$ is an open subset $V\subseteq Z$ such that $s_Z|_V:V\to H^{(0)}$ and $q|_V:V\to Z/H$ are injective.
\end{defn}

If $Z$ is a groupoid correspondence from $G$ to $H$, then $s_Z:Z\to H^{(0)}$ is a local homeomorphism, and so is $q:Z\to Z/H$ by \cite[Lemma 2.10]{AKM}, so any point in $Z$ has an open neighborhood that is a slice.
By \cite[Proposition 7.1]{AKM}, we may write any element in $\mathcal{C}(Z)$ as a finite sum $\sum_{i=1}^nf_i$, where $f_i\in C_c(V_i)$ and each $V_i$ is a slice.

\begin{defn} (\cite[Definition 5.7]{Par09})
Let $A$ and $B$ be nondegenerate Banach algebras. A Morita cycle from $A$ to $B$ is a nondegenerate Banach $(A,B)$-correspondence $((\X,\Y), \pi_A)$ such that $\pi_A(A)\subseteq\mathcal{K}_B(\X,\Y)$.
\end{defn}

The significance of Morita cycles from $A$ to $B$ is that they represent classes in $KK^{ban}(A,B)$, and act on the right of $KK^{ban}(C,A)$ for any nondegenerate Banach algebra $C$ (see \cite[Section 5]{Par09} for details).
Moreover, Morita equivalences are Morita cycles (cf. \cite[Proposition 5.21]{Par09}).

\subsection{Full Morita cycles from proper correspondences}

Assume in this subsection that $G^{(0)}$ and $H^{(0)}$ are
paracompact.

\begin{lem}[The imprimitivity groupoid]
\label{lem:correspondence-imprimitivity-groupoid}
Let $Z$ be a groupoid correspondence from $G$ to $H$, and put
\[
   Q=Z/H.
\]
There is a locally compact, locally Hausdorff, \'etale groupoid
\[
   K_Z
   =\bigl(Z*_s Z\bigr)/H
\]
with unit space naturally identified with $Q$.  Writing $[z,w]$ for the
class of $(z,w)$ under the diagonal right $H$-action, its range, source,
and inverse maps are
\[
   r([z,w])=q(z),\qquad
   s([z,w])=q(w),\qquad
   [z,w]^{-1}=[w,z].
\]
With its original right $H$-action and the left action determined by
\[
   [z,w]\cdot(wh)=zh,
\]
the space $Z$ is a $K_Z$--$H$ equivalence.
\end{lem}

\begin{proof}
This is the standard imprimitivity groupoid construction associated with a free and proper right $H$-space; see
\cite[Theorem~3.5(1)]{MuhlyWilliamsImprimitivity}.
For background on principal groupoid spaces in the locally compact,
locally Hausdorff setting used here, see also
\cite[Section~2]{MW}. We only record why the construction is \'etale in
the present, possibly non-Hausdorff, setting.

Let $z,w\in Z$ satisfy $s_Z(z)=s_Z(w)$.  Choose slices
$U_0,V_0\subseteq Z$ containing $z,w$, respectively.  Replacing them
by
\[
 U=U_0\cap s_Z^{-1}(W),
 \qquad
 V=V_0\cap s_Z^{-1}(W),
 \qquad
 W=s_Z(U_0)\cap s_Z(V_0),
\]
we may assume that
\[
 s_Z(U)=s_Z(V)=W.
\]
Then
\[
 [U,V]
 \coloneqq
 \{[u,v]:u\in U,\ v\in V,\ s_Z(u)=s_Z(v)\}
\]
is open in $K_Z$: indeed,
$U*_{s_Z}V$ is open in $P \coloneqq Z*_{s_Z}Z$, and the orbit map for the
diagonal right $H$-action is open since $H$ is
\'etale and the action map $P*H\to P$ is a local homeomorphism
\cite[Lemma~2.9]{AKM}. The source and range maps restrict to homeomorphisms
\[
 s:[U,V]\longrightarrow q(V),
 \qquad
 r:[U,V]\longrightarrow q(U).
\]
Thus the sets $[U,V]$ form a basis of open bisections of $K_Z$.
Hence $K_Z$ is \'etale.  Since each $[U,V]$ is homeomorphic to an
open subset of the locally compact Hausdorff space $Q=Z/H$, it is
also locally compact and locally Hausdorff.
The remaining assertions are the usual properties of the
imprimitivity groupoid construction.
\end{proof}

The left $G$-action on $Z$ descends to an action on $Q=Z/H$,
\[
 \gamma\cdot q(z) \coloneqq q(\gamma z),
\]
with anchor $r_*:Q\to G^{(0)}$.  For
$x=q(z)\in Q$ and $\gamma\in G$ with $s(\gamma)=r_*(x)$, define
\[
 \beta_Z(\gamma,x) \coloneqq [\gamma z,z].
\]
This is independent of the representative $z$ of $x$: if $z'=zh$,
then, since the two actions commute,
\[
 [\gamma z',z']
 =[(\gamma z)h,zh]
 =[\gamma z,z].
\]

The map $\beta_Z$ is compatible with multiplication.  Namely, if
$x\in Q$, $\gamma\in G$ satisfies
\[
    s(\gamma)=r_*(x),
\]
and $\alpha\in G$ satisfies $s(\alpha)=r(\gamma)$, then
\begin{equation}
\label{eq:betaZ-multiplication}
    \beta_Z(\alpha,\gamma\cdot x)\,
    \beta_Z(\gamma,x)
    =
    \beta_Z(\alpha\gamma,x).
\end{equation}
Indeed, if $x=q(z)$, then
\[
    [\alpha\gamma z,\gamma z]\,[\gamma z,z]
    =
    [\alpha\gamma z,z].
\]

\begin{prop}[The homomorphism associated with a proper correspondence]
\label{prop:ThetaZ}
Let $Z$ be a proper groupoid correspondence from $G$ to $H$.  For
$a\in\mathcal C(G)$ and $[z,w]\in K_Z$, define
\begin{equation}
\label{eq:ThetaZ-definition}
   \Theta_Z(a)([z,w])
   \coloneqq
   \sum_{\substack{\gamma\in G\\ \gamma w=z}} a(\gamma).
\end{equation}
Then the sum is finite,
\[
   \Theta_Z(a)\in\mathcal C(K_Z),
\]
and
\[
   \Theta_Z:\mathcal C(G)\longrightarrow\mathcal C(K_Z)
\]
is an algebra homomorphism satisfying
\[
   \|\Theta_Z(a)\|_{I,s}\leq \|a\|_{I,s},
   \qquad
   \|\Theta_Z(a)\|_{I,r}\leq \|a\|_{I,r}.
\]
Consequently, for every $p\in[1,\infty]$, it extends uniquely to a
contractive homomorphism
\[
   \Theta_Z:F^p(G)\longrightarrow F^p(K_Z).
\]
\end{prop}

\begin{proof}
First note that the definition is independent of the representative of
$[z,w]$.  Indeed, replacing $(z,w)$ by $(zh,wh)$ does not change the
condition $\gamma w=z$, since the left and right actions commute:
\[
   \gamma(wh)=(\gamma w)h.
\]
The sum in \eqref{eq:ThetaZ-definition} is finite.  To see this, write
$a$ as a finite sum of functions supported in open bisections of $G$.
For each such bisection and each $w\in Z$, there is at most one
$\gamma$ in the bisection with $s(\gamma)=r_Z(w)$.

To prove that $\Theta_Z(a)\in\mathcal C(K_Z)$,  by the bisection
decomposition lemma from Section~\ref{sec:full-Lp}, it suffices to
suppose that $a\in C_c(A)$ for an open bisection $A\subseteq G$.  Put
\[
   D_A=r_*^{-1}(s(A))\subseteq Q.
\]
For $x\in D_A$, let
\[
   \gamma_A(x)=(s|_A)^{-1}(r_*(x))
\]
and define
\[
   \tau_A(x)=\beta_Z(\gamma_A(x),x).
\]
Set
\[
   \beta_A=\tau_A(D_A)
   =
   \{\beta_Z(\gamma,x):
       \gamma\in A,\ s(\gamma)=r_*(x)\}.
\]

We claim that $\beta_A$ is an open bisection of $K_Z$.  Fix
$x=q(z)\in D_A$.  Choose a slice $V\subseteq Z$ containing $z$ and,
after replacing $V$ by
$V\cap r_Z^{-1}(s(A))$, assume that
\[
   r_Z(V)\subseteq s(A).
\]
The bisection $A$ acts by a homeomorphism
\[
   r_Z^{-1}(s(A))
   \longrightarrow
   r_Z^{-1}(r(A)),
   \qquad
   v\longmapsto \gamma_v v,
\]
where $\gamma_v$ is the unique element of $A$ satisfying
$s(\gamma_v)=r_Z(v)$.  Put
\[
   U=A\cdot V=\{\gamma_vv:v\in V\}.
\]
Then $U$ is again a slice.  Indeed, the $G$-action preserves $s_Z$, so
injectivity of $s_Z|_V$ gives injectivity of $s_Z|_U$.  If
$q(\gamma_{v_1}v_1)=q(\gamma_{v_2}v_2)$, then applying $r_*$ gives
\[
   r(\gamma_{v_1})=r(\gamma_{v_2}),
\]
and hence $\gamma_{v_1}=\gamma_{v_2}$ by injectivity of $r|_A$.
Applying the inverse action then gives $q(v_1)=q(v_2)$, so $v_1=v_2$
by injectivity of $q|_V$.  Thus $q|_U$ is also injective.  Moreover,
\[
   s_Z(U)=s_Z(V).
\]

By Lemma~\ref{lem:correspondence-imprimitivity-groupoid},
\[
   [U,V]
   =
   \{[u,v]:u\in U,\ v\in V,\ s_Z(u)=s_Z(v)\}
\]
is an open bisection of $K_Z$.  It contains
$\tau_A(x)=[\gamma_A(x)z,z]$, and in fact
\[
   [U,V]\subseteq\beta_A.
\]
Indeed, if $[u,v]\in[U,V]$, write $u=\gamma_{v_0}v_0$ with
$v_0\in V$.  Since
\[
   s_Z(v_0)=s_Z(u)=s_Z(v),
\]
injectivity of $s_Z|_V$ gives $v_0=v$, and hence
\[
   [u,v]=[\gamma_vv,v]
   =\beta_Z(\gamma_v,q(v)).
\]
Thus every point of $\beta_A$ has an open neighborhood contained in
$\beta_A$, so $\beta_A$ is open.

The source map restricts to a bijection
\[
   s|_{\beta_A}:\beta_A\longrightarrow D_A,
   \qquad
   s(\tau_A(x))=x.
\]
Since $K_Z$ is \'{e}tale, this restriction is a homeomorphism, with
inverse $\tau_A$.  The range map is also injective on $\beta_A$:
if
\[
   r(\tau_A(x_1))=r(\tau_A(x_2)),
\]
then, writing $\gamma_i=\gamma_A(x_i)$, we have
\[
   \gamma_1\cdot x_1=\gamma_2\cdot x_2.
\]
Applying $r_*$ gives $r(\gamma_1)=r(\gamma_2)$, so
$\gamma_1=\gamma_2$ because $A$ is a bisection, and applying the
inverse action gives $x_1=x_2$.  Hence $\beta_A$ is an open
bisection.

Choose $z\in Z$ with $q(z)=x$.  Then
\[
   \tau_A(x)=[\gamma_A(x)z,z],
\]
so \eqref{eq:ThetaZ-definition} gives
\[
   \Theta_Z(a)(\tau_A(x))
   =
   \sum_{\delta z=\gamma_A(x)z}a(\delta).
\]
Since $a$ is supported in $A$, only $\delta\in A$ can contribute.
For such a $\delta$,
\[
   s(\delta)=r_Z(z)=s(\gamma_A(x)),
\]
and injectivity of $s|_A$ therefore gives
$\delta=\gamma_A(x)$.  Thus
\[
   \Theta_Z(a)(\tau_A(x))
   =a(\gamma_A(x)).
\]
Since $\tau_A$ is a homeomorphism and $\gamma_A$ is continuous,
$\Theta_Z(a)$ is continuous on $\beta_A$, and it vanishes outside
$\beta_A$.  Moreover,
\[
   \supp(\Theta_Z(a))
   \subseteq
   \tau_A\!\left(
      r_*^{-1}\bigl(s(\supp a)\bigr)
   \right).
\]
The set $s(\supp a)$ is compact, and $r_*$ is proper, so the set on
the right is compact.  Hence
\[
   \Theta_Z(a)\in C_c(\beta_A)
   \subseteq\mathcal C(K_Z).
\]
A finite bisection decomposition now gives
$\Theta_Z(a)\in\mathcal C(K_Z)$ for arbitrary
$a\in\mathcal C(G)$.

For multiplicativity, it is convenient to rewrite
\eqref{eq:ThetaZ-definition} intrinsically as
\[
   \Theta_Z(a)(k)
   =
   \sum_{\beta_Z(\gamma,s(k))=k}a(\gamma).
\]
Indeed, if $k=[z,w]$, then
\[
   \beta_Z(\gamma,s(k))=k
   \quad\Longleftrightarrow\quad
   [\gamma w,w]=[z,w]
   \quad\Longleftrightarrow\quad
   \gamma w=z,
\]
where the last equivalence uses freeness of the right $H$-action.

Fix $k\in K_Z$ and put $x=s(k)$.  On expanding
$(\Theta_Z(a)*\Theta_Z(b))(k)$, a summand is determined by composable
$\alpha,\gamma\in G$ such that
\[
   \beta_Z(\gamma,x)=k_2,
   \qquad
   \beta_Z(\alpha,\gamma\cdot x)=k_1,
   \qquad
   k_1k_2=k.
\]
By \eqref{eq:betaZ-multiplication}, this is equivalent to
\[
   \beta_Z(\alpha\gamma,x)=k.
\]
Conversely, every such pair $(\alpha,\gamma)$ determines the
corresponding factorization of $k$.  Therefore
\[
\begin{aligned}
   (\Theta_Z(a)*\Theta_Z(b))(k)
   &=
   \sum_{\beta_Z(\alpha\gamma,x)=k}
      a(\alpha)b(\gamma)\\
   &=
   \sum_{\beta_Z(\delta,x)=k}
      \sum_{\alpha\gamma=\delta}
      a(\alpha)b(\gamma)\\
   &=
   \sum_{\beta_Z(\delta,x)=k}
      (a*b)(\delta)\\
   &=
   \Theta_Z(a*b)(k).
\end{aligned}
\]
Thus $\Theta_Z$ is an algebra homomorphism.

For the $I$-norm estimates, fix $x\in Q$.  The map
\[
   \{\gamma\in G:s(\gamma)=r_*(x)\}
   \longrightarrow
   \{k\in K_Z:s(k)=x\},
   \qquad
   \gamma\longmapsto\beta_Z(\gamma,x),
\]
groups the arrows in the source fibre of $G$ according to the arrow
of $K_Z$ that they induce.  Hence
\[
\begin{aligned}
   \sum_{s(k)=x}|\Theta_Z(a)(k)|
   &=
   \sum_{s(k)=x}
      \left|
         \sum_{\beta_Z(\gamma,x)=k}a(\gamma)
      \right|\\
   &\leq
   \sum_{s(\gamma)=r_*(x)}|a(\gamma)|.
\end{aligned}
\]
Taking the supremum over $x\in Q$ gives
\[
   \|\Theta_Z(a)\|_{I,s}
   \leq
   \|a\|_{I,s}.
\]

Similarly, for $y\in Q$, arrows with
$r(\gamma)=r_*(y)$ are grouped by
\[
   \gamma
   \longmapsto
   \beta_Z(\gamma,\gamma^{-1}\cdot y),
\]
which has range $y$.  Thus
\[
   \sum_{r(k)=y}|\Theta_Z(a)(k)|
   \leq
   \sum_{r(\gamma)=r_*(y)}|a(\gamma)|,
\]
and therefore
\[
   \|\Theta_Z(a)\|_{I,r}
   \leq
   \|a\|_{I,r}.
\]

Finally, let
\[
   \rho:\mathcal C(K_Z)\longrightarrow B(L^p(\mu))
\]
be any $\|\cdot\|_I$-contractive representation.  By
the preceding $I$-norm estimates, the composition
$\rho\circ\Theta_Z$ is an $\|\cdot\|_I$-contractive representation of
$\mathcal C(G)$.  Hence, by the definition of the full norm,
\[
   \|\rho(\Theta_Z(a))\|
   \leq
   \|a\|_{F^p(G)}.
\]
Taking the supremum over all such $\rho$ gives
\[
   \|\Theta_Z(a)\|_{F^p(K_Z)}
   \leq
   \|a\|_{F^p(G)}.
\]
Thus $\Theta_Z$ extends uniquely to the contractive homomorphism
\[
   \Theta_Z:F^p(G)\longrightarrow F^p(K_Z).
\]
\end{proof}

The algebra map $\Theta_Z$ recovers exactly the original left action on
$Z$.  Indeed, for $a\in\mathcal C(G)$, $\varphi\in\mathcal C(Z)$ and
$z\in Z$,
\begin{equation}
\label{eq:ThetaZ-module-action}
\begin{aligned}
   (\Theta_Z(a)*\varphi)(z)
   &=\sum_{w\in Z_{s_Z(z)}} \Theta_Z(a)([z,w])\varphi(w)\\
   &=\sum_{w\in Z_{s_Z(z)}}\left(\sum_{\gamma w=z}a(\gamma)\right)\varphi(w)\\
   &=\sum_{r(\gamma)=r_Z(z)}
        a(\gamma)\varphi(\gamma^{-1}z)
    =(a\cdot\varphi)(z).
\end{aligned}
\end{equation}
Here, as usual, the sums are finite after bisection decomposition.  The
corresponding identity on $\mathcal C(Z^{\mathrm{op}})$ follows in the same
way.

\begin{thm}[Proper correspondences give full Morita cycles]
Let $G$ and $H$ be locally compact, locally Hausdorff, \'etale groupoids
with paracompact unit spaces, let $p\in[1,\infty]$, and let $Z$ be a
proper groupoid correspondence from $G$ to $H$.  Then $Z$ canonically
determines a Morita cycle from $F^p(G)$ to $F^p(H)$.
\end{thm}

\begin{proof}
Put $Q=Z/H$.  Since $Q$ is locally compact Hausdorff and
$r_*:Q\to G^{(0)}$ is proper, its image $r_*(Q)$ is closed in
$G^{(0)}$, and the corestriction
\[
    r_*:Q\longrightarrow r_*(Q)
\]
is a perfect map.  The space $r_*(Q)$ is paracompact, being a closed
subspace of the paracompact Hausdorff space $G^{(0)}$.  It follows from
\cite[Theorem~5.1.35]{EngelkingGeneralTopology} that $Q$ is paracompact.

By Lemma~\ref{lem:correspondence-imprimitivity-groupoid}, $Z$ is a
$K_Z$--$H$ equivalence, with paracompact unit spaces $Q$ and $H^{(0)}$.
Corollary~\ref{cor:full-morita} therefore gives a Morita equivalence
between $F^p(K_Z)$ and $F^p(H)$.  More concretely, if
\[
   L_Z=K_Z\sqcup Z\sqcup Z^{\mathrm{op}}\sqcup H
\]
is the linking groupoid of this equivalence, set
\[
   \X_Z^{\mathrm{full}}
      =\overline{\mathcal C(Z^{\mathrm{op}})}^{\,F^p(L_Z)},
   \qquad
   \Y_Z^{\mathrm{full}}
      =\overline{\mathcal C(Z)}^{\,F^p(L_Z)}.
\]
The resulting $F^p(K_Z)$--$F^p(H)$ Morita equivalence is, in particular,
a Morita cycle.  Compose its left action with the contractive homomorphism
from Proposition~\ref{prop:ThetaZ}:
\[
   F^p(G)
   \xrightarrow{\ \Theta_Z\ }
   F^p(K_Z)
   \longrightarrow
   \mathcal K_{F^p(H)}
      (\X_Z^{\mathrm{full}},\Y_Z^{\mathrm{full}}).
\]
By \eqref{eq:ThetaZ-module-action}, this composite agrees on the dense
algebraic corner $\mathcal C(Z)$ with the original left
$\mathcal C(G)$-action associated with the groupoid correspondence.

It remains only to verify nondegeneracy of the composite action.
Every $\varphi\in\mathcal C(Z)$ vanishes outside some compact subset
$C\subseteq Z$.  Choose $e\in C_c(G^{(0)})$ such that
\[
    e=1\quad\text{on }r_Z(C).
\]
Then
\[
    (e\cdot\varphi)(z)
      =e(r_Z(z))\varphi(z)
      =\varphi(z),
\]
and hence $\mathcal C(Z)\subseteq
F^p(G)\Y_Z^{\mathrm{full}}$.  Since $\mathcal C(Z)$ is dense in
$\Y_Z^{\mathrm{full}}$, the left action is nondegenerate on
$\Y_Z^{\mathrm{full}}$.  The same argument on
$\mathcal C(Z^{\mathrm{op}})$ gives the corresponding
nondegeneracy on $\X_Z^{\mathrm{full}}$.
\end{proof}

\begin{rem}[The equivalence case]
If $Z$ is a $G$--$H$ equivalence, then the preceding construction reduces
exactly to the full Morita equivalence of
Corollary~\ref{cor:full-morita}.  Indeed, the bracket map for the left
$G$-action gives a canonical groupoid isomorphism
\[
   K_Z\longrightarrow G,
   \qquad
   [z,w]\longmapsto {}_G[z,w],
\]
where ${}_G[z,w]$ is the unique element of $G$ satisfying
${}_G[z,w]\,w=z$; cf.~\cite[Lemma~2.20]{MW}. Under this identification, freeness of the left action
reduces \eqref{eq:ThetaZ-definition} to a single term, so that
$\Theta_Z$ becomes the identity map on $F^p(G)$.  The linking groupoid
$K_Z\sqcup Z\sqcup Z^{\mathrm{op}}\sqcup H$ therefore identifies with the
usual linking groupoid
$G\sqcup Z\sqcup Z^{\mathrm{op}}\sqcup H$.
\end{rem}

\subsection{The associated reduced Banach pair and left action} \label{sect:BanCorresp}

 We now turn to the reduced $L^p$-operator algebras.  Fix
 $p\in[1,\infty]$ and put
 \[
    B=F^p_{\mathrm{red}}(H).
 \]
 The algebraic module actions are those defined in
Section~\ref{sect_LinkGp}.  These formulas only use the given actions of
 $G$ and $H$ on $Z$ and the fact that they commute.  We first construct a
 canonical Banach $B$-pair.  The left $\mathcal C(G)$-action will be added
 afterwards; it gives a Banach
 $F^p_{\mathrm{red}}(G)$--$B$ correspondence precisely when it is
 contractive for the reduced norm.

We begin with the $B$-valued pairing.  The formula used in the equivalence
setting cannot, in general, be used for a groupoid correspondence.  Recall
that, for a groupoid equivalence, $\psi\in\mathcal{C}(Z^{\mathrm{op}})$,
$\phi\in\mathcal{C}(Z)$ and $\eta\in H$, the pairing can be written as
\begin{equation}
  \langle\psi,\phi\rangle_H(\eta)
  =\sum_{r(\gamma)=r_Z(z)}
    \psi(\overline{\gamma^{-1}\cdot z})
    \phi(\gamma^{-1}\cdot z\cdot \eta),
  \label{eq:equivalence-pairing}
\end{equation}
where $z\in Z$ is any point satisfying $s_Z(z)=r(\eta)$.  Let $v=s_Z(z)$.
For a groupoid equivalence, every fibre $Z_v=s_Z^{-1}(v)$ is a single $G$-orbit and the
left $G$-action is free.  Thus
\[
  G^{r_Z(z)}\longrightarrow Z_v,
  \qquad \gamma\longmapsto\gamma^{-1}\cdot z,
\]
is a bijection.  Making the change of variables $w=\gamma^{-1}\cdot z$ in
\eqref{eq:equivalence-pairing} gives
\begin{equation}
  \langle\psi,\phi\rangle_H(\eta)
  =\sum_{s_Z(w)=r(\eta)}\psi(\overline w)\phi(w\cdot\eta).
  \label{eq:correspondence-pairing}
\end{equation}

For a general groupoid correspondence, the orbit map need not be
bijective, so \eqref{eq:equivalence-pairing} need not be independent
of the choice of $z$.  For example, if $G$ and $H$ are both the one-point unit
groupoid and $Z=\{z_1,z_2\}$, the two choices give
\[
\psi(\overline{z_1})\phi(z_1)
\quad\text{and}\quad
\psi(\overline{z_2})\phi(z_2).
\]  
We therefore use
\eqref{eq:correspondence-pairing} as the definition of the pairing for an arbitrary
groupoid correspondence.

For $v\in H^{(0)}$, put
\[
  H_v\coloneqq s^{-1}(v),
  \qquad
  Z_v\coloneqq s_Z^{-1}(v),
\]
and define
\[
  E_H^p\coloneqq\bigoplus_{v\in H^{(0)}}^p\ell^p(H_v),
  \qquad
  E_Z^p\coloneqq\bigoplus_{v\in H^{(0)}}^p\ell^p(Z_v).
\]
For $p=\infty$, these are the corresponding $\ell^\infty$-direct sums.
Each $Z_v$ is discrete because $s_Z$ is a local homeomorphism.
The space $E_H^p$ carries the direct sum $\lambda_H^p$ of the regular
representations of $H$.

For $\phi\in\mathcal{C}(Z)$, define
$T_\phi:E_H^p\to E_Z^p$ by
\[
  (T_\phi\xi)(z)
  =\sum_{h\in H_{s_Z(z)}}\phi(z\cdot h^{-1})\xi(h).
\]
Equivalently, this is
\[
  (T_\phi\xi)(z)
  =\sum_{r(\eta)=s_Z(z)}\phi(z\cdot \eta)\xi(\eta^{-1}).
\]
For $\psi\in\mathcal{C}(Z^{\mathrm{op}})$, define
$S_\psi:E_Z^p\to E_H^p$ by
\[
  (S_\psi\zeta)(h)
  =\sum_{z\in Z_{s(h)}}
    \psi(\overline{z\cdot h^{-1}})\zeta(z).
\]

\begin{lem} \label{lem:slicenorm}
The operators $T_\phi$ and $S_\psi$ defined above are bounded.  If
$V\subseteq Z$ is a slice and $\phi\in C_c(V)$, extended by zero
outside $V$, then
\[
  \|T_\phi\|=\|\phi\|_\infty.
\]
If $\psi\in C_c(V^{\mathrm{op}})$, then
\[
  \|S_\psi\|=\|\psi\|_\infty.
\]
These statements hold for every $p\in[1,\infty]$.
\end{lem}

\begin{proof}
Relative to the canonical bases, the matrix coefficient of $T_\phi$
in the row indexed by $z\in Z_v$ and the column indexed by $h\in H_v$
is $\phi(zh^{-1})$.  There is at most one nonzero coefficient in each
row.  Indeed, if $zh_1^{-1},zh_2^{-1}\in V$, then these two points have
the same image in $Z/H$.  Injectivity of the quotient map on $V$ gives
$zh_1^{-1}=zh_2^{-1}$, and freeness of the right $H$-action gives
$h_1=h_2$.  There is also at most one nonzero coefficient in each column:
if $z_1h^{-1},z_2h^{-1}\in V$, then these points have the same
$s_Z$-value, so injectivity of $s_Z|_V$ gives $z_1h^{-1}=z_2h^{-1}$ and
hence $z_1=z_2$.  Thus $T_\phi$ is a weighted partial permutation.
Its norm on every $\ell^p$-space, including $p=1$ and $p=\infty$, is the
supremum of the absolute values of its matrix coefficients.  Every value
$\phi(w)$ occurs as a coefficient by taking $z=w$ and
$h=s_Z(w)\in H^{(0)}$.  Therefore
$\|T_\phi\|=\|\phi\|_\infty$.

The proof for $S_\psi$ is the same.  Since every element of $\mathcal{C}(Z)$
and $\mathcal{C}(Z^{\mathrm{op}})$ is a finite sum of functions supported in compact subsets of slices, the general operators are finite sums of bounded operators.  
\end{proof}

The maps $\phi\mapsto T_\phi$ and $\psi\mapsto S_\psi$ are
injective.  For example, if $v=s_Z(z)$ and $\delta_v$ is the point mass at
the unit $v\in H_v$, then
\[
  (T_\phi\delta_v)(z)=\phi(z),
\]
and the analogous observation applies to $S_\psi$.  We may therefore
define norms by
\[
  \|\phi\|_{\Y_Z}\coloneqq\|T_\phi\|,
  \qquad
  \|\psi\|_{\X_Z}\coloneqq\|S_\psi\|,
\]
and let $\Y_Z$ and $\X_Z$ be the respective completions.  Equivalently,
\[
  \Y_Z=\overline{\{T_\phi:\phi\in\mathcal{C}(Z)\}}
  \subseteq\mathcal{B}(E_H^p,E_Z^p),
\]
\[
  \X_Z=\overline{\{S_\psi:\psi\in\mathcal{C}(Z^{\mathrm{op}})\}}
  \subseteq\mathcal{B}(E_Z^p,E_H^p).
\]

\begin{prop}
The $\mathcal C(H)$-actions on $\mathcal C(Z^{\mathrm{op}})$ and $\mathcal C(Z)$ defined in Section~\ref{sect_LinkGp} extend to make $(\X_Z,\Y_Z)$ a nondegenerate Banach $F_{\mathrm{red}}^p(H)$-pair.  On the dense subspaces, its pairing is given by \eqref{eq:correspondence-pairing}; equivalently,
\[
  \langle S_\psi,T_\phi\rangle_B
  =S_\psi T_\phi
  =\lambda_H^p(\langle\psi,\phi\rangle_H).
\]
\end{prop}

\begin{proof}
By finite slice
decompositions, it suffices to take
$\psi\in C_c(U^{\mathrm{op}})$ and $\phi\in C_c(V)$ for slices
$U,V\subseteq Z$.  Put $Q=Z/H$ and
\[
  D_{U,V}\coloneqq
  \{(w,y)\in U\times V:q(w)=q(y)\},
\]
where $q:Z\to Q$ is the quotient map.
Write $\check\psi(w)\coloneqq\psi(\overline w)$ for $w\in U$.
For $(w,y)\in D_{U,V}$, there is a unique element
$[w,y]_H\in H$ such that
\[
  w[w,y]_H=y.
\]
The bracket map is continuous because the right $H$-action is principal;
equivalently, the map
\[
  Z*H\longrightarrow Z\times_Q Z,
  \qquad (w,h)\longmapsto(w,wh),
\]
is a homeomorphism.  The restriction
\[
  \vartheta_{U,V}:D_{U,V}\longrightarrow H,
  \qquad (w,y)\longmapsto[w,y]_H,
\]
is a homeomorphism onto an open Hausdorff subset $O_{U,V}$ of $H$.  To
see the openness explicitly, use that $s_Z|_U$ is a homeomorphism onto an
open subset of $H^{(0)}$.  For $h$ with $r(h)\in s_Z(U)$, put
\[
  w(h)=(s_Z|_U)^{-1}(r(h)).
\]
Then
\[
  O_{U,V}
  =\{h:r(h)\in s_Z(U),\ w(h)h\in V\},
\]
which is open, and $h\mapsto(w(h),w(h)h)$ is the inverse of
$\vartheta_{U,V}$.  The set $D_{U,V}$ is Hausdorff because it is a
subspace of $U\times V$, so $O_{U,V}$ is Hausdorff as well.

On $O_{U,V}$ the function in \eqref{eq:correspondence-pairing} is
\[
  h\longmapsto
  \psi(\overline{w(h)})\phi(w(h)h),
\]
and there is at most one nonzero summand.  It is continuous there, and it vanishes outside the compact set
\[
  \vartheta_{U,V}
  \bigl(D_{U,V}\cap
    (\mathrm{supp}_U\check\psi\times\mathrm{supp}_V\phi)\bigr).
\]
Indeed, $Q$ is Hausdorff, so the set inside the parentheses is a closed
subset of a compact Hausdorff space.  Hence the pairing is an element of
$C_c(O_{U,V})$, extended by zero in the sense used to define
$\mathcal{C}(H)$.  This proves that
$\langle\psi,\phi\rangle_H\in\mathcal{C}(H)$ in general.

Let $\lambda_H^p$ be the direct sum of the regular representations, so
that
\[
  (\lambda_H^p(f)\xi)(h)
  =\sum_{r(\eta)=r(h)}f(\eta)\xi(\eta^{-1}\cdot h).
\]
For $h\in H_v$, a direct calculation gives
\begin{align*}
  (S_\psi T_\phi\xi)(h)
  &=\sum_{z\in Z_v}\sum_{k\in H_v}
    \psi(\overline{z\cdot h^{-1}})\phi(z\cdot k^{-1})\xi(k)\\
  &=\sum_{r(\eta)=r(h)}
    \left(
      \sum_{s_Z(w)=r(\eta)}
        \psi(\overline w)\phi(w\cdot\eta)
    \right)
    \xi(\eta^{-1}\cdot h)\\
  &=\bigl(
      \lambda_H^p(\langle\psi,\phi\rangle_H)\xi
    \bigr)(h).
\end{align*}
Here the second equality uses the bijective changes of variables
\[
  w=z\cdot h^{-1},
  \qquad
  \eta=hk^{-1}.
\]
All sums involved are finite after decomposing $\psi$ and $\phi$ into functions with compact support contained in slices, so the calculation is valid also for
$p=\infty$.  

The module formulas give, for $b\in\mathcal{C}(H)$,
\[
  T_{\phi\cdot b}=T_\phi\lambda_H^p(b),
  \qquad
  S_{b\cdot\psi}=\lambda_H^p(b)S_\psi.
\]
These identities follow either by direct changes of variables in the
finite sums or from associativity of the right $H$-action.  Since
$\lambda_H^p$ is the defining isometric representation of
$F_{\mathrm{red}}^p(H)$, these identities extend the
actions to the completions.  If $S\in \X_Z$ and $T\in \Y_Z$, then $ST$ is
an operator-norm limit of operators of the form
$\lambda_H^p(\langle\psi,\phi\rangle_H)$, and therefore belongs to
$B=F_{\mathrm{red}}^p(H)$.  Thus
\[
  \langle S,T\rangle_B\coloneqq ST
\]
defines the completed pairing, and
\[
  \|\langle S,T\rangle_B\|\leq\|S\|\,\|T\|.
\]
The remaining Banach-pair compatibility identities follow from
associativity of operator composition.

Finally, the pair is nondegenerate.  If $\phi$ vanishes outside a compact
set $K\subseteq Z$, choose $u\in C_c(H^{(0)})$ with
$u=1$ on $s_Z(K)$.  Then
\[
  (\phi\cdot u)(z)
  =\phi(z)u(s_Z(z)),
\]
so $\phi\cdot u=\phi$.  The same argument gives
$u\cdot\psi=\psi$ for every $\psi\in\mathcal{C}(Z^{\mathrm{op}})$ vanishing outside $K$.  Thus the dense algebraic subspaces already lie in
$\Y_Z\mathcal{C}(H)$ and $\mathcal{C}(H)\X_Z$, respectively, which proves
nondegeneracy after completion.
\end{proof}

For the left action, define
\begin{equation}
  (\Pi_Z(a)\zeta)(z)
  =\sum_{r(\gamma)=r_Z(z)}
    a(\gamma)\zeta(\gamma^{-1}\cdot z).
  \label{eq:PiZ}
\end{equation}
The commuting actions of $G$ and $H$ imply the identities
\[
\Pi_Z(a)T_\phi=T_{a\cdot\phi},
\qquad
S_\psi\Pi_Z(a)=S_{\psi\cdot a}.
\]
For $p<\infty$ these may first be checked on finitely supported
vectors and then extended by continuity; for $p=\infty$ the same
finite sum calculation applies directly to arbitrary bounded vectors.

These identities show that the two maps are module maps and satisfy the
formal adjoint identity.
Consequently,
\[
  \langle S_{\psi\cdot a},T_\phi\rangle_B
  =S_\psi\Pi_Z(a)T_\phi
  =\langle S_\psi,T_{a\cdot\phi}\rangle_B.
\]

Suppose now that $\Pi_Z(a)$ extends to a bounded operator on $E_Z^p$
and satisfies
\begin{equation}
  \|\Pi_Z(a)\|
  \leq \|a\|_{F_{\mathrm{red}}^p(G)}
  \qquad(a\in\mathcal{C}(G)).
  \label{eq:reduced-left-bound}
\end{equation}
Then the identities above and
\eqref{eq:reduced-left-bound} extend the algebraic action to a
contractive homomorphism
\[
  \pi_Z:F_{\mathrm{red}}^p(G)
  \longrightarrow
  \mathcal{L}_{F_{\mathrm{red}}^p(H)}(\X_Z,\Y_Z).
\]
Thus $((\X_Z,\Y_Z),\pi_Z)$ becomes a Banach
$\bigl(F_{\mathrm{red}}^p(G),F_{\mathrm{red}}^p(H)\bigr)$-correspondence.
The action is nondegenerate.  Indeed, if $\phi$ vanishes outside a compact
set $K\subseteq Z$, choose $e\in C_c(G^{(0)})$ with
$e=1$ on $r_Z(K)$.  Then
\[
  (e\cdot\phi)(z)
  =e(r_Z(z))\phi(z),
\]
and hence $e\cdot\phi=\phi$; the analogous statement holds on
$\mathcal{C}(Z^{\mathrm{op}})$.  Density proves nondegeneracy on the completed
pair.  Proposition~\ref{prop:reduced-ext} below proves
\eqref{eq:reduced-left-bound} when $p=1$, when $p=\infty$, and when
$1<p<\infty$ and the left $G$-action on $Z$ is proper.  Before such a
bound is established, $(\X_Z,\Y_Z)$ is a Banach
$F_{\mathrm{red}}^p(H)$-pair equipped with a $\mathcal{C}(G)$-action, but it is not yet a Banach
$(F_{\mathrm{red}}^p(G),F_{\mathrm{red}}^p(H))$-correspondence.

\subsection{Compact multipliers}

Retain the notation $B,E_H^p,E_Z^p,\X_Z,\Y_Z$ from Section~\ref{sect:BanCorresp}.
For \(f\in C_b(Q)\), define multiplication operators on the dense
subspaces $\mathcal{C}(Z)$ and $\mathcal{C}(Z^{op})$ by
\[
    (M_f\cdot\phi)(z)
    =
    f(q(z))\phi(z),
    \qquad
    (\psi\cdot M_f)(\bar z)
    =
    \psi(\bar z)f(q(z)),
\]
where $q:Z\to Z/H$ is the quotient map. Since \(q(z)=q(zh)\), formula
\eqref{eq:correspondence-pairing} gives
\[
\begin{aligned}
    \langle\psi\cdot M_f,\phi\rangle_H(h)
    &=
    \sum_{z\in Z_{r(h)}}
    \psi(\bar z)f(q(z))\phi(zh)                                      \\
    &=
    \sum_{z\in Z_{r(h)}}
    \psi(\bar z)f(q(zh))\phi(zh)                                     \\
    &=
    \langle\psi,M_f\cdot\phi\rangle_H(h).
\end{aligned}
\]
Thus the two multiplication maps are formal adjoints. They are module maps,
and the concrete norm construction gives
\[
    \|M_f\cdot\phi\|_{\Y_Z}
    \leq
    \|f\|_\infty\|\phi\|_{\Y_Z},
\]
and
\[
    \|\psi\cdot M_f\|_{\X_Z}
    \leq
    \|f\|_\infty\|\psi\|_{\X_Z}.
\]
Consequently,
\[
    M_f\in\mathcal L_B(\X_Z,\Y_Z),
    \qquad
    \|M_f\|\leq\|f\|_\infty.
\]

\begin{prop} \label{prop:compact-multiplier-characterization}
For \(f\in C_b(Q)\), one has
\[
    M_f\in \mathcal K_B(\X_Z,\Y_Z)
    \quad\Longleftrightarrow\quad
    f\in C_0(Q).
\]
Moreover,
\[
    \|M_f\|=\|f\|_\infty.
\]
\end{prop}

\begin{proof}
For $x\in Q$, choose $z\in q^{-1}(x)$, and let $V\subseteq Z$ be a slice
containing $z$.  Choose $\phi\in C_c(V)$ such that
\[
\phi(z)=1
\qquad\text{and}\qquad
\|\phi\|_\infty=1.
\]
By Lemma~\ref{lem:slicenorm},
\[
\|\phi\|_{\Y_Z}=1
\quad\text{and}\quad
\|(f\circ q)\phi\|_{\Y_Z}
 =\|(f\circ q)\phi\|_\infty.
\]
Hence
\[
\|M_f\|
 \geq \|M_f\phi\|_{\Y_Z}
 =\|(f\circ q)\phi\|_\infty
 \geq |f(x)|.
\]
Taking the supremum over $x\in Q$, and using the contractive estimate
$\|M_f\|\leq\|f\|_\infty$ established above, gives
\[
\|M_f\|=\|f\|_\infty.
\]

First suppose that $f\in C_c(Q)$, and put $K=\mathrm{supp}(f)$.  Since slices
form a basis for $Z$ and $q$ is a local homeomorphism, choose
finitely many slices $V_1,\ldots,V_n$ such that, with
$U_i=q(V_i)$,
\[
    K\subseteq\bigcup_{i=1}^n U_i.
\]
Choose $\rho_i\in C_c(U_i)$ such that
\[
    \sum_{i=1}^n\rho_i=1\ \text{on }K,
\]
and put $f_i=f\rho_i$.  Then
\[
    f=\sum_{i=1}^n f_i,
    \qquad
    \mathrm{supp}(f_i)\subseteq U_i.
\]

For each $i$, choose $\chi_i\in C_c(U_i)$ with
$\chi_i=1$ on $\mathrm{supp}(f_i)$, and define
$\alpha_i,\beta_i\in C_c(V_i)$, extended by zero outside $V_i$, by
\[
    \alpha_i(z)=f_i(q(z)),
    \qquad
    \beta_i(z)=\chi_i(q(z))
    \qquad(z\in V_i).
\]
Let $\widetilde{\beta_i}\in \mathcal{C}(Z^{op})$ be defined by
\[
    \widetilde{\beta_i}(\bar z)=\beta_i(z).
\]

Let $m_{f_i}\in\mathcal B(E_Z^p)$ denote multiplication by
$f_i\circ q$.  We claim that
\[
    T_{\alpha_i}S_{\widetilde{\beta_i}}=m_{f_i}.
\]
Indeed, for an arbitrary $\zeta\in E_Z^p$ and $z\in Z$, the sums
below are finite by the slice conditions, and
\[
\begin{aligned}
    (T_{\alpha_i}S_{\widetilde{\beta_i}}\zeta)(z)
    &=
    \sum_{\substack{r(\eta)=s_Z(z)\\ s_Z(w)=s_Z(z)}}
       \alpha_i(z\cdot\eta)\beta_i(w\cdot\eta)\zeta(w).
\end{aligned}
\]
A nonzero summand requires $z\cdot\eta,w\cdot\eta\in V_i$.  These two
points have the same $s_Z$-value, so injectivity of
$s_Z|_{V_i}$ gives $z\cdot\eta=w\cdot\eta$, and freeness of the right
$H$-action then gives $z=w$.

For fixed $z$, there is at most one $\eta$ with $z\cdot\eta\in V_i$:
if both $z\cdot\eta_1$ and $z\cdot\eta_2$ belong to $V_i$, then they have
the same image under $q$, so injectivity of $q|_{V_i}$ gives
$z\cdot\eta_1=z\cdot\eta_2$, and freeness gives $\eta_1=\eta_2$.  Such an
$\eta$ exists precisely when $q(z)\in U_i$.  Consequently,
\[
    (T_{\alpha_i}S_{\widetilde{\beta_i}}\zeta)(z)
    =
    f_i(q(z))\chi_i(q(z))\zeta(z)
    =
    f_i(q(z))\zeta(z),
\]
which proves the claim.

This operator identity gives both components of the corresponding
rank-one operator.  On the dense subspace $\mathcal{C}(Z)\subseteq \Y_Z$,
\[
\begin{aligned}
    \theta_{\alpha_i,\widetilde{\beta_i}}^r(T_\xi)
    &=
    T_{\alpha_i}
       (S_{\widetilde{\beta_i}}T_\xi)             \\
    &=
    (T_{\alpha_i}S_{\widetilde{\beta_i}})T_\xi    \\
    &=
    m_{f_i}T_\xi
     =
    T_{f_i\cdot\xi}
     =
    M_{f_i}^r(T_\xi).
\end{aligned}
\]
Similarly, on $\mathcal{C}(Z^{op})\subseteq \X_Z$,
\[
\begin{aligned}
    \theta_{\alpha_i,\widetilde{\beta_i}}^l(S_\psi)
    &=
    (S_\psi T_{\alpha_i})S_{\widetilde{\beta_i}} \\
    &=
    S_\psi
      (T_{\alpha_i}S_{\widetilde{\beta_i}})       \\
    &=
    S_\psi m_{f_i}
     =
    M_{f_i}^l(S_\psi).
\end{aligned}
\]
By continuity,
\[
    M_{f_i}
    =
    \theta_{\alpha_i,\widetilde{\beta_i}}
    \in \mathcal{K}_B(\X_Z,\Y_Z).
\]
Therefore
\[
    M_f
    =
    \sum_{i=1}^n
       \theta_{\alpha_i,\widetilde{\beta_i}}
    \in \mathcal{K}_B(\X_Z,\Y_Z).
\]

For general $f\in C_0(Q)$, choose $f_j\in C_c(Q)$ with
$\|f_j-f\|_\infty\to0$.  Since the multiplier representation is
contractive,
\[
    \|M_{f_j}-M_f\|
    \leq
    \|f_j-f\|_\infty
    \longrightarrow0,
\]
and hence $M_f$ is compact.

Conversely, suppose that $M_f\in \mathcal{K}_B(\X_Z,\Y_Z)$, and let
$\varepsilon>0$.  By density of $\mathcal{C}(Z)$ in $\Y_Z$ and of
$\mathcal{C}(Z^{op})$ in $\X_Z$, together with the continuity of the
rank-one operators in both variables, there are
$\alpha_j\in \mathcal{C}(Z)$ and $\psi_j\in \mathcal{C}(Z^{op})$ such that
\[
    R=\sum_{j=1}^m\theta_{\alpha_j,\psi_j},
    \qquad
    \|M_f-R\|<\varepsilon.
\]
For each $j$, choose a compact set $K_j\subseteq Z$ outside which
$\alpha_j$ vanishes, and put
\[
    C=\bigcup_{j=1}^m q(K_j)\subseteq Q.
\]
Choose $u\in C_c(Q)$ such that
\[
    0\leq u\leq1,\qquad u|_C=1.
\]
Then $M_u\alpha_j=\alpha_j$ for every $j$, and the rank-one
composition identity gives
\[
    M_uR=R.
\]
Therefore
\[
\begin{aligned}
    \|M_f-M_{uf}\|
    &=
    \|M_f-M_uM_f\|                                   \\
    &\leq
    \|M_f-R\|
    +
    \|M_u(R-M_f)\|                                    \\
    &<
    2\varepsilon.
\end{aligned}
\]
Since the multiplier representation is isometric,
\[
    \|f-uf\|_\infty<2\varepsilon.
\]
But $uf\in C_c(Q)$, and $\varepsilon>0$ was arbitrary.  Hence
$f\in C_0(Q)$.
\end{proof}

\subsection{Orbit properness and compactness}

\begin{lem}\label{lem:proper-pullback-C0}
Let \(r\colon X\to Y\) be a continuous map between locally compact
Hausdorff spaces. Then the following are equivalent:
\begin{enumerate}
    \item \(r\) is proper;
    \item \(a\circ r\in C_0(X)\) for every \(a\in C_0(Y)\).
\end{enumerate}
\end{lem}

\begin{proof}
Suppose first that \(r\) is proper. For \(a\in C_0(Y)\) and
\(\varepsilon>0\),
\[
\begin{aligned}
    \{x\in X:|a(r(x))|\geq\varepsilon\}
    &=
    r^{-1}
    \bigl(
        \{y\in Y:|a(y)|\geq\varepsilon\}
    \bigr).
\end{aligned}
\]
The set inside the inverse image is compact, and therefore so is its
inverse image. Hence \(a\circ r\in C_0(X)\).

Conversely, suppose that \(a\circ r\in C_0(X)\) for every \(a\in C_0(Y)\).
Let \(K\subset Y\) be compact. Choose \(a\in C_c(Y)\) such that $a|_K\equiv 1$.
Then
\[
    L=
    \{x\in X:|a(r(x))|\geq\tfrac12\}
\]
is compact. Since \(K\) is closed in \(Y\), the set \(r^{-1}(K)\) is
closed in \(X\), and $r^{-1}(K)\subset L$.
Thus \(r^{-1}(K)\) is compact, proving that \(r\) is proper.
\end{proof}

For later use, note that $F^p_{\mathrm{red}}(G)$ has a contractive approximate
identity contained in $C_c(G^{(0)})$.  Indeed, direct the
compact subsets $K\subseteq G^{(0)}$ by inclusion and choose
$e_K\in C_c(G^{(0)})$ such that
\[
0\leq e_K\leq1,
\qquad
e_K|_K=1.
\]
Since a unit-space function acts in every regular representation by
multiplication,
\[
\|e_K\|_{F^p_{\mathrm{red}}(G)}
 =\|e_K\|_\infty\leq1.
\]
For $b\in\mathcal C(G)$, write
\[
b=\sum_{j=1}^n b_j,
\qquad b_j\in C_c(U_j),
\]
where each $U_j\subseteq G$ is open and Hausdorff.  Choose compact
sets $K_j\subseteq U_j$ outside which $b_j$ vanishes.  Once
\[
K\supseteq
\bigcup_{j=1}^n\bigl(r(K_j)\cup s(K_j)\bigr),
\]
we have
\[
e_K*b_j=b_j=b_j*e_K
\qquad(j=1,\ldots,n),
\]
and hence $e_K*b=b=b*e_K$.
Density of $\mathcal{C}(G)$ proves the assertion.

Combining Lemma~\ref{lem:proper-pullback-C0} with
Proposition~\ref{prop:compact-multiplier-characterization}, we obtain
\begin{equation}
    r_*\text{ is proper}
    \quad\Longleftrightarrow\quad
    M_{a\circ r_*}\in \mathcal K_B(\X_Z,\Y_Z)
    \text{ for every }a\in C_0(G^{(0)}).
  \label{eq:properness}
\end{equation}

\begin{thm} \label{thm:morita-cycle-implies-proper}
Let \(Z\) be a groupoid correspondence from \(G\) to \(H\).
Suppose that the action of \(\mathcal C(G)\) extends
to a Banach correspondence
\[
    \pi_Z\colon
    F_{\mathrm{red}}^p(G)
    \longrightarrow
    \mathcal L_{F_{\mathrm{red}}^p(H)}(\X_Z,\Y_Z).
\]
Then this Banach correspondence is a Morita cycle if and only if \(Z\) is proper.
\end{thm}

\begin{proof}
For $a\in C_c(G^{(0)})$ and $\phi\in \mathcal{C}(Z)$,
\[
(a\cdot\phi)(z)
 =a(r_Z(z))\phi(z)
 =(a\circ r_*)(q(z))\phi(z).
\]
The compatible action on $\mathcal{C}(Z^{op})$ is multiplication by the same
function.  Hence
\[
\pi_Z(a)=M_{a\circ r_*}.
\]
By density of $C_c(G^{(0)})$ in $C_0(G^{(0)})$, this equality holds for
every $a\in C_0(G^{(0)})$.

Suppose first that the Banach correspondence is a Morita cycle.  Then
$\pi_Z(a)$ is compact for every $a\in C_0(G^{(0)})$, so \eqref{eq:properness} implies
that $r_*:Q\to G^{(0)}$ is proper.  Thus $Z$ is proper.

Conversely, suppose that $Z$ is proper.  By \eqref{eq:properness},
\[
\pi_Z(a)\in \mathcal{K}_{F^p_{\mathrm{red}}(H)}(\X_Z,\Y_Z)
\qquad
(a\in C_0(G^{(0)})).
\]
Let $(e_\lambda)\subseteq C_c(G^{(0)})$ be the contractive approximate
identity described above.  If $b\in F^p_{\mathrm{red}}(G)$, then
\[
\pi_Z(e_\lambda)\pi_Z(b)
 =\pi_Z(e_\lambda b)
 \longrightarrow\pi_Z(b).
\]
Since each $\pi_Z(e_\lambda)$ is compact and
$\mathcal{K}_{F^p_{\mathrm{red}}(H)}(\X_Z,\Y_Z)$ is a closed two-sided ideal in
$\mathcal{L}_{F^p_{\mathrm{red}}(H)}(\X_Z,\Y_Z)$, it follows that $\pi_Z(b)$ is
compact.  Hence the Banach correspondence is a Morita cycle.
\end{proof}

\subsection{Reduced-norm extension and properness of the left action}

We now turn to a different properness condition.  The results of the
preceding subsection concern orbit properness of the correspondence.
Here we use properness of the left $G$-action as a sufficient condition
for the algebraic left action to extend to $F^p_{\mathrm{red}}(G)$.

Recall from Section~\ref{sect:BanCorresp} the space
  $E_Z^p=\bigoplus_{v\in H^{(0)}}^p\ell^p(Z_v)$ and the operator $\Pi_Z(a)$ from \eqref{eq:PiZ}.
The actions of \(G\) and \(H\) commute, so \(\Pi_Z(a)\) preserves each
space \(\ell^p(Z_v)\), \(v\in H^{(0)}\). The module actions on
\(\Y_Z\) and \(\X_Z\) are obtained by composing their concrete
representatives on the left and right, respectively, with
\(\Pi_Z(a)\). Consequently,
\[
    \|a\cdot\phi\|_{\Y_Z}
    \leq
    \|\Pi_Z(a)\|\|\phi\|_{\Y_Z},
    \qquad
    \|\psi\cdot a\|_{\X_Z}
    \leq
    \|\psi\|_{\X_Z}\|\Pi_Z(a)\|.
\]

\begin{prop}\label{prop:reduced-ext}
The algebraic action extends continuously to
\(F_{\mathrm{red}}^p(G)\) in each of the following cases:
\begin{enumerate}
    \item \(p=1\);
    \item \(p=\infty\);
    \item \(1<p<\infty\) and the left \(G\)-action on \(Z\) is proper.
\end{enumerate}
\end{prop}

\begin{proof}
Suppose first that \(p=\infty\). For every bounded \(\xi\),
\[
\begin{aligned}
    |(\Pi_Z(a)\xi)(z)|
    &\leq
    \sum_{r(\gamma)=r_Z(z)}
    |a(\gamma)|\,|\xi(\gamma^{-1}z)|                              \\
    &\leq
    \|a\|_{I,r}\|\xi\|_\infty.
\end{aligned}
\]
Hence
\[
    \|\Pi_Z(a)\|
    \leq
    \|a\|_{I,r} = \|a\|_{\infty,red},
\]
and the action extends continuously to
\(F_{\mathrm{red}}^\infty(G)\).

Suppose next that \(p=1\). Then
\[
\begin{aligned}
    \|\Pi_Z(a)\xi\|_1
    &\leq
    \sum_z
    \sum_{r(\gamma)=r_Z(z)}
    |a(\gamma)|\,|\xi(\gamma^{-1}z)|                              \\
    &=
    \sum_w
    |\xi(w)|
    \sum_{s(\gamma)=r_Z(w)}
    |a(\gamma)|                                                   \\
    &\leq
    \|a\|_{I,s}\|\xi\|_1.
\end{aligned}
\]
Therefore
\[
    \|\Pi_Z(a)\|
    \leq
    \|a\|_{I,s} = \|a\|_{1,red},
\]
and the action extends continuously to
\(F_{\mathrm{red}}^1(G)\).

Now suppose that \(1<p<\infty\) and that the left \(G\)-action on \(Z\) is
proper. Fix \(z\in Z\), put
\[
    u=r_Z(z),
\]
and define
\[
    K_z=\{k\in G:s(k)=r(k)=r_Z(z), k\cdot z=z\}.
\]
The set \(K_z\) is the fibre over \((z,z)\) of the proper action map
\[
    G*Z\longrightarrow Z\times Z,
    \qquad
    (\gamma,w)\longmapsto(\gamma w,w).
\]
Thus \(K_z\) is compact. Since \(G\) is étale, the source fibre
\(G_u=s^{-1}(u)\) is discrete. Hence \(K_z\) is finite.

The orbit map
\[
    G_u\longrightarrow Gz,
    \qquad
    \gamma\longmapsto\gamma z,
\]
has fibres equal to the right \(K_z\)-cosets. Indeed, $\gamma_1z=\gamma_2z$ if and only if $\gamma_2^{-1}\gamma_1\in K_z$.
Define
\[
    J_z\colon\ell^p(Gz)\longrightarrow\ell^p(G_u)
\]
by
\[
    (J_z\xi)(\gamma)
    =
    |K_z|^{-1/p}\xi(\gamma z).
\]
Since every point of \(Gz\) has exactly \(|K_z|\) preimages,
\[
    \|J_z\xi\|_p^p
    =
    |K_z|^{-1}
    \sum_{\gamma\in G_u}|\xi(\gamma z)|^p                         
    =
    \sum_{w\in Gz}|\xi(w)|^p.
\]
Thus \(J_z\) is an isometry.

Let \(\Pi_{Gz}(a)\) denote the restriction of
\(\Pi_Z(a)\) to \(\ell^p(Gz)\). For \(\gamma\in G_u\),
\[
\begin{aligned}
    (\lambda_u(a)J_z\xi)(\gamma)
    &=
    \sum_{\eta\in G^{r(\gamma)}}
    a(\eta)(J_z\xi)(\eta^{-1}\gamma)                              \\
    &=
    |K_z|^{-1/p}
    \sum_{\eta\in G^{r(\gamma)}}
    a(\eta)\xi(\eta^{-1}\gamma z)                                 \\
    &=
    (J_z\Pi_{Gz}(a)\xi)(\gamma).
\end{aligned}
\]
Consequently,
\[
    J_z\Pi_{Gz}(a)
    =
    \lambda_u(a)J_z,
\]
and therefore
\[
    \|\Pi_{Gz}(a)\|
    \leq
    \|\lambda_u(a)\|
    \leq
    \|a\|_{p,\mathrm{red}}.
\]

For each \(v\in H^{(0)}\), the \(G\)-space \(Z_v\) is the
disjoint union of its \(G\)-orbits, and
\[
    \ell^p(Z_v)
    =
    \bigoplus_{O\in G\backslash Z_v}^{p}\ell^p(O).
\]
The operator \(\Pi_Z(a)\) is block diagonal with respect to this
decomposition. Hence its norm is the supremum of the norms of the orbit
blocks, so
\[
    \|\Pi_Z(a)\|
    \leq
    \|a\|_{p,\mathrm{red}}.
\]
Together with the module estimates preceding this proposition, this proves that the action extends continuously to
\(F_{\mathrm{red}}^p(G)\).
\end{proof}

Combining Theorem~\ref{thm:morita-cycle-implies-proper} and Proposition~\ref{prop:reduced-ext}, we obtain the following:

\begin{thm} \label{thm:corrected-proper-correspondence}
Let \(Z\) be a groupoid correspondence from \(G\) to \(H\).
Assume, in addition, that either
\begin{enumerate}
    \item \(p\in\{1,\infty\}\); or
    \item \(1<p<\infty\) and the left \(G\)-action on \(Z\) is proper.
\end{enumerate}
Then the Banach $\bigl(
        F_{\mathrm{red}}^p(G),
        F_{\mathrm{red}}^p(H)
    \bigr)$-correspondence
associated with \(Z\) is a Morita cycle if and only if $Z$ is proper in the sense of Definition~\ref{def:groupoid-correspondence}.
\end{thm}

\section{Applications}\label{sect_app}

\subsection{Transformation groupoids}

Suppose that $G$ and $H$ are countable discrete groups acting freely and properly on a second-countable, locally compact Hausdorff space $Z$, on the left and right, respectively, with commuting actions.
Then the transformation groupoids $G\rtimes Z/H$ and $H\rtimes G\backslash Z$ are equivalent (cf. \cite[Example 2.33]{Will}). 
Thus, by Theorem~\ref{thm:main}, their reduced $L^p$-operator algebras $F^p_{red}(G\rtimes Z/H)$ and $F^p_{red}(H\rtimes G\backslash Z)$ are Morita equivalent.

By \cite[Proposition 3.1]{BK} (also see \cite[Proposition 6.4]{CGT}), for $p\in[1,\infty)$, these reduced $L^p$-operator algebras are isometrically isomorphic to the reduced $L^p$ crossed products associated with the group actions, i.e., 
\begin{align*}
F^p_{red}(G\rtimes Z/H) &\cong F^p_{red}(G, C_0(Z/H)), \\
F^p_{red}(H\rtimes G\backslash Z) &\cong F^p_{red}(H, C_0(G\backslash Z)).
\end{align*}
For the full algebras, the corresponding identifications are also
isometric.  Indeed, the disintegration theorem and norm estimates of
\cite{BKM1} identify the full $L^p$-operator crossed
product of a transformation group with the full $L^p$-operator algebra
of the associated transformation groupoid; see in particular
\cite[Example~4.8 and Theorems~5.5 and~5.13]{BKM1} and the discussion
in the introduction there.

Theorem~\ref{thm:main} and Corollary~\ref{cor:full-morita} therefore give:

\begin{thm} \label{thm:pGreen} ($L^p$ version of Green's symmetric imprimitivity theorem)
Suppose that $G$ and $H$ are countable discrete groups acting freely and properly on a second-countable, locally compact Hausdorff space $Z$, on the left and right, respectively, with commuting actions.
Then, for $p\in[1,\infty)$, the reduced $L^p$ crossed products $F^p_{red}(G,C_0(Z/H))$ and $F^p_{red}(H,C_0(G\backslash Z))$ are Morita equivalent, and the full $L^p$ crossed products $F^p(G,C_0(Z/H))$ and $F^p(H,C_0(G\backslash Z))$ are Morita equivalent.
\end{thm}

\begin{rem}
Second countability is assumed above just to ensure that the unit space of the transformation groupoid is paracompact. The result also holds when $Z$ is a compact Hausdorff space.
\end{rem}

\begin{cor} ($L^p$ version of Green's imprimitivity theorem)
Suppose that $H$ is a subgroup of a countable discrete group $G$.
Then, for every $p\in[1,\infty)$, \[
F^p(G,C_0(G/H))
\quad\text{and}\quad
F^p(H)
\]
are Morita equivalent, and
\[
F^p_{\mathrm{red}}(G,C_0(G/H))
\quad\text{and}\quad
F^p_{\mathrm{red}}(H)
\]
are Morita equivalent.
\end{cor}

\begin{proof}
This is the special case of Theorem~\ref{thm:pGreen} where $Z=G$ and $H$ is a subgroup of $G$ with the canonical actions.
\end{proof}

\begin{cor}
Let $G$ be a countable discrete group and let $p\in[1,\infty)$.
Then the canonical map
\[
F^p(G,C_0(G))
   \longrightarrow
F^p_{\mathrm{red}}(G,C_0(G))
\]
is an isometric isomorphism, and this algebra is Morita equivalent to
$\mathbb C$.
\end{cor}

\begin{proof}
The left translation action of $G$ on itself is amenable, so the full
and reduced $L^p$-operator crossed products coincide isometrically by
\cite[Corollary~4.14]{BK}.  The Morita equivalence with $\mathbb C$
then follows from the previous corollary with $H=\{e\}$.
\end{proof}

\begin{cor} 
Suppose that $G$ is a countable discrete group acting freely and properly
on a second-countable, locally compact Hausdorff space $X$. Then, for every
$p\in[1,\infty)$, the canonical map
\[
F^p(G,C_0(X))
   \longrightarrow
F^p_{\mathrm{red}}(G,C_0(X))
\]
is an isometric isomorphism, and these algebras are Morita
equivalent to $C_0(X/G)$. In particular, the $K$-theory of
$F^p(G,C_0(X))$, equivalently of
$F^p_{\mathrm{red}}(G,C_0(X))$, is independent of
$p\in[1,\infty)$.
\end{cor}

\begin{proof}
The properness of the action implies that the transformation groupoid
$G\ltimes X$ is proper. Since its orbit space $X/G$ is paracompact (indeed, $X/G$ is second countable, locally compact, and Hausdorff),
$G\ltimes X$ is amenable by \cite[Proposition~2.2.5]{ADR}, and hence the action $G\curvearrowright X$ is
amenable. Therefore the canonical regular
homomorphism
\[
F^p(G,C_0(X))
   \longrightarrow
F^p_{\mathrm{red}}(G,C_0(X))
\]
is an isometric isomorphism for every $p\in[1,\infty)$ by
\cite[Corollary~4.14]{BK}.
The Morita equivalence with $C_0(X/G)$ follows from
Theorem~\ref{thm:pGreen} applied with $Z=X$ and the trivial group acting on the right.
The assertion about $K$-theory follows.
\end{proof}

More generally, using the $L^P$-crossed product notation of \cite[Section~4]{BK}, Remark~\ref{rem:L^P} gives a Morita equivalence
\[
F^P_{\mathrm{red}}(G,C_0(X))\sim_M C_0(X/G)
\]
for every nonempty $P\subseteq[1,\infty)$.  Proposition
\ref{prop:commKT} then shows that the canonical maps associated with nonempty
$P_1\subseteq P_2\subseteq [1,\infty)$ induce isomorphisms on $K$-theory.

\subsection{Coarse groupoids}

To each uniformly locally finite coarse space $(X,\mathcal{E})$ (for instance, a uniformly locally finite metric space with its bounded coarse structure), Skandalis, Tu and Yu \cite{STY} introduced a locally compact Hausdorff \'{e}tale groupoid $G(X)$ called the coarse groupoid. Its unit space is a locally compact Hausdorff space, and is in fact the Stone-\v{C}ech compactification $\beta X$ when the coarse structure $\mathcal{E}$ is unital. Thus, $G(X)$ is not second countable in general.

Coarsely equivalent uniformly locally finite coarse spaces have equivalent coarse groupoids \cite[Corollary 3.6]{STY}. Therefore, if these coarse spaces are unital (so the unit spaces of their coarse groupoids are compact Hausdorff), the reduced $L^p$-operator algebras of the coarse groupoids are Morita equivalent by Theorem~\ref{thm:main}.
The same groupoid equivalence, together with Corollary~\ref{cor:full-morita}, also shows
that the full $L^p$-operator algebras are Morita equivalent for every $p\in[1,\infty]$.

\begin{rem}
The reduced $L^p$-operator algebra of the coarse groupoid $G(X)$ is isometrically isomorphic to the $\ell^p$ uniform Roe algebra of $X$ \cite[Proposition 5.1]{CD}. 
For metric spaces with bounded geometry, coarse equivalence of two such spaces is equivalent to their $\ell^p$ uniform Roe algebras being stably isometrically isomorphic for each $p\in[1,\infty)$ \cite[Theorem 3.4]{CL}.
\end{rem}

\subsection{Groupoids associated with inverse semigroups}

In \cite[Section 4.3]{Pat}, Paterson introduced the universal groupoid of an inverse semigroup.
These universal groupoids are locally compact, \'{e}tale, and ample, but are not necessarily Hausdorff, and their unit spaces are locally compact Hausdorff.
If an inverse semigroup is countable, then its universal groupoid is second countable so the unit space is paracompact.
The universal groupoid $G(S)$ is Hausdorff if and only if $S$ is a weak semilattice \cite[Theorem 5.17]{Ste1}, and this includes $E$-unitary inverse semigroups as examples. 

Strongly Morita equivalent inverse semigroups have equivalent universal groupoids \cite[Theorem 4.7]{Ste2}.
Thus for strongly Morita equivalent countable inverse semigroups, the reduced $L^p$-operator algebras of their universal groupoids are Morita equivalent by Theorem~\ref{thm:main}, and their full $L^p$-operator algebras are Morita equivalent by Corollary~\ref{cor:full-morita}.

In \cite[Section 13]{Exel}, Exel introduced the tight groupoid of a countable inverse semigroup with zero as a certain reduction of the universal groupoid.
A characterization of Hausdorffness of the tight groupoid is given in \cite[Theorem 3.16]{ExelPardo}.

If two countable inverse semigroups with zero are strongly Morita equivalent, then the equivalence between their universal groupoids restricts to an equivalence between their tight groupoids \cite[Theorem 4.7 and the paragraph following it]{Ste2}, so the reduced $L^p$-operator algebras of the tight groupoids are Morita equivalent; likewise for the full $L^p$-operator algebras.

The full and reduced $C^*$-algebras of an inverse semigroup are isomorphic to those of its universal groupoid \cite[Theorems 4.4.1 and 4.4.2]{Pat}.
Moreover, for $p\in[1,\infty]$, the respective full $L^p$-operator algebras are isometrically isomorphic \cite[Corollary 6.10]{BKM1}, but we are not aware of an analogous statement in the literature for the reduced $L^p$-operator algebras when $p\neq 2$.

\begin{cor}
Let $S$ and $T$ be countable strongly Morita equivalent inverse semigroups.
Then their full $L^p$-operator algebras $F^p(S)$ and $F^p(T)$ are Morita
equivalent for every $p\in[1,\infty]$.
\end{cor}

For the remainder of this subsection, assume that $S$ has no zero and
that $p\in[1,\infty)$.  We first compare the regular representations of
$S$ and its universal groupoid $G(S)$.  When $G(S)$ is Hausdorff, this
comparison yields an isometric isomorphism of the corresponding reduced
$L^p$-operator algebras.  

Recall that the left regular representation $\lambda$ of an inverse semigroup $S$ on $\ell^p(S)$ is given by
\[\lambda(s)\left(\sum_{t\in S}a_tt\right)=\sum_{\{ t\in S:tt^*\leq s^*s \}}a_tst=\sum_{\{ t\in S:s^*st=t \}}a_tst.\]

For $e\in E$, let $S_e=\{t\in S:t^*t=e\}$.
Note that $S=\bigsqcup_{e\in E}S_e$, so $\ell^p(S)=\bigoplus_{e\in E}\ell^p(S_e)$. 
For $e\in E$ and $s\in S$, define \[\lambda_e(s)\left(\sum_{t\in S_e}a_tt\right)=\sum_{\{ t\in S \colon  tt^*\leq s^*s,t^*t=e \}}a_tst.\]
Note that $t\mapsto st$ is injective on the domain $tt^*\leq s^*s$, so $\lambda_e(s)$ is a contraction on $\ell^p(S_e)$.
Extending $\lambda$ linearly to $\mathbb CS$, the operator norm closure of $\lambda(\mathbb{C}S)$ in $B(\ell^p(S))$ is the reduced $L^p$-operator algebra of $S$, denoted by $F^p_{red}(S)$.

\begin{lem} (cf. \cite[Proposition 4.4.5]{Pat} for the $p=2$ case)
$\lambda_e$ is a representation of $S$ on $\ell^p(S_e)$, and the left regular representation $\lambda$ is a direct sum (over $E$) of the representations $\lambda_e$.
\end{lem}

\begin{proof}
It is straightforward to check that if $t\in S_e$ and $tt^*\leq s^*s$, then $st\in S_e$.
Moreover, if $tt^*\leq s_2^*s_2$ and $s_2tt^*s_2^*=(s_2t)(s_2t)^*\leq s_1^*s_1$, then $tt^*=s_2^*s_2tt^*s_2^*s_2\leq s_2^*s_1^*s_1s_2=(s_1s_2)^*(s_1s_2)$.
Conversely, if $tt^*\leq s_2^*s_1^*s_1s_2$, then clearly $tt^*\leq s_2^*s_2$, and $s_2tt^*s_2^*\leq s_2s_2^*s_1^*s_1s_2s_2^*\leq s_1^*s_1$.
Hence $\lambda_e(s_1s_2)=\lambda_e(s_1)\lambda_e(s_2)$ as the point masses span a dense subspace of $\ell^p(S_e)$. 
Since $S=\bigsqcup_{e\in E}S_e$, we have $\ell^p(S)=\bigoplus_{e\in E}\ell^p(S_e)$, and $(\bigoplus_{e\in E}\lambda_e)(s)(t)=\lambda_{t^*t}(s)(t)=\lambda(s)(t)$ for $s,t\in S$.
\end{proof}

For the rest of this subsection, we shall write $G$ for $G(S)$.
Let $\overline{E}=\{\overline{e}:e\in E\}$, where $\overline{e}$ is the filter $\{f\in E:f\geq e\}$. 
Identifying $\overline{e}$ with its characteristic function, $\overline{E}$ is dense in $G^{(0)}$ by \cite[Proposition 4.3.1]{Pat} or \cite[Proposition 3.2(a)]{KS}.
Also, the source fiber $G_{\overline{e}}$ is $\{ [s,\overline{s^*s}]:s\in S_e \}$ by \cite[Proposition 3.2(b)]{KS}.
Moreover, by \cite[Proposition 3.2(c)]{KS}, the map from $S$ to $G$ sending $s$ to $[s,\overline{s^*s}]$ is injective with dense range.
Hence there is an invertible isometry $V_e:\ell^p(S_e)\to\ell^p(G_{\overline{e}})$ for $p\in[1,\infty]$. 

For $s\in S$, let
\[
\mathcal{O}_s=\{[s,\phi]\in G:\phi(s^*s)=1\},
\]
which is a compact open bisection of $G$.  Let $A_{\mathbb C}(G)$
denote the Steinberg algebra of $G$.  The map
\[
S\longrightarrow A_{\mathbb C}(G),\qquad
s\longmapsto\chi_{\mathcal O_s},
\]
extends to an algebra isomorphism
\[
\pi:\mathbb CS\longrightarrow A_{\mathbb C}(G)
\]
by \cite[Theorem~6.3 and Remark~4.2]{Ste1}.

In the next proposition, we shall write $\lambda^S$ and $\lambda^G$ for the semigroup and groupoid regular representations, respectively.
Let $\lambda^G_{\overline{e}}:\mathcal C(G)\to B(\ell^p(G_{\overline{e}}))$ be the regular representation associated with the unit $\overline{e}$.

\begin{prop}
For each $s\in S$ and $e\in E$, \[ \lambda^S_e(s) = V_e^{-1}\lambda^G_{\overline{e}}(\pi(s))V_e. \]
\end{prop}

\begin{proof}
Let $t\in S_e$, and assume $tt^*\leq s^*s$ (equivalently, $s^*st=t$). 
Then $st\in S_e$, and $\lambda^S_e(s)\delta_t=\delta_{st}$, so \[V_e\lambda^S_e(s)\delta_t=\delta_{[st,\overline{e}]}.\]
On the other hand, $V_e\delta_t=\delta_{[t,\overline{e}]}$, so for $\gamma\in G_{\overline{e}}$,
\begin{align*}
\lambda^G_{\overline{e}}(\pi(s))V_e\delta_t(\gamma) &= \lambda^G_{\overline{e}}(\chi_{\mathcal{O}_s})\delta_{[t,\overline{e}]}(\gamma) \\
&= \sum_{r(\eta)=r(\gamma)} \chi_{\mathcal{O}_s}(\eta)\delta_{[t,\overline{e}]}(\eta^{-1}\gamma) \\
&= \chi_{\mathcal{O}_s}(\gamma [t,\overline{e}]^{-1}).
\end{align*}
If $\gamma[t,\overline{e}]^{-1}=[s,\phi]\in\mathcal{O}_s$, then $\phi=t\cdot\overline{e}$, and $\gamma=[s,\phi][t,\overline{e}]=[st,\overline{e}]$.
Note that since $t\in S_e$, we have $t^*\cdot\overline{tet^*}=t^*\cdot\overline{tt^*}=\overline{t^*tt^*t}=\overline{e}$, so $[st,\overline{e}][t,\overline{e}]^{-1}=[st,\overline{e}][t^*,\overline{tt^*}]=[stt^*,\overline{tt^*}]$. 
Moreover, $\overline{tt^*}(tt^*)=1$ and $(stt^*)(tt^*)=s(tt^*)$, so $[stt^*,\overline{tt^*}]=[s,\overline{tt^*}]\in\mathcal{O}_s$ since $tt^*\leq s^*s$.
Hence, \[ \lambda^G_{\overline{e}}(\pi(s))V_e\delta_t=\delta_{[st,\overline{e}]}. \]

For the case where $tt^*\nleq s^*s$, we have $V_e\lambda^S_e(s)\delta_t=0$. 
On the other hand, for any $\gamma=[u,\overline{e}]\in G_{\overline{e}}$, we have $\gamma[t,\overline{e}]^{-1}=[ut^*,t\cdot\overline{e}]=[ut^*,\overline{tt^*}]\notin\mathcal{O}_s$, so $\lambda^G_{\overline{e}}(\pi(s))V_e\delta_t=0$.
\end{proof}

\begin{prop}\label{prop:inverse-semigroup-reduced}
Let $S$ be an inverse semigroup without zero whose universal groupoid
$G=G(S)$ is Hausdorff (equivalently, $S$ is a weak semilattice by \cite[Theorem 5.17]{Ste1}), and let $p\in[1,\infty)$.  Then the isomorphism
\[
\pi:\mathbb CS\longrightarrow A_{\mathbb C}(G)
\]
extends uniquely to an isometric isomorphism
\[
F^p_{\mathrm{red}}(S)\cong F^p_{\mathrm{red}}(G).
\]
\end{prop}

\begin{proof}
Since $G$ is Hausdorff and ample, $A_{\mathbb C}(G)$ is $I$-norm dense in $C_c(G)$ by \cite[Proposition~4.2]{CFST}, and hence is dense
in $F^p_{\mathrm{red}}(G)$.
For $f\in C_c(G)$, the map
\[
G^{(0)}\longrightarrow [0,\infty),
\qquad
u\longmapsto
\|\lambda^G_u(f)\|_{\mathcal B(\ell^p(G_u))}
\]
is lower semicontinuous.  Indeed, the argument of
\cite[Lemma~2.6]{Guo} applies with the $\ell^2$-norm replaced by the
$\ell^p$-norm.  Hence the supremum defining the reduced norm may be
taken over any dense subset of $G^{(0)}$.

Since
\[
\overline E=\{\overline e:e\in E\}
\]
is dense in $G^{(0)}$, for $\xi\in\mathbb CS$ we have
\[
\|\pi(\xi)\|_{p,\mathrm{red}}
=\sup_{e\in E}
  \|\lambda^G_{\overline e}(\pi(\xi))\|  
=\sup_{e\in E}
  \|\lambda^S_e(\xi)\| 
=\|\lambda^S(\xi)\|.
\]
Here the second equality follows from the preceding proposition by linearity.

Thus $\pi$ is isometric for the reduced norms.  Since
$\pi(\mathbb CS)=A_{\mathbb C}(G)$ is dense in
$F^p_{\mathrm{red}}(G)$, it extends to the asserted isometric
isomorphism.
\end{proof}

The next result follows from Proposition~\ref{prop:inverse-semigroup-reduced} and Theorem~\ref{thm:main}.

\begin{cor}
Let $S$ and $T$ be countable, strongly Morita equivalent inverse semigroups without zero, and suppose both are weak semilattices. Then their respective reduced $L^p$-operator algebras $F^p_{red}(S)$ and $F^p_{red}(T)$ are Morita equivalent for $p\in[1,\infty)$.
\end{cor}

\section*{Acknowledgments}

AD would like to thank David Blecher for fruitful conversations regarding Morita equivalence of operator spaces. 
YCC is supported by the Research Start-up Fund of Jilin University (grant number 419080524100).
ZW is supported by the Research Start-up Funding Program of Hangzhou Normal University (grant number 4235C50224204070).

\section*{Use of AI tools}
OpenAI's ChatGPT (GPT-5.6 Sol) was used as an editorial and checking aid during the preparation of this manuscript, including English-language editing, proofreading, consistency checks, and assistance in auditing the exposition and simplifying the presentation of some arguments. The mathematical ideas, constructions, and results were developed by the authors. In particular, the central mathematical ideas and the strategy of the proofs were conceived by the authors. All AI-assisted suggestions were independently checked by the authors, who take full responsibility for the mathematical content of the article.

\bibliographystyle{plain}
\bibliography{mybib}

\begin{thebibliography}{10}

\bibitem{ADR}
C.~Anantharaman-Delaroche and J.~Renault.
\newblock {\em Amenable groupoids}, volume~36 of {\em Monographies de
  L'Enseignement Math\'{e}matique}.
\newblock L'Enseignement Math\'{e}matique, Geneva, 2000.

\bibitem{AKM}
C.~Antunes, J.~Ko, and R.~Meyer.
\newblock The bicategory of groupoid correspondences.
\newblock {\em New York J. Math.}, 28:1329--1364, 2022.

\bibitem{AOP}
A.~Austad, E.~Ortega, and M.~Palmstr{\o}m.
\newblock Polynomial growth and property {RD$_p$} for \'{e}tale groupoids with
  applications to {$K$}-theory.
\newblock {\em J. Noncommut. Geom.}, 19(2):601--645, 2025.

\bibitem{BK}
K.~Bardadyn and B.~Kwa\'{s}niewski.
\newblock Topologically free actions and ideals in twisted {B}anach algebra
  crossed products.
\newblock {\em Proc. Roy. Soc. Edinburgh Sect. A}, 156(1):157--187, 2026.

\bibitem{BKM1}
K.~Bardadyn, B.~Kwa\'{s}niewski, and A.~McKee.
\newblock Banach algebras associated to twisted \'{e}tale groupoids: inverse
  semigroup disintegration and representations on {$L^p$}-spaces.
\newblock {\em J. Funct. Anal.}, 289(12):111163, 2025.

\bibitem{BKM2}
K.~Bardadyn, B.~Kwa\'{s}niewski, and A.~McKee.
\newblock Banach algebras associated to twisted \'{e}tale groupoids: simplicity
  and pure infiniteness.
\newblock {\em Trans. Amer. Math. Soc.}, 2026.
\newblock DOI: 10.1090/tran/9724.

\bibitem{CGT}
Y.~Choi, E.~Gardella, and H.~Thiel.
\newblock Rigidity results for {$L^p$}-operator algebras and applications.
\newblock {\em Adv. Math.}, 452:109747, 2024.

\bibitem{Chung4}
Y.C. Chung.
\newblock Morita equivalence of two {$\ell^p$} {R}oe-type algebras.
\newblock {\em J. Noncommut. Geom.}, 19(3):1069--1088, 2025.

\bibitem{CD}
Y.C. Chung and X.~Du.
\newblock Ideal structure of {$\ell^p$} uniform {R}oe algebras.
\newblock {\em Preprint}, 2026.
\newblock arXiv:2606.11586.

\bibitem{CD2}
Y.C. Chung and X.~Du.
\newblock Morita induction and preservation of geometric and dynamical ideals.
\newblock {\em Preprint}, 2026.
\newblock arXiv:2608.30380.

\bibitem{CL}
Y.C. Chung and K.~Li.
\newblock Rigidity of {$\ell^p$} {R}oe-type algebras.
\newblock {\em Bull. Lond. Math. Soc.}, 50(6):1056--1070, 2018.

\bibitem{CFST}
L.P. Clark, C.~Farthing, A.~Sims, and M.~Tomforde.
\newblock A groupoid generalisation of {L}eavitt path algebras.
\newblock {\em Semigroup Forum}, 89(3):501--517, 2014.

\bibitem{Delfin24}
A.~Delf\'{\i}n.
\newblock Representations of {$C^*$}-correspondences on pairs of {H}ilbert
  spaces.
\newblock {\em J. Operator Theory}, 92(1):167--188, 2024.

\bibitem{Delfin26}
A.~Delf\'{\i}n.
\newblock {$L^p$}-modules and {$L^p$}-correspondences.
\newblock {\em Banach J. Math. Anal.}, 20(1):Paper No. 13, 2026.

\bibitem{EngelkingGeneralTopology}
R.~Engelking.
\newblock {\em General Topology}, volume~6 of {\em Sigma Series in Pure
  Mathematics}.
\newblock Heldermann Verlag, Berlin, revised and completed edition, 1989.

\bibitem{Exel93}
R.~Exel.
\newblock A {F}redholm operator approach to {M}orita equivalence.
\newblock {\em {$K$}-Theory}, 7(3):285--308, 1993.

\bibitem{Exel}
R.~Exel.
\newblock Inverse semigroups and combinatorial {$C^*$}-algebras.
\newblock {\em Bull. Braz. Math. Soc. (N.S.)}, 39(2):191--313, 2008.

\bibitem{ExelPardo}
R.~Exel and E.~Pardo.
\newblock The tight groupoid of an inverse semigroup.
\newblock {\em Semigroup Forum}, 92(1):274--303, 2016.

\bibitem{Fremlin}
D.~H. Fremlin.
\newblock {\em Measure Theory. Vol. 2: Broad Foundations}.
\newblock Torres Fremlin, Colchester, 2003.
\newblock Corrected second printing of the 2001 original.

\bibitem{GL}
E.~Gardella and M.~Lupini.
\newblock Representations of \'{e}tale groupoids on {$L^p$}-spaces.
\newblock {\em Adv. Math.}, 318:233--278, 2017.

\bibitem{Guo}
L.~Guo.
\newblock A groupoid approach to the equivariant coarse {B}aum--{C}onnes
  conjecture.
\newblock 2026.
\newblock arXiv:2604.25595.

\bibitem{HO}
E.V. Hetland and E.~Ortega.
\newblock Rigidity of twisted groupoid {$L^p$}-operator algebras.
\newblock {\em J. Funct. Anal.}, 285(6):110037, 2023.

\bibitem{KS}
M.~Khoshkam and G.~Skandalis.
\newblock Regular representations of groupoid {$C^*$}-algebras and applications
  to inverse semigroups.
\newblock {\em J. Reine Angew. Math.}, 546:47--72, 2002.

\bibitem{MRW}
P.S. Muhly, J.~Renault, and D.P. Williams.
\newblock Equivalence and isomorphism for groupoid {$C^*$}-algebras.
\newblock {\em J. Operator Theory}, 17(1):3--22, 1987.

\bibitem{MuhlyWilliamsImprimitivity}
P.S. Muhly and D.P. Williams.
\newblock Groupoid cohomology and the {Dixmier--Douady} class.
\newblock {\em Proc. London Math. Soc. (3)}, 71(1):109--134, 1995.

\bibitem{MW}
P.S. Muhly and D.P. Williams.
\newblock {\em Renault's equivalence theorem for groupoid crossed products},
  volume~3 of {\em New York Journal of Mathematics. NYJM Monographs}.
\newblock State University of New York, University at Albany, Albany, NY, 2008.

\bibitem{Par09I}
W.~Paravicini.
\newblock Induction for {B}anach algebras, groupoids and {$KK^{ban}$}.
\newblock {\em J. {$K$}-Theory}, 4(3):405--468, 2009.

\bibitem{Par09}
W.~Paravicini.
\newblock Morita equivalences and {$KK$}-theory for {B}anach algebras.
\newblock {\em J. Inst. Math. Jussieu}, 8(3):565--593, 2009.

\bibitem{Par15}
W.~Paravicini.
\newblock {$kk$}-theory for {B}anach algebras {I}: {T}he non-equivariant case.
\newblock {\em J. Funct. Anal.}, 268(10):3108--3161, 2015.

\bibitem{Pat}
A.L.T. Paterson.
\newblock {\em Groupoids, inverse semigroups, and their operator algebras},
  volume 170 of {\em Progress in Mathematics}.
\newblock Birkh\"{a}user Boston, Inc., Boston, MA, 1999.

\bibitem{RaWi98}
I.~Raeburn and D.P. Williams.
\newblock {\em Morita equivalence and continuous-trace {$C^*$}-algebras},
  volume~60 of {\em Mathematical Surveys and Monographs}.
\newblock American Mathematical Society, Providence, RI, 1998.

\bibitem{Ren1}
J.~Renault.
\newblock {\em A groupoid approach to {$C^*$}-algebras}, volume 793 of {\em
  Lecture Notes in Mathematics}.
\newblock Springer-Verlag, Berlin-New York, 1980.

\bibitem{Ren2}
J.~Renault.
\newblock {$C^*$}-algebras of groupoids and foliations.
\newblock In {\em Operator algebras and applications, {P}art 1 ({K}ingston,
  {O}nt., 1980)}, volume~38 of {\em Proc. Sympos. Pure Math.}, pages 339--350.
  Amer. Math. Soc., Providence, RI, 1982.

\bibitem{Ren3}
J.~Renault.
\newblock Transverse properties of dynamical systems.
\newblock In {\em Representation Theory, Dynamical Systems, and Asymptotic
  Combinatorics}, volume 217 of {\em Amer. Math. Soc. Transl. Ser. 2}, pages
  185--199. Amer. Math. Soc., Providence, RI, 2006.

\bibitem{Rie74}
M.A. Rieffel.
\newblock Induced representations of {$C^*$}-algebras.
\newblock {\em Adv. Math.}, 13:176--257, 1974.

\bibitem{Rie74a}
M.A. Rieffel.
\newblock Morita equivalence for {$C^*$}-algebras and {$W^*$}-algebras.
\newblock {\em J. Pure Appl. Algebra}, 5:51--96, 1974.

\bibitem{CRM}
A.~Sims, G.~Szab\'{o}, and D.P. Williams.
\newblock {\em Operator Algebras and Dynamics: Groupoids, Crossed Products, and
  Rokhlin Dimension}.
\newblock Advanced Courses in Mathematics -- CRM Barcelona. Birkh\"{a}user,
  Cham, 2020.

\bibitem{SW}
A.~Sims and D.P. Williams.
\newblock Renault's equivalence theorem for reduced groupoid {$C^*$}-algebras.
\newblock {\em J. Operator Theory}, 68(1):223--239, 2012.

\bibitem{STY}
G.~Skandalis, J.-L. Tu, and G.~Yu.
\newblock The coarse {B}aum--{C}onnes conjecture and groupoids.
\newblock {\em Topology}, 41(4):807--834, 2002.

\bibitem{Ste1}
B.~Steinberg.
\newblock A groupoid approach to discrete inverse semigroup algebras.
\newblock {\em Adv. Math.}, 223(2):689--727, 2010.

\bibitem{Ste2}
B.~Steinberg.
\newblock Strong {M}orita equivalence of inverse semigroups.
\newblock {\em Houston J. Math.}, 37(3):895--927, 2011.

\bibitem{Tu}
J.-L. Tu.
\newblock Non-{H}ausdorff groupoids, proper actions and {$K$}-theory.
\newblock {\em Doc. Math.}, 9:565--597, 2004.

\bibitem{Will}
D.P. Williams.
\newblock {\em A Tool Kit for Groupoid {$C^*$}-Algebras}, volume 241 of {\em
  Mathematical Surveys and Monographs}.
\newblock American Mathematical Society, Providence, RI, 2019.

\end{thebibliography}
\end{document}